\documentclass[Afour,sageh,times]{sagej}

\usepackage{amsmath,amsfonts,amsthm}
\usepackage{amssymb}
\usepackage{prettyref}
\usepackage{mathrsfs}
\usepackage{graphicx}
\usepackage{wrapfig}
\usepackage{subfig}
\usepackage{framed}
\usepackage{multirow}
\usepackage{booktabs}
\usepackage{algorithm}
\usepackage[table]{xcolor}
\usepackage[export]{adjustbox}
\usepackage[noend]{algpseudocode}
\usepackage{enumitem}
\usepackage{comment}
\usepackage[colorlinks,allcolors=gray,hypertexnames=true]{hyperref} 

\usepackage[nomessages]{fp}
\usepackage[overload]{empheq}

\newtheorem{theorem}{Theorem}
\newtheorem{lemma}[theorem]{\TE{Lemma}}
\newtheorem{remark}[theorem]{\TE{Remark}}
\newtheorem{assume}[theorem]{\TE{Assumption}}

\newtheorem{corollary}[theorem]{\TE{Corollary}}

\algnewcommand{\LineComment}[1]{\State \(\triangleright\) #1}
\algdef{SE}[DOWHILE]{Do}{doWhile}{\algorithmicdo}[1]
{\algorithmicwhile\ #1}
\newcommand*{\colorboxed}{}
\def\colorboxed#1#{%
  \colorboxedAux{#1}%
}
\newcommand*{\colorboxedAux}[3]{%
  \begingroup
    \colorlet{cb@saved}{.}%
    \color#1{#2}%
    \boxed{%
      \color{cb@saved}%
      #3%
    }%
  \endgroup
}

\newrefformat{fig}{Figure~\ref{#1}}
\newrefformat{par}{Section~\ref{#1}}
\newrefformat{appen}{Appendix~\ref{#1}}
\newrefformat{sec}{Section~\ref{#1}}
\newrefformat{sub}{Section~\ref{#1}}
\newrefformat{table}{Table~\ref{#1}}
\newrefformat{ass}{Assumption~\ref{#1}}
\newrefformat{alg}{Algorithm~\ref{#1}}
\newrefformat{def}{Definition~\ref{#1}}
\newrefformat{rem}{Remark~\ref{#1}}
\newrefformat{thm}{Theorem~\ref{#1}}
\newrefformat{lem}{Lemma~\ref{#1}}
\newrefformat{cor}{Corollary~\ref{#1}}
\newrefformat{step}{Step~\ref{#1}}
\newrefformat{ln}{Line~\ref{#1}}
\newrefformat{eq}{Equation~\ref{#1}}
\newrefformat{pb}{Problem~\ref{#1}}
\newrefformat{it}{Item~\ref{#1}}
\newrefformat{te}{Term~\ref{#1}}
\def\Eqref Eq:#1:{\eqref{eq:#1}}
\newrefformat{Eq}{Equation~\Eqref#1:}

\newcommand{\TE}[1]{\textbf{#1}}

\newcommand{\FPP}[2]{\frac{\partial{#1}}{\partial{#2}}}
\newcommand{\FPPR}[2]{{\partial{#1}}/{\partial{#2}}}

\newcommand{\TWO}[2]{\left(\setlength{\arraycolsep}{1pt}\begin{array}{cc}{#1}, & {#2}\end{array}\right)}
\newcommand{\TWOC}[2]{\left(\setlength{\arraycolsep}{1pt}\begin{array}{c}#1 \\ #2\end{array}\right)}

\newcommand{\THREE}[3]{\left(\setlength{\arraycolsep}{1pt}\begin{array}{ccc}{#1}, & {#2}, & {#3}\end{array}\right)}

\newcommand{\FOUR}[4]{\left(\setlength{\arraycolsep}{1pt}\begin{array}{cccc}{#1}, & {#2}, & {#3}, & {#4}\end{array}\right)}

\newcommand{\SIX}[6]{\left(\setlength{\arraycolsep}{1pt}\begin{array}{cccccc}{#1}, & {#2}, & {#3}, & {#4}, & {#5}, & {#6}\end{array}\right)}

\newcommand{\SIXR}[6]{\left(\setlength{\arraycolsep}{1pt}\begin{array}{cccccc}{#1}^T, & {#2}^T, & {#3}^T, & {#4}^T, & {#5}^T, & {#6}^T\end{array}\right)^T}

\newcommand{\dist}{\text{dist}}

\newcommand{\argmin}[1]{\underset{#1}{\text{argmin}}}

\newcommand{\ST}{\text{s.t.}}

\newcommand{\TWORCell}[2]{\begin{tabular}{@{}l@{}}#1 \\ #2\end{tabular}}

\newcommand{\rad}{\rho}
\newcommand{\dom}{\text{dom}}

\newcommand{\logdet}{\text{logdet}}
\newcommand{\CH}{\text{CH}}
\newcommand{\CLIP}{\text{CLIP}}

\newcommand\numeq[1]{\stackrel{\scriptscriptstyle(\mkern-1.5mu#1\mkern-1.5mu)}{=}}
\newcommand\numleq[1]{\stackrel{\scriptscriptstyle(\mkern-1.5mu#1\mkern-1.5mu)}{\leq}}
\newcommand\numgeq[1]{\stackrel{\scriptscriptstyle(\mkern-1.5mu#1\mkern-1.5mu)}{\geq}}

\newcommand{\proofread}[1]{}
\newif\ifArxiv
\usepackage{xcolor}
\definecolor{Blue}{rgb}{0.2, 0.2, 0.8}
\definecolor{Black}{rgb}{0.0, 0.0, 0.0}

\newcommand\BibTeX{{\rmfamily B\kern-.05em \textsc{i\kern-.025em b}\kern-.08em
T\kern-.1667em\lower.7ex\hbox{E}\kern-.125emX}}

\def\volumeyear{2016}

\begin{document}

\runninghead{Zherong and Kui}

\newif\iflong
\longtrue

\newif\ifreview
\reviewtrue

\title{BC-ADMM: A Parallel Decoupled Non-convex Constrained Optimizer for Robot Applications}

\author{Zherong Pan and Kui Wu\affilnum{1}}

\affiliation{\affilnum{1}LIGHTSPEED}

\email{\{zherong.pan.usa@gmail.com, walker.kui.wu@gmail.com\}}

\begin{abstract}
Non-convex constrained optimizations are ubiquitous in robotic applications such as multi-agent navigation, UAV trajectory optimization, and soft robot simulation. As a common feature in these problems, the associated non-convex constraints, including collision constraints, inversion-free constraints, and strain limits, are also non-smooth with ill-defined gradients. It is well-known that such constraints are notoriously difficult to handle, for which off-the-shelf optimizers can fail catastrophically. Instead, prior works tend to design problem-specific optimizers that trade performance for robustness.
To efficiently solve this problem class in a unified manner, we propose a variant of Alternating Direction Method of Multiplier (ADMM), called BC-ADMM. Over the past decade, ADMM has achieved great success in efficiently solving many large-scale (constrained) optimization problems by decoupling them into sub-problems that can be solved in parallel. However, prior ADMM algorithms do not have convergence guarantee when dealing with a large number of non-convex constraints with loopy constraint graphs. Instead, our BC-ADMM relaxes each non-convex constraint into a bi-convex function, further breaking the constraint into two sub-problems. We show that such relaxation leads to both theoretical convergence speed guarantees and practical convergence guarantees in the asymptotic sense. Through numerical experiments in a row of four robotic applications, we show that BC-ADMM has faster convergence than conventional gradient descent and Newton's method in terms of wall clock time.
\end{abstract}


\maketitle
\section{Introduction}
Non-convex constrained optimization problems are pervasive in robotics, particularly in areas such as geometric processing, motion planning, soft-robot modeling, and multi-agent navigation. These challenges commonly arise in applications like deformable body simulation~\cite{lipson2014challenges,kovalsky2015large}, collision-free trajectory optimization~\cite{zhang2020optimization}, and multi-robot coordination~\cite{guy2009clearpath}, among others. Solving such problems is notoriously difficult for two main reasons. First, they often involve large-scale formulations with a high number of decision variables and constraints. For instance, the number of variables can grow superlinearly with mesh resolution. In multi-agent navigation, the number of collision constraints increases quadratically with the number of agents-reaching on the order of $10^4$ constraints for $10^2$ agents in our evaluations. Similarly, in soft robot modeling using the finite element method, the number of elements can reach $10^3$, while collision and strain-limiting constraints may scale up to $10^6$. Second, many of these non-convex constraints are also non-smooth, often lacking well-defined gradients. A common example is the collision avoidance constraint. To illustrate this challenge, consider the simple scenario shown in~\prettyref{fig:non-smooth}, where a circular robot attempts to reach a target position while avoiding a triangle-shaped obstacle. This problem can be formulated as the following optimization:
\begin{equation}
\begin{aligned}
\label{eq:non-smooth}
\argmin{x_1}\;&\|x_1-x_1^\star\|^2\\
\ST\quad&\text{inf}_{x'\in\CH(x_2,x_3,x_4)}\|x_1-x'\|\geq r,
\end{aligned}
\end{equation}
which involves a single non-convex collision constraint. It is well-known that this constraint function is non-smooth due to the point-to-set distance function $\text{inf}_{x'\in\CH(x_2,x_3,x_4)}\|x_1-x'\|$. Indeed, as illustrated in~\prettyref{fig:non-smooth-cases}, such a distance function has a vanishing gradient when $x_1\in\CH^\circ(x_2,x_3,x_4)$ where $\CH^\circ$ indicates the interior of the convex hull, but a unit-norm directional derivative exists when $x_1\notin\CH^\circ(x_2,x_3,x_4)$.

\begin{figure}[ht]
\centering
\includegraphics[width=.84\linewidth]{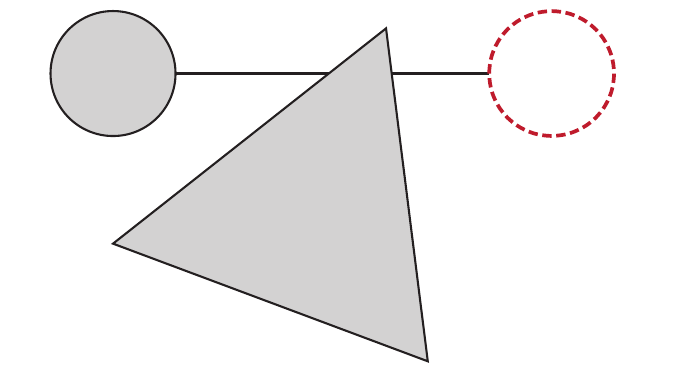}
\put(-170,83){$x_1$}
\put(-40,83){$x_1^\star$}
\put(-88,10){$x_2$}
\put(-160,37){$x_3$}
\put(-97,88){$x_4$}
\caption{\label{fig:non-smooth} We optimize the position of a circular robot with a radius $r$ under collision constraint, which is located at $x_1\in\mathbb{R}^2$. The robot needs to reach the goal position $x_1^\star$, while avoiding the triangular obstacle spanned by its three vertices $x_2,x_3,x_4$. We can formulate this problem as an optimization with the collision constraint $\text{inf}_{x'\in\CH(x_2,x_3,x_4)}\|x_1-x'\|\geq r$. Here $\CH$ denotes the closed convex set spanned by a set of vertices.}
\end{figure}

Solving constrained optimizations with non-smooth, non-convex constraint functions is notoriously difficult for general-purpose optimization algorithms, such as sequential quadratic programming and primal-dual interior point method~\cite{nocedal1999numerical}, as implemented in mature software packages SNOPT~\cite{gill2005snopt} and IPOPT~\cite{wachter2002interior}. All these algorithms are infeasible point optimizers, meaning that they do not assume the intermediary solutions during an optimization are strictly feasible. Instead, they rely on the gradient of constraints to be non-vanishing, which could be used to recover feasibility from infeasible initial guesses. Unfortunately, such non-vanishing-gradient assumptions on the constraint functions are frequently violated in practical problems, e.g. when $x_1\in\CH^\circ(x_2,x_3,x_4)$ in~\prettyref{eq:non-smooth}. The violation of these assumptions can lead to catastrophic failure of infeasible point optimizers by getting stuck at an infeasible solution, as analyzed in~\cite{schulman2014motion,byravan2014space}.

\begin{figure}[ht]
\centering
\includegraphics[width=.84\linewidth]{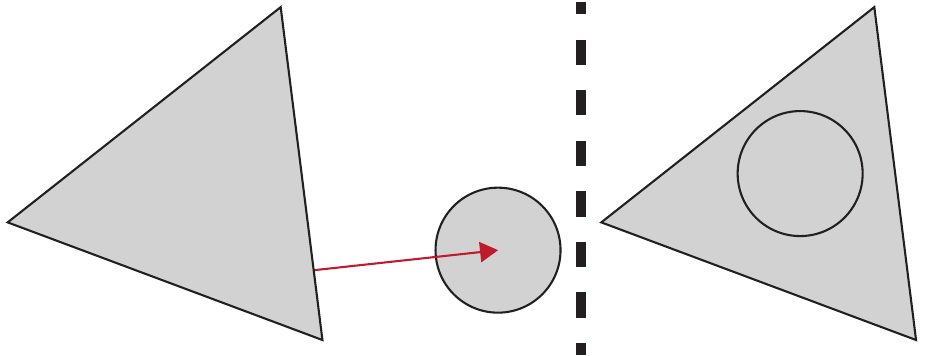}
\put(-98,14){$x_1$}
\put(-33,35){$x_1$}
\put(-190,-10){$x_1\notin\CH^\circ(x_2,x_3,x_4)$}
\put(-70,-10){$x_1\in\CH^\circ(x_2,x_3,x_4)$}
\caption{\label{fig:non-smooth-cases} We illustrate two cases of the non-smooth collision constraint function. Left: When $x_1\notin\CH^\circ(x_2,x_3,x_4)$, the function is locally differentiable with a gradient (red arrow) satisfying $\left\|\FPPR{\left[\text{inf}_{x'\in\CH(x_2,x_3,x_4)}\|x_1-x'\|-r\right]}{x_1}\right\|=1$. Right: When $x_1\in\CH^\circ(x_2,x_3,x_4)$, then gradient vanishes.}
\end{figure}
To address the difficulty of infeasible point optimizers, researchers have noticed that many non-smooth constraints are indeed sufficiently smooth in the feasible domain. For example, the point-to-set distance function in~\prettyref{eq:non-smooth} is differentiable with unit-norm gradient when restricted to its feasible domain. Therefore, several works have proposed to use a feasible interior point method by using a barrier function to restrict the solution in its feasible domain. In the example of~\prettyref{eq:non-smooth}, the constrained optimization is transformed into the following unconstrained optimization, which can then be solved using gradient descend or Newton's method:
\begin{equation}
\begin{aligned}
\label{eq:non-smooth-feasible}
\argmin{x}\;\|x-x^\star\|^2-
\log_\epsilon\left[\text{inf}_{y\in\CH(y_1,y_2,y_3)}\|x-y\|-r\right]
\end{aligned}
\end{equation}
with $\log_\epsilon$ being a locally supported variant of log-barrier function depending on a hyperparameter $\epsilon$. Starting from a feasible initial guess, the feasible interior point method maintains constraint feasibility and thus avoids the issue of non-smoothness. It can be shown that the function $\log_\epsilon$ can be designed such that the local solution of~\prettyref{eq:non-smooth-feasible} is also a local solution of~\prettyref{eq:non-smooth} as $\epsilon$ tends to zero. Variants of such feasible interior point methods have been proposed to solve problems in deformable body simulation~\cite{li2020incremental,10478182}, collision-free trajectory optimization~\cite{ni2022robust,10505800}, and multi-robot navigation~\cite{van2008reciprocal,karamouzas17}. These works use problem-specific barrier functions to achieve the ideal balance between computational efficacy, robustness, and solution quality. However, despite their robustness, the computational efficacy has been compromised as compared with the infeasible point methods~\cite{gill2005snopt,wachter2002interior}. This is because the optimizer needs to take very small step sizes so that none of the constraints are violated. Typically, the global step size is upper bounded by the single, most stringent constraint.

In this work, we propose an unified structure-aware optimizer to improve the computational efficacy of feasible interior point method. Unlike black-box numerical algorithms, structure-aware optimizers~\cite{chang2014multi,jiang2019structured,khatana2022dc} exploit the properties in objective functions and accelerate computation. Recently, Alternating Direction Method of Multipliers (ADMM), a prominent structure-aware optimization algorithm, has been extensively applied to general physics simulation~\cite{overby2017admmpd}, geometric processing~\cite{ouyang2020anderson}, motion planning~\cite{ni2022robust} and multi-agent navigation~\cite{Saravanos-RSS-21} problems. These methods utilize the fact that objective functions and constraints can be decomposed into terms related to small subsets of variables, leading to a set of smaller sub-problems that can be solved efficiently and in parallel. The solutions of these sub-problems are then summarized into a unified set of variables, forming a consensus~\cite{boyd2011distributed}. Unlike the infeasible interior point method, where the entire solution is updated using a global step size, ADMM solves each sub-problem independently, allowing them to take much larger step sizes. Further, by allowing these sub-problems to be solved in parallel, ADMM enjoys a low iterative cost, allowing it to quickly converge to an approximate solution for large problem instances.

\begin{figure}[ht]
\centering
\includegraphics[width=.84\linewidth]{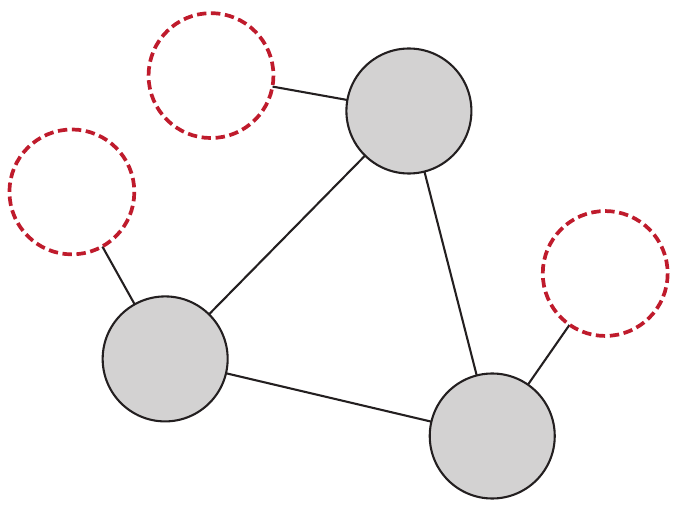}
\put(-60,19){$x_1$}
\put(-155,42){$x_2$}
\put(-82,115){$x_3$}
\put(-27,66){$x_1^\star$}
\put(-183,90){$x_2^\star$}
\put(-142,125){$x_3^\star$}
\put(-131,20){$\tilde{P}_r(x_1,x_2)$}
\put(-157,90){$\tilde{P}_r(x_2,x_3)$}
\put(-114,60){$\tilde{P}_r(x_3,x_1)$}
\caption{\label{fig:fail}We optimize the position of three circular robots with a uniform radius $r$ under collision constraints, where their position variables are denoted as $x_{1,2,3}\in\mathbb{R}^2$ concatenated into a decision variable $x\in\mathbb{R}^6$. The objective function $f:\mathbb{R}^6\to\mathbb{R}^+$ is designed for the three robots to reach their distinctive goals $x_i^\star$, which is defined as $\sum_{i=1}^3\|x_i-x_i^\star\|^2$. Each collision constraint is modeled as an extended-real log-barrier function $\tilde{P}\in\mathbb{R}^4\to\mathbb{R}^+\cup\{\infty\}$ defined as $\tilde{P}_r(x_i,x_j)=-\log_\epsilon(\|x_i-x_j\|-2r)$, with $\log_\epsilon$ being a locally supported variant of log-barreir function. Altogether, we need three barrier functions $\tilde{P}_r(x_1,x_2)$, $\tilde{P}_r(x_2,x_3)$, and $\tilde{P}_r(x_3,x_1)$, of which the corresponding constraint graph has a loop. Unfortunately, existing ADMM algorithm does not have convergence guarantee for this problem.}
\end{figure}
However, we observe that when the sub-problems imply non-convex constraints and the connectivity between constraints has loops, then existing structured optimization algorithms are not guaranteed to converge. Unfortunately, such problems are ubiquitous in all the aforementioned paradigms. An illustrative toy example is provided in~\prettyref{fig:fail} where we optimize collision-free poses for 3 robots. Variants of this problem have been considered in prior works~\cite{van2008reciprocal,karamouzas17}.  In this case, collision constraints exist between each pair of the 3 robots. If we consider each robot as the node of a graph, and each collision constraint is an edge connecting the pair of robots potentially in collision, then the constraint graph has a loop. Surprisingly, existing convergence analyses of the ADMM algorithm can fail for loopy constraint graphs. This limitation prevents ADMM from being adopted in a wide range of robotic applications.

\begin{figure}[ht]
\centering
\includegraphics[width=.84\linewidth]{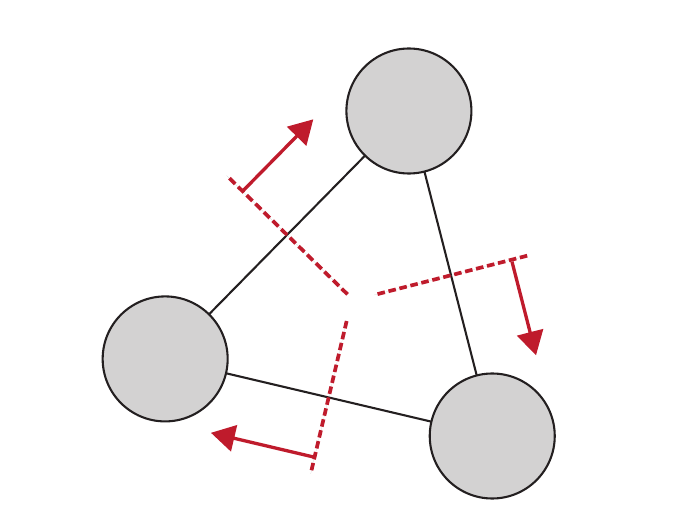}
\put(-60,19){$x_1$}
\put(-155,42){$x_2$}
\put(-82,115){$x_3$}
\put(-135,110){$n_{23}$}
\put(-148,90){$d_{23}$}
\put(-43,60){$n_{13}$}
\put(-50,82){$d_{13}$}
\put(-140,10){$n_{12}$}
\put(-105,10){$d_{12}$}
\caption{\label{fig:fail-reformulation} We can reformulate the log-barrier function $\tilde{P}_r(x_i,x_j)$ in~\prettyref{fig:fail} into an augmented function $P_r(x_i,x_j,n_{ij},d_{ij})$. This is done by introducing separating planes between the pair of robots $x_i$ and $x_j$ and require that the two robots lie on different sides of the separating plane, with the normals and offsets of these separating planes (red) denoted as $n_{ij}$ and $d_{ij}$, respectively. We cancatenate all these variables into $z=\SIXR{n_{12}}{d_{12}}{n_{23}}{d_{23}}{n_{31}}{d_{31}}$. By proper design, we can make the function $P_r$ convex in $\TWO{x_i}{x_j}$ and in $\TWO{n_{ij}}{d_{ij}}$ but not both. We refer readers to~\prettyref{sec:collision-potential} for more details on the design of the function $P_r$.}
\end{figure}
In this work, we propose a variant of ADMM called Biconvex ADMM (BC-ADMM). For a special kind of non-convex optimization that endows a bi-convex relaxation, our algorithm takes as input a strictly feasible solution and maintains its feasibility during the optimization. BC-ADMM adopts the same algorithm pipeline as standard distributed ADMM~\cite{boyd2011distributed}. Under proper parameter choices, we show that our algorithm has an oracle complexity of $O(\epsilon^{-2})$. Further, by maintaining the feasibility of solutions, BC-ADMM pertains the desirable property similar to anytime motion planner~\cite{5980479}, aka. it allows the algorithm to be terminated anytime while returning a feasible solution. We propose practical implementations of our method and show that a large class of known optimization problems in motion planning, robot simulation, and trajectory optimization problems can be modeled after our framework. In the case of~\prettyref{fig:fail}, we can relax the collision constraint using the celebrated hyperplane separation theorem by inserting a separating plane between each pair of potentially colliding robots as illustrated in~\prettyref{fig:fail-reformulation}. Our \TE{main contributions} are summarized below:
\begin{itemize}[leftmargin=*]
\item A bi-convex relaxation of a class of non-smooth optimization problems, including \TE{motion planning, soft-robot simulation, multi-agent navigation, and trajectory optimization}.
\item A provably robust variant of ADMM for solving the relaxed constrained optimization.
\end{itemize}
\section{Related Work}
In this section, we first review the development and limitations of general-purpose constraint optimizers. Next, we review related works on structured optimization techniques. Finally, we discuss several applications in the robotic community, which involve optimization problems that can be handled by our BC-ADMM.

\subsection{General-purpose Constrained Optimizations}
General-purpose constrained optimization solvers are designed to be applicable for a large problem class with no assumptions on the underlying structures of objective and constraint functions. Typically, they only require these functions to be (twice-)differentiable. Algorithms for constrained optimization include the augmented Lagrangian method, the sequential quadratic programming, and the interior point method. These algorithms have been built into mature software packages, such as LANCELOT~\cite{conn2013lancelot}, SNOPT~\cite{gill2005snopt}, and IPOPT~\cite{wachter2002interior}, and are widely used in optimization problems from robotic applications. However, their convergence to feasible solutions relies on strong assumptions on the Jacobian of constraint functions, which are known as constraint qualifications~\cite{boggs1995sequential,birgin2008improving}. These constraint qualifications ensure that an infeasible solution can be guided back to the feasible domain by the gradients of the constraint functions. Unfortunately, it is non-trivial to check and ensure that constraint qualifications are satisfied for practical problems. Even worse, in problems as illustrated in~\prettyref{fig:non-smooth}, constraint qualifications are known to fail in many cases. All these factors lead to the lack of robustness in using general-purpose constrained optimizers with failure cases observed and analyzed in several prior works~\cite{schulman2014motion,byravan2014space,pan2022environment}.

Recently, it has been discovered that several non-smooth constraint functions are indeed smooth when restricted to their feasible domain, which includes collision constraints~\cite{li2020incremental,karamouzas17,ni2021robust}, inversion-free constraints~\cite{fu2016computing}, and strain limits~\cite{li2020incremental}. Researchers thus propose to use a feasible interior point method to restrict the solution to its feasible domain, yielding much better robustness than their infeasible counterpart. However, such robustness comes at the cost of computational efficacy~\cite{nocedal1999numerical}. This is because the feasible interior point method needs to take very small step sizes to ensure all the constraints are satisfied after moving the solution. As a result, the step sizes are typically upper-bounded by the most-stringent constraint.

\subsection{Structured Optimization Algorithms}
Compared with general-purpose algorithms, structured optimizers exploit additional structures in the objective and constraint functions to significantly boost the performance. For example, when both objective and constraints are convex, then globally optimal solutions can be found via local optimization~\cite{boyd2004convex}. When the objective function is multi-convex~\cite{xu2013block}, then block coordinate descent algorithms can be used with guaranteed convergence. When the objective function is a submodular quadratic function, then binary optimization can be solved efficiently using graph-cut. In our work, we propose a variant of ADMM, which is yet another class of structured optimization algorithms. ADMM works by breaking a large-scale problem into many small problems that can be solved in a numerically stable and efficient manner. Early works on ADMM focus on convex optimization problems~\cite{boyd2011distributed,deng2016global}. Over the years, the convergence of ADMM has been confirmed for a much wider class of constrained optimization problems with non-convex objectives and possibly non-convex constraints (see e.g.~\cite{wang2019global}). The massively parallel nature of ADMM makes it a stellar fit for efficiently solving large-scale problem instances. Indeed, ADMM has been adopted to handle collision and joint constraints for both physics simulation~\cite{overby2017admmpd} and geometric processing~\cite{zhang2019accelerating}. However, we notice that state-of-the-art ADMM convergence analysis requires the associated constraint graph to contain no loops, which is the focus of this work.

\subsection{Applications of BC-ADMM}
We discuss several robotic applications of our BC-ADMM, and we show that all these methods use the feasible interior point method, which suffers from small step sizes.

\paragraph{Multi-agent Navigation}
In early works, navigating agent motions are simulated using force-based models~\cite{reynolds1987flocks} or continuum models~\cite{hughes2002continuum}. However, these methods suffer from numerical stability issues and require careful parameter tuning to generate smooth and collision-free agent motions. Velocity-based methods~\cite{van2008reciprocal} are then proposed to introduce hard velocity constraints between a pair of agents to resolve collisions. But these constraints are solved for each agent separately, which limits the allowed range of collision-free motions. Most recently, \citet{karamouzas17} draws connections with optimization-based numerical integrators and introduces a formulation that jointly optimizes agent positions at the next timestep under pairwise barrier functions to prevent collisions between agents.

\paragraph{UAV Trajectory Optimization}
Generating collision-free, executable, and optimal UAV trajectories has been an active area of robotic research~\cite{deits2015efficient,zhou2019robust,8462878,ni2021robust,ni2022robust}. Among other technical disparities, these formulations vary in their techniques for handling the collision constraints. The non-convex collision constraints can be relaxed into a disjoint union of convex constraints~\cite{deits2015efficient,8462878} by pre-computing a set of convex corridor shapes, under which the globally optimal trajectory can be found by solving a mixed-integer convex program. On the other hand, \citet{8462878} proposed to introduce non-convex soft penalty functions to search for locally optimal collision-free trajectories, without utilizing safe corridors. \citet{ni2021robust} further enhanced the soft-penalty function into a log-barrier function, providing rigorous collision-free guarantees. Unsurprisingly, their underlying feasible interior point optimizer suffers from the small step size. In view of this drawback, follow-up works~\cite{ni2022robust,wang2020alternating} propose to decompose the optimization problem into sub-problems and adopt alternating minimization. However, none of these algorithms can decouple the collision constraints into sub-problems, which significantly limits the potential of alternating minimization.

\paragraph{Deformable Object \& Soft Robot Simulation}
Deformable objects are ubiquitous in robotic manipulation tasks. The simulation, manipulation, and control of deformable bodies have been an active area of research for decades~\cite{rus2015design}. In particular, various formulations have been proposed to model the dynamics of deformable objects, including mass-spring systems~\cite{huang2020dynamic,baraff2023large}, As-Rigid-As-Possible (ARAP) deformations~\cite{sorkine2007rigid,kwok2017gdfe,fang2018geometry}, and the finite element method~\cite{duriez2013control,qin2024modeling}. Whichever formulation is adopted, a practical deformable body simulator needs to handle various non-convex constraints, of which the two most common constraints stem from collision and strain limits~\cite{goldenthal2007efficient,wang2010multi,kim2012comparison}. The strain limits are introduced to model the behavior where a soft body hardens infinitely under large forces to prevent compression or stretching beyond the material limit. These hard constraints can prevent the underlying optimizer from taking large step sizes and making faster progress. This issue has been partially addressed by prior works~\cite{liu2013fast,overby2017admmpd}. Unfortunately, these formulations either cannot handle both types of constraints or do not have convergence guarantees.
\section{Problem Formulation}
We begin the construction of BC-ADMM by considering the unconstrained optimization problem induced from the feasible interior point method. Considering the following optimization corresponding to the toy problem in~\prettyref{fig:fail}:
\begin{equation}
\begin{aligned}
\label{eq:fail}
\argmin{x}\;&\sum_{i=1}^3\|x_i-x_i^\star\|^2+\\
&\tilde{P}_r(x_1,x_2)+
\tilde{P}_r(x_2,x_3)+
\tilde{P}_r(x_3,x_1),
\end{aligned}
\end{equation}
the three barrier potential terms each model a collision constraint between a pair of robots. However, if solve~\prettyref{eq:fail} using a global step size, then the step size must be upper bounded by the most stringent constraint, which oftentimes hinders convergence speed. Instead, ADMM works by decoupling the constraint barrier functions from the objective function. Specifically, we define the objective function as $f(x)=\sum_{i=1}^3\|x_i-x_i^\star\|^2$ and the barrier functions are summarized into another function:
\begin{align}
\label{eq:fail-g}
\tilde{g}(y)=\tilde{P}_r(y_1,y_2)+
\tilde{P}_r(y_3,y_4)+
\tilde{P}_r(y_5,y_6),
\end{align}
with $y\in\mathbb{R}^{12}$. ADMM then works by reformulating~\prettyref{eq:fail} into the following equivalent form:
\begin{align}
\label{eq:nlp}
\argmin{x}\;f(x)+\tilde{g}(Ax),
\end{align}
which is our main problem under consideration with a general matrix $A$. In our toy problem, the matrix is defined as:
\begin{align}
A=\left(\setlength{\arraycolsep}{1pt}\begin{array}{cccccc}
I & 0 & 0 & 0 & I & 0 \\ 
0 & I & I & 0 & 0 & 0 \\
0 & 0 & 0 & I & 0 & I \\
\end{array}\right)^T\in\mathbb{R}^{12\times6}.
\end{align}
We further take the following assumption for~\prettyref{eq:nlp}:
\begin{assume}
i) $f:\mathbb{R}^n\to\mathbb{R}^+\cup\{\infty\}$ is continuous and differentiable in $\dom(f)$ and $\tilde{g}:\mathbb{R}^m\to\mathbb{R}^+\cup\{\infty\}$ is continuous and twice continuously differentiable in $\dom(\tilde{g})$. ii) $f(\bullet)+\tilde{g}(A\bullet)$ is coercive. iii) A strictly feasible $x^0\in\mathbb{R}^n$ is given, such that $f(x^0)+\tilde{g}(Ax^0)<\infty$.
\end{assume}
Many underlying problems in geometric processing, physics simulation, and trajectory generation satisfy our assumptions. For example, in inversion-free deformation, $f$ is a quadratic objective function, $\tilde{g}$ is the elastic energy with log-barrier functions to limit the strain, and the rest shape of the deformable body is strictly feasible and known. For multi-agent trajectory generation, $f$ is the trajectory smoothness objective, $\tilde{g}$ is the collision-free constraint, and the initial state of agents is collision-free and known.
\begin{remark}
The two properties of function $\tilde{g}$ are both critical to the convergence analysis of our algorithm. As an example, the function that is frequently considered in our applications is the log-barrier function $\tilde{g}(x)=-\log(x)$ (or its locally supported variant $\tilde{g}(x)=-\log_\epsilon(x)$ with similar properties), as in illustrative problem of~\prettyref{fig:fail}. This function has $\dom(\tilde{g})=\{x|x>0\}$ in which it is smooth and thus twice continuously differentiable. In addition, it is continuous in the entire domain $\mathbb{R}$. As a counter-example, we consider the function $\tilde{g}(x)=\iota_{\{x|x>0\}}(x)$, where $\iota_\mathcal{Z}(z)$ is the indicator function, which takes value $0$ when $z\in\mathcal{Z}$ and $\infty$ otherwise. We again have $\dom(\tilde{g})=\{x|x>0\}$ in which it is constant and thus twice continuously differentiable. However, this function is not continuous at $x=0$.
\end{remark}

\subsection{ADMM Relaxation}
To adopt structured optimization, ADMM reformulates~\prettyref{eq:nlp} using a slack variable $y$ into the following equivalent form:
\begin{align}
\label{eq:rnlp}
\argmin{x,y}f(x)+\tilde{g}(y)\quad\ST\;Ax=y.
\end{align}
The above reformulation allows ADMM to introduce the following augmented Lagrangian function associated with~\prettyref{eq:rnlp}:
\begin{equation}
\begin{aligned}
\mathcal{L}(x,y,\lambda)&=f(x)+\tilde{g}(y) \\
&+\frac{\beta}{2}\|Ax-y\|^2+\lambda^T(Ax-y).
\end{aligned}
\end{equation}
It is well-known that any solution of~\prettyref{eq:rnlp} corresponds to the critical point of the Lagrangian function in all the parameters $x$, $y$, and $\lambda$. The Lagrangian function is widely used to gauge convergence in the augmented Lagrangian method~\cite{conn2013lancelot} and sequential quadratic programming~\cite{boggs1995sequential}, and we refer readers to these works for more discussions. As a key difference from these methods, however, ADMM searches for the critical point in a decoupled manner by updating parameters $x$, $y$, and $\lambda$ sequentially. The main benefit of this strategy is that each subproblem has a small size and can be solved extremely efficiently. Taking the toy function $\tilde{g}$ from~\prettyref{eq:fail-g} for example, the $y$ update can be decomposed into three subproblems related to $\TWO{y_1}{y_2}$, $\TWO{y_3}{y_4}$, and $\TWO{y_5}{y_6}$, respectively, which can be solved in parallel.

For solving~\prettyref{eq:rnlp}, several variants of ADMM algorithms~\cite{deng2016global,jiang2019structured,themelis2020douglas} have been developed. However, we show that these algorithms all assume additional properties on the problem data that do not hold in the aforementioned problem instances. In particular,~\citet{deng2016global} assumed that $\tilde{g}$ is convex, which is not the case with collision or strain limits. \citet{jiang2019structured,themelis2020douglas} either assumed that $A$ has full row rank, which implies that there is always a consensus $x$ satisfying all the constraints as long as each constraint is satisfied in a separate manner, or they assume that $\tilde{g}$ is a real instead of an extended real function, which implies that $\tilde{g}$ cannot be used to model constraints. This is not the case with non-trivial problem instances. In the case of~\prettyref{fig:fail}, again for example, $A$ cannot have full row rank. Further, $\tilde{g}$ is a non-convex extended real function. As summarized in~\prettyref{table:ADMM-summary}, the existing ADMM is not a suitable solver for~\prettyref{eq:rnlp} even if we know that the problem is strictly feasible.
\begin{table}[h]
\centering
\scalebox{0.8}{
\begin{tabular}{cccc}
\toprule
Method & $f$ & $g$ & \TWORCell{Additional}{Assumption} \\
\midrule
\cite{deng2016global} & CVX,ER & CVX,ER & F \\
\cite{jiang2019structured} & NC,ER & NC,R & F  \\
\cite{themelis2020douglas} & NC,R & NC,ER & F,S \\
Ours & NC,ER & BCVX,ER & F \\
\bottomrule
\end{tabular}}
\caption{\label{table:ADMM-summary} Assumptions on the problem data in different methods. (NC: Nonconvex, CVX: Convex, BCVX: Bi-convex, R: Real Function, ER: Extended Real Function, F: Feasible solution exists, S: $A$ is Surjective and has full row rank) Our method allows the solution of problems where $g$ can be non-convex, extended-real and $A$ does not have full row rank.}
\end{table}

\subsection{Bi-convex Relaxation \& BC-ADMM}
To overcome the limitations of existing ADMM variants, BC-ADMM introduces a bi-convex relaxation. In the toy example shown in~\prettyref{fig:fail}, we demonstrate how the collision constraints can be reformulated using the hyperplane separation theorem, as illustrated in~\prettyref{fig:fail-reformulation}. This reformulation introduces a new set of variables representing the separating planes, which we concatenate into a vector $z$. While this increases the overall number of decision variables, it also induces a special structure in the problem, which forms the basis of our convergence analysis. In the case of~\prettyref{fig:fail-reformulation}, we define:
\begin{equation}
\begin{aligned}
z=&\SIXR{n_{12}}{d_{12}}{n_{23}}{d_{23}}{n_{31}}{d_{31}}\\
g(y,z)=&P_r(y_1,y_2,n_{12},d_{12})+\\
&P_r(y_2,y_3,n_{23},d_{23})+\\
&P_r(y_3,y_1,n_{31},d_{31}).
\end{aligned}
\end{equation}
In summary, we consider the bi-convex variant of~\prettyref{eq:nlp} with a slack variable $z$ as follows:
\begin{align}
\label{eq:snlp}
\argmin{x,z\in\mathcal{Z}}f(x)+g(Ax,z),
\end{align}
where we take the following additional assumption:
\begin{assume}
\label{ass:bivariable-case1}
i) $f:\mathbb{R}^n\to\mathbb{R}^+\cup\{\infty\}$ is continuous and differentiable in $\dom(f)$ and $g:\mathbb{R}^{m+l}\to\mathbb{R}^+\cup\{\infty\}$ is continuous and twice continuously differentiable in $\dom(g)$. ii) $f(\bullet)+g(A\bullet,\bullet)$ is coercive. iii) A strictly feasible solution $x^0\in\mathbb{R}^n,z^0\in\mathcal{Z}$ is given, such that $f(x^0)+g(Ax^0,z^0)<\infty$. iv) $g(\bullet,z)$ is convex,  $g(y,\bullet)$ is $\sigma$-strongly convex, and $\mathcal{Z}$ is convex and non-empty.
\end{assume}
In other words, the non-convexity of $g$ is encoded in the additional variable $z$, and $g$ is convex otherwise. We will show that a large class of problems has such relaxations. In our BC-ADMM algorithm, we relax~\prettyref{eq:snlp} similar to the standard ADMM algorithm and introduce a slack variable $y$, giving the following problem:
\begin{align}
\label{eq:sadmm}
\argmin{x,y,z}f(x)+g(y,z)\quad\ST\;Ax=y.
\end{align}
BC-ADMM then proceeds by introducing the following Lagrangian function:
\begin{equation}
\begin{aligned}
&\mathcal{L}(x,y,z,\lambda)\\
=&f(x)+g(y,z)+\frac{\beta}{2}\|Ax-y\|^2+\lambda^T(Ax-y).
\end{aligned}
\end{equation}
Finally, BC-ADMM updates $x,y,z,\lambda$ iteratively by the following rule, with superscript denoting the iteration number:
\begin{small}
\begin{subequations}
\begin{align}[left={\rotatebox[origin=c]{90}{$\text{BC-ADMM}$}\empheqlbrace}]
y^0=&Ax^0\quad
z^0=z(y^0)\quad
\lambda^0=\nabla_yg(y^0,z^0)\\
\label{eq:step-a}
x^{k+1}=&\argmin{x}\mathcal{L}(x,y^k,z^k,\lambda^k)+\frac{\beta_x}{2}\|x-x^k\|^2\\
\label{eq:step-b}
y^{k+1}=&\argmin{y}\mathcal{L}(x^{k+1},y,z^k,\lambda^k)+\frac{\beta_y}{2}\|y-y^k\|^2\\
\label{eq:step-c}
z^{k+1}=&\argmin{z\in\mathcal{Z}}\mathcal{L}(x^{k+1},y^{k+1},z,\lambda^k)\\
\label{eq:step-d}
\lambda^{k+1}=&\lambda^k+\beta(Ax^{k+1}-y^{k+1}).
\end{align}
\end{subequations}
\end{small}
Here, for~\prettyref{eq:step-a}, we assume that a locally optimal solution can be computed starting from the initial guess of $x^k$. For~\prettyref{eq:step-b} and~\prettyref{eq:step-c}, the locally optimal solution is also global due to convexity. Indeed, since $g$ is $\sigma$-strongly convex in $z$, we know that $z$ is the unique minimizer of the function $g$, which depends only on $y$. Due to the uniqueness, we can write $z$ as a function of $y$ denoted as $z(y)$. Note that the benefit of ADMM is retained in our bi-convex relaxation. Indeed, the update of $y$ and $z$ can both be done by solving small problems in parallel. Taking the toy problem of~\prettyref{eq:fail} for example, the $z$ update corresponds to adjusting the separating plane for each collision constraint, which can be decomposed into subproblems related to $\TWO{n_{12}}{d_{12}}$, $\TWO{n_{23}}{d_{23}}$, and $\TWO{n_{31}}{d_{31}}$, respectively.
\begin{remark}
\label{rem:well-definedness}
Note that the two steps~\prettyref{eq:step-b} and~\prettyref{eq:step-c} might not be well-defined because we only assume the function $g$ is proper as a function of $\TWO{y}{z}$ but $g(\bullet,z^k)$ or $g(y^{k+1},\bullet)$ might not be proper, meaning that $y^{k+1}$ and $z^{k+1}$ might not have a finite solution. But we can easily see that BC-ADMM is well-defined by induction. As our base step, note that $g(y^0,z^0)<\infty$ by our initial condition with $z^0\in\mathcal{Z}$. As our inductive step, suppose $g(y^k,z^k)<\infty$ with $z^k\in\mathcal{Z}$, then since~\prettyref{eq:step-b} is minimizing the objective, we have $g(y^{k+1},z^k)<\infty$. Further, since~\prettyref{eq:step-c} is minimizing the objective, we have $g(y^{k+1},z^{k+1})<\infty$.
\end{remark}

\prettyref{rem:well-definedness} implies that BC-ADMM can be practically implemented without encountering infinite values or improper functions during each subproblem solve. However, we are still unclear whether solving the series of subproblems would in turn solve the problem in~\prettyref{eq:sadmm}, which is the main goal of our analysis. To this end, we follow the standard analysis technique and treat the subproblem solvers in~\prettyref{eq:step-b}, ~\ref{eq:step-c}, and~\ref{eq:step-d} as oracles, while we analyze the number of BC-ADMM iterations needed until convergence. We will show that, by a proper choice of positive parameters $\beta,\beta_x,\beta_y$, each iteration generated by BC-ADMM will be strictly feasible and the algorithm converges with an oracle complexity of $O(\epsilon^{-2})$. The result is formally claimed below and proved in \iflong \prettyref{sec:convergence}. \else our extended paper. \fi
\begin{theorem}
\label{thm:BC-ADMM-case1}
Taking \iflong \prettyref{ass:bivariable} \else \prettyref{ass:bivariable-case1} \fi and under sufficiently large $\beta_x,\beta_y,\beta$, each iteration generated by BC-ADMM satisfies $f(x^k)+g(Ax^k,z^k)<\infty$. Further, BC-ADMM converges the $\epsilon$-stationary solution of~\prettyref{eq:snlp} with an oracle complexity of $O(\epsilon^{-2})$.
\end{theorem}
\TE{Sketch of Proof.} To establish~\prettyref{thm:BC-ADMM-case1}, we propose a novel analysis technique. Specifically, we note that, when the function $g$ is not an extended-real function, but rather a real function with Lipschitz continuous gradient, then~\prettyref{eq:sadmm} takes a similar form as the prior setting~\cite{jiang2019structured}. Therefore, we propose to replace $g$ with another function $G$, such that $G$ is a real function with Lipschitz continuous gradient. Further, we make sure that the two functions, $g$ and $G$, match up to second derivatives in a sufficiently large compact set, as illustrated in~\prettyref{fig:whitney}. Such a function $G$ is guaranteed to exist by the Whitney extension theorem~\cite{whitney1992analytic}. With the function $g$ replaced by $G$, we derive a modified BC-ADMM algorithm denoted as BC-ADMM$_G$, for which a slightly modified version of the analysis in~\cite{jiang2019structured} can be applied to establish convergence. Finally, we show that BC-ADMM$_G$ generates a solution trajectory that lies in the compact domain where $g=G$, which establishes the convergence of BC-ADMM.
\begin{figure}[ht]
\centering
\includegraphics[width=0.99\linewidth]{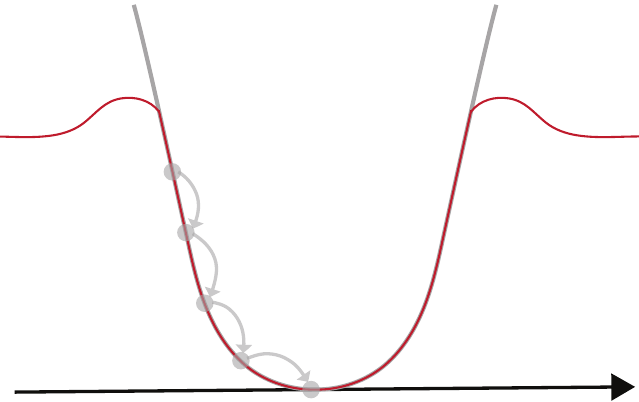}
\put(-180,140){$g$}
\put(-240,105){$G$}
\put(-20,10){$x$}
\put(-155,50){$\FOUR{x^k}{y^k}{z^k}{\lambda^k}$}
\caption{\label{fig:whitney} An illustration of our technique of analysis. We replace the extended real function $g$ (gray) with a real function $G$ with Lipschitz continuous gradient (red). With $g$ replaced by $G$, we derive a modified BC-ADMM algorithm denoted as BC-ADMM$_G$. The solution sequence generated by BC-ADMM$_G$ (gray dots and arrows), is restricted to the compact set where $g=G$, which establishes the convergence of BC-ADMM.}
\end{figure}

\prettyref{thm:BC-ADMM-case1} extends prior ADMM algorithms to solve optimizations under a large number of constraints in the form of soft barrier functions, while achieving the same order of convergence rate of optimizing general non-convex functions using the first-order methods~\cite{jiang2019structured}. However, we emphasize that the parameter choices for $\beta_x,\beta_y,\beta$ in our analysis are not defined in a constructive manner. In other words, we cannot design an algorithm to compute these parameters that achieves the claimed convergence speed. To tackle this difficulty, we propose a practical algorithm in the next section, which is guaranteed to converge under arbitrary choices of these parameters.
\section{Practical Algorithm and Applications}
Our analytic results require sufficiently large $\beta_x,\beta_y,\beta$, which cannot be algorithmically determined. Therefore, we need a practical algorithm that can automatically detect the parameters. In this section, we combine several approaches to detect these parameters, leading to an efficient and robust implementation of our method with an additional assumption that $f$ is convex. 

\subsection{Practical Algorithm}
Our practical implementation is summarized in~\prettyref{alg:practical}, where we adopt several approaches to ensure its robustness, anytime feasible property, and efficacy. First, to ensure that our method guarantees convergence without carefully tuning its parameters, we propose to use a lazy strategy to update $z^{k+1}$. Note that when $z^{k+1}$ is fixed, i.e.~\prettyref{eq:step-c} is not executed, BC-ADMM reduces to the standard ADMM applied to convex optimization when $f$ is convex. With $z$ fixed, a wide range of parameters leads to convergence of $x$ to the (globally) optimal solution of $f(\bullet)+g(A\bullet,z)$~\cite{deng2017parallel}. Therefore, we propose to only update $z^{k+1}$ when $x^{k+1}$ sufficiently reduces the objective function level and the optimal solution of the convex objective is sufficiently approached. Specifically, we require the objective function value to be reduced by at least $(1-\eta)(g(y^{k+1},z^{k+1})-g(y^{k+1},z^\star))$. We further require that the convergence measure $\Phi^{k+1}$ of the convex optimization has a norm less than $\eta^K$. We will show that these conditions will be ultimately met and $z^{k+1}$ will be updated. Further, our lazy update allows us to track a monotonically decreasing best solution denoted as $\langle x,z\rangle^\star$. This strategy is outlined in~\prettyref{ln:lazy-z}. In practice, we find that BC-ADMM converges even without these conditions, so we choose to be less stringent and set $\eta$ very close to $1$. 

Second, without sufficiently large parameters, we are not guaranteed to return a feasible solution upon early termination. To ensure the anytime feasible property, we propose to keep track of the best solution, denoted as $x^\star,z^\star$. When a solution is detected to be infeasible, we restart from $x^\star,z^\star$, while increasing $\beta_y,\beta$ by a constant factor of $\kappa_y$ and $\kappa$. Note that the factor $\kappa$ we use for $\beta$ is strictly larger than the factor $\kappa_y$ for $\beta_y$. This is because our analysis requires $\beta$ to be much larger than $\beta_y$. This strategy is outlined in~\prettyref{ln:feasibility}. Combining the strategies above, we argue for the asymptotic convergence of~\prettyref{alg:practical} in the following result, which is proved in \iflong \prettyref{sec:practical-convergence}. \else our extended paper. \fi
\begin{theorem}
\label{thm:asymptotic}
Assume that new constraints have never been detected at~\prettyref{ln:detector}. Taking~\prettyref{ass:bivariable} and assuming $f$ is convex,~\prettyref{alg:practical} converges to the $\epsilon$-stationary solution of~\prettyref{eq:snlp}.
\end{theorem}
\TE{Sketch of Proof} Our analysis is divided into two sections. First, we show that~\prettyref{ln:feasibility} would adaptively increase $\beta_x,\beta_y,\beta$. Ultimately, the assumptions of~\prettyref{thm:BC-ADMM-case1} will hold. In other words, these parameters will only be modified for finitely many iterations. Second, we analyze the lazy update in two cases, where $z^{k+1}$ can be updated for finitely or infinitely many times. In both cases, we combine the uniqueness of function $z(y)$ and the sufficient decrease condition below~\prettyref{ln:lazy-z} to establish the termination condition blow~\prettyref{ln:termination} after finitely many iterations.

\prettyref{thm:asymptotic} result immediately implies that the termination condition below~\prettyref{ln:termination} will be met within finitely many iterations. As a result,~\prettyref{alg:practical} becomes a practical optimization solver that converges with a wide range of parameter choices for $\beta_y$, and $\beta$, which automatically determines $\beta_x$. Unfortunately, we are unable to prove convergence speed in this case, but instead demonstrates satisfactory performance across several applications.

Finally, we observe that in many problems, a large number of constraints co-exist. When the number of agents in~\prettyref{fig:fail} grows, for example, we must introduce a collision constraint for each pair of agents, leading to a quadratic rate of growth in the number of terms in $g$ with respect to the number of agents. However, many of these constraints can be ignored when the pair of agents is far apart. To effectively limit the number of constraints, we could introduce a preemptive constraint detector. Formally, we assume $g=\sum_ig_i(y,z)$ is decomposed into finitely many terms, each modeling a singleton constraint. We further introduce the function $\mathcal{V}=$Constraint-Detected($x,g$), which determines whether $g_i$ is violated and returns the violated index set $\mathcal{V}=\{i|g_i(Ax,z)=\infty\;\forall z\in\mathcal{Z}, g_i\text{ not included in }g\}$. Using this function, we can start by ignoring all the constraints, and then iteratively introduce constraints preemptively detected. For example, in the case of~\prettyref{fig:fail}, the constraint detector can take the form of a collision detector~\cite{pan2012fcl}. Unfortunately, some constraint detectors can only find constraints that are already violated, but BC-ADMM requires the problem to be strictly feasible. Therefore, we propose to revert to the previous solution when such constraints are detected, where the problem data, including $f,g,A$, is modified. This strategy is outlined in~\prettyref{ln:detector} and we argue for the rigorous convergence under this setting in the following result.
\begin{corollary}
Taking~\prettyref{ass:bivariable} and assuming $f$ is convex,~\prettyref{alg:practical} converges to the $\epsilon$-stationary solution of~\prettyref{eq:snlp}.
\end{corollary}
\begin{proof}
This result only extends~\prettyref{thm:asymptotic} by enabling constraint detector. Note that when new constraints are detected in~\prettyref{ln:detector}, we would roll back to the previous best solution $\langle x,z\rangle^\star$. We claim that no constraints are violated at that configuration, because it will be rolled back to even earlier solutions otherwise. Since no constraints are violated at $\langle x,z\rangle^\star$, we conclude that $g_i(Ax^\star,z(Ax^\star))<\infty$ for any $i\in\mathcal{V}$. By our initialization of $z^{k+1}\gets z(y^{k+1})=z(Ax^\star)$ at~\prettyref{ln:init-new-z}, we know that $g(y^{k+1},z^{k+1})<\infty$, so that~\prettyref{ass:bivariable-case1} still holds and~\prettyref{alg:practical} is essentially restarting with a new function $g$ and a feasible solution. If no more constraints are detected, then the asymptotic convergence of~\prettyref{alg:practical} follows from~\prettyref{thm:asymptotic}. Our result follows by noting that new constraints can only be detected a finite number of times.
\end{proof}

\begin{algorithm}[ht]
\caption{\label{alg:practical}BC-ADMM}
\begin{algorithmic}[1]
\Require{$x^0$, $A$, convex $f$, bi-convex $g$, $\epsilon>0$}
\Require{$\kappa>\kappa_y>1$, $\gamma\in(0,1)$, $\eta\in(0,1)$, $\beta_y>0$, $\beta>0$}
\Ensure{$x^k$}
\State $\epsilon_x\gets\beta_y\gamma/(\beta_y+\beta\gamma)$,
$\epsilon_y\gets\beta\gamma/(\beta_y+\beta\gamma)$
\State $\beta_x\gets\beta(1/\epsilon_x-1)\|A^TA\|/\gamma$
\State $\langle x,z\rangle^0\gets
\langle x^0,z(y^0)\rangle$
\State $y^0\gets Ax^0$
\State $\lambda^0\gets\nabla_yg(y^0,z^0)$
\State $K\gets1$
\For{$k=0,1,2,\cdots$}
\State $x^{k+1}\gets\argmin{x}\mathcal{L}(x,y^k,z^k,\lambda^k)+\frac{\beta_x}{2}\|x-x^k\|^2$
\LineComment{The update in $y^{k+1}$ oftentimes allow parallel solver}
\State $y^{k+1}\gets\argmin{y}\mathcal{L}(x^{k+1},y,z^k,\lambda^k)+\frac{\beta_y}{2}\|y-y^k\|^2$
\State $z^{k+1}\gets\argmin{z\in\mathcal{Z}}\mathcal{L}(x^{k+1},y^{k+1},z,\lambda^k)$
\LineComment{Termination condition}
\label{ln:termination}
\If{\small$\left\|\TWOC
{\nabla_x f(x^{k+1})+A^T\nabla_y g(y^{k+1},z^{k+1})}
{Ax^{k+1}-y^{k+1}}\right\|_\infty\leq\epsilon$}
\State Return $\langle x,y,z\rangle^{k+1}$
\EndIf
\State $\Theta^{k+1}\gets\left\|\TWOC
{\nabla_x f(x^{k+1})+A^T\nabla_y g(y^{k+1},z^k)}
{Ax^{k+1}-y^{k+1}}\right\|_\infty$
\State $\Lambda^{k+1}\gets \begin{array}{c}f(x^\star)+g(Ax^\star,z^\star)+\\(1-\eta)(g(y^{k+1},z^{k+1})-g(y^{k+1},z^\star))\end{array}$
\LineComment{Track the best solution and lazy update $z^{k+1}$}
\label{ln:lazy-z}
\footnotesize
\If{$f(x^{k+1})+g(Ax^{k+1},z^{k+1})<\Lambda^{k+1}$$\land$ 
$\Theta^{k+1}\leq\eta^K$}
\normalsize
\LineComment{Handle new constraints detected}
\label{ln:detector}
\If{$\emptyset\neq\mathcal{V}\gets$Constraints-Detected($x^{k+1},g$)}
\State $\langle x,z\rangle^{k+1}\gets
\langle x,z\rangle^\star$
\State $y^{k+1}\gets Ax^{k+1}$
\State $\lambda^{k+1}\gets\nabla_yg(y^{k+1},z^{k+1})$
\LineComment{Insert detected constraints $\mathcal{V}$}
\State $g\gets g+\sum_{i\in\mathcal{V}}g_i$
\State $z^{k+1}\gets z(y^{k+1})$\label{ln:init-new-z}
\State Continue
\Else 
\State $\langle x,z\rangle^\star\gets
\langle x,z\rangle^{k+1}$
\State $K\gets K+1$
\EndIf
\Else
\State $z^{k+1}\gets z^k$
\EndIf
\State $\lambda^{k+1}\gets\lambda^k+\beta(Ax^{k+1}-y^{k+1})$
\LineComment{Ensure feasibility by increasing $\beta_x,\beta_y,\beta$}
\label{ln:feasibility}
\If{$f(x^{k+1})+g(Ax^{k+1},z^{k+1})=\infty$}
\State $\langle x,z\rangle^{k+1}\gets
\langle x,z\rangle^\star$
\State $y^{k+1}\gets Ax^{k+1}$
\State $\lambda^{k+1}\gets\nabla_yg(y^{k+1},z^{k+1})$
\State $\beta_y\gets\kappa_y\beta_y$,
$\beta\gets\kappa\beta$
\State $\epsilon_x\gets\beta_y\gamma/(\beta_y+\beta\gamma)$,
$\epsilon_y\gets\beta\gamma/(\beta_y+\beta\gamma)$
\State $\beta_x\gets\beta(1/\epsilon_x-1)\|A^TA\|/\gamma$
\EndIf
\EndFor
\end{algorithmic}
\end{algorithm}
\subsection{Applications}
We present several optimization problems in robotic applications that BC-ADMM can solve. To begin with, we prove the following result, which shows that if two functions $g_1$ and $g_2$ satisfy our assumption, then so does their summation $g_1+g_2$. As a result, we can arbitrarily composite the terms discussed in this section to deal with more complex problem instances.
\begin{corollary}
\label{cor:summation}
We suppose $g_1(y_1,z_1)\in\mathbb{R}^{m_1+l_1}$ and $\mathcal{Z}_1$ satisfy~\prettyref{ass:bivariable-case1} iv) and vi) and $g_2(y_2,z_2)\in\mathbb{R}^{m_2+l_2}$ and $\mathcal{Z}_2$ satisfy~\prettyref{ass:bivariable-case1} iv) and vi). Then $g=g_1+g_2\in\mathbb{R}^{m+l}$ with $m=m_1+m_2$, $l=l_1+l_2$, and $\mathcal{Z}=\mathcal{Z}_1\times\mathcal{Z}_2$ satisfy~\prettyref{ass:bivariable-case1} iv) and vi).
\end{corollary}

\begin{figure}[ht]
\centering
\includegraphics[width=0.83\linewidth]{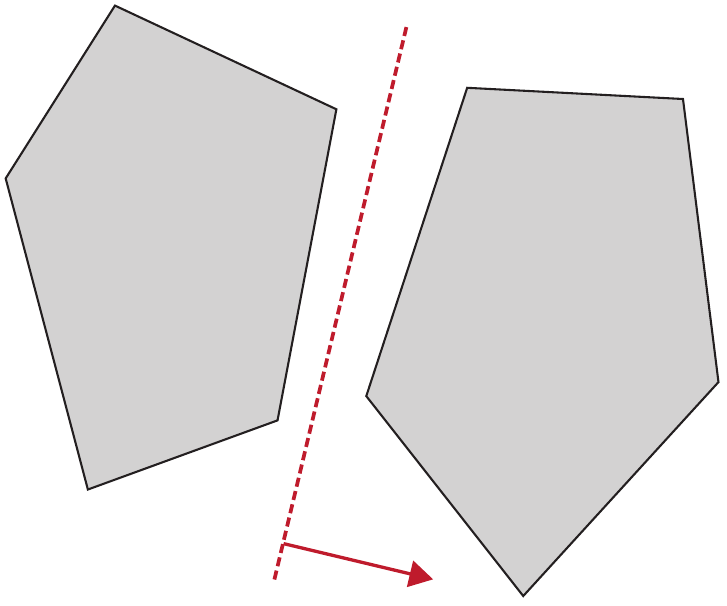}
\put(-60,15){$y_1$}
\put(-18,62){$y_2$}
\put(-25,128){$y_3$}
\put(-68,130){$y_4$}
\put(-92,56){$y_5$}
\put(-173,40){$y_6$}
\put(-135,55){$y_7$}
\put(-122,128){$y_8$}
\put(-168,150){$y_9$}
\put(-192,110){$y_{10}$}
\put(-100,15){$n$}
\put(-130,15){$d$}
\caption{\label{fig:collision} We illustrate the barrier potential energy to prevent collisions between a pair of convex objects, each having $5$ vertices so that $N=M=5$. The barrier energy introduces a separating plane, which is parameterized using variables $z=\TWO{n}{d}$, and requires that the convex objects reside in different sides of the separating plane.}
\end{figure}
\subsubsection{\label{sec:collision-potential}Collision Constrained Optimization}
As our first application, we consider the multi-robot pose optimization problem presented in~\prettyref{fig:fail}. The key challenge in this problem lies in the handling of the collision constraint, which was originally modeled as a potential $\tilde{P}_r(y_i,y_j)=-\log_\epsilon(\|y_i-y_j\|-2r)$. This potential is a general non-convex extended real function. However, a bi-convex potential has recently been proposed in~\cite{liang2024second}. As illustrated in~\prettyref{fig:collision}, we consider two convex hulls denoted as $\CH(y_{1,\cdots,N})$ and $\CH(y_{{N+1},\cdots,{N+M}})$, where the first hull has $N$ vertices $y_{1,\cdots,N}$ while the second hull has $M$ vertices $y_{{N+1},\cdots,{N+M}}$. To separate the two hulls,~\citet{liang2024second} introduced a separating plane denoted as $z=\TWO{n}{d}\in\mathbb{R}^4$, with $n$ being the plane normal and $d$ being the plane offset. Put together, the potential energy preventing the two hulls from collisions is formulated as:
\begin{equation}
\begin{aligned}
P_r(y,z)=&-\sum_{i=1}^N\log_\epsilon(y_i^Tn+d-r)\\
&-\sum_{i=N+1}^{N+M}\log_\epsilon(-y_i^Tn-d-r)+\frac{\sigma}{2}\|z\|^2,
\end{aligned}
\end{equation}
where we add a regularization $\sigma\|z\|^2/2$ to make it $\sigma$-strongly convex in $z$ and the locally supported log-barrier function is defined as $\log_\epsilon(x)=-\max(0,(\epsilon-x)^4/x^5)$. It is easy to see that $P(y,z)$ is well-defined, bi-convex, and satisfies~\prettyref{ass:bivariable-case1} with $\mathcal{Z}=\{n|\|n\|\leq1\}$. Indeed, we can show that $P_r(y,z)$ is monotonically decreasing with the length of normal vector $\|n\|$ and the additional constraint $z\in\mathcal{Z}$ enforces a unit normal $\|n\|=1$, leading to a well-defined separating plane. Revisiting the case of~\prettyref{fig:fail}, we can replace the potential $\tilde{P}_r(y_i,y_j)=-\log_\epsilon(\|y_i-y_j\|-2r)$ with the new potential $P_r(y,z)$ and $y=\TWO{y_i}{y_j}$ to prevent collisions between the agents. However, note that the collision constraint presented here applies to more general convex objects than two points $y_i,y_j$, including line segments and triangles, by changing the number of vertices in each convex hull. Finally, we emphasize that if all the collision constraints are added at once, there can be a quadratic number of terms in $g$, significantly increasing the computational burden. To alleviate this problem, we could use a collision detector such as~\cite{pan2012fcl} as the constraint detector in~\prettyref{alg:practical}.

\subsubsection{Mass-Spring System and Strain-Limit}
During the simulation and control of soft robots and deformable objects, the mass-spring system is one of the most prominent and fundamental material models~\cite{huang2020dynamic,baraff2023large}. A spring force between two particles $y_i$ and $y_j$ is modeled as the following potential energy:
\begin{align}
\tilde{P}_{k,l}(y)=\frac{k}{2}(\|y_i-y_j\|-l)^2,
\end{align}
where $k$ is the spring stiffness, $l$ is the rest length, and $y=\TWO{y_i}{y_j}$. Note that this potential energy does not satisfy our assumption because it is non-differentiable at $y_i=y_j$. However, the situation of $y_i=y_j$ is undesirable and corresponds to the configuration that the spring is infinitely compressed. Therefore, several prior works~\cite{goldenthal2007efficient,wang2010multi} propose to avoid such configurations via a hard constraint that $\|y_i-y_j\|\geq\underline{l}$ where $\underline{l}$ is the maximal allowed level of compression, which is known as strain limiting. We can realize such constraint via the following modified spring energy with strain limiting as follows:
\begin{align}
\tilde{P}_{k,l,\underline{l}}(y)=\frac{k}{2}(\|y_i-y_j\|-l)^2-\log_\epsilon(\|y_i-y_j\|-\underline{l}).
\end{align}
Equipped with strain limiting, $\tilde{P}_{k,l,\underline{l}}$ becomes twice continuously differentiable in its domain. Again, $\tilde{P}_{k,l,\underline{l}}$ is a non-convex extended real function that does not satisfy our assumption. To derive a bi-convex relaxation, we follow~\cite{liu2013fast} to introduce a direction $z\in\mathbb{R}^3$ and define:
\begin{equation}
\begin{aligned}
&P_{k,l,\underline{l}}(y,z)\\
=&\frac{k}{2}\|y_i-y_j-zl\|^2-\log_\epsilon(z^T(y_i-y_j)-\underline{l}),
\end{aligned}
\end{equation}
where $z$ is restricted to the surface unit sphere $\mathcal{Z}=\{z|\|z\|=1\}$. However, the set $\mathcal{Z}$ is non-convex and does not satisfy our~\prettyref{ass:bivariable-case1}. Instead, we propose a more general assumption and prove it holds for $P_{k,l,\underline{l}}(y,z)$ in \iflong \prettyref{lem:mass-spring}. \else our extended paper. \fi We further prove that the unique global minimizer of $z\in\mathcal{Z}$ is $z(y)=(y_i-y_j)/\|y_i-y_j\|$ and $P_{k,l,\underline{l}}(y,z(y))=\tilde{P}_{k,l,\underline{l}}(y)$.

\begin{figure}[ht]
\centering
\includegraphics[width=0.83\linewidth]{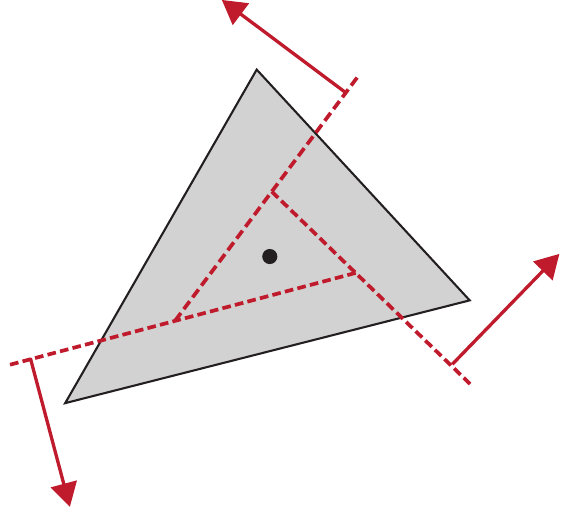}
\put(-100,85){$y_0$}
\put(-165,48){$y_1$}
\put(-58,75){$y_2$}
\put(-115,135){$y_3$}
\put(-35,50){$n_1$}
%
\put(-90,155){$n_2$}
%
\put(-198,38){$n_3$}
\caption{\label{fig:ARAP} We illustrate the barrier potential energy to enforce strain limiting for the ARAP material model in 2D cases ($d=2$). We first introduce an additional vertex $y_0$ and require that $y_0$ resides strictly inside the triangular element. This is achieved by introducing $d+1=3$ separating planes, parameterized with normal directions $z=\THREE{n_1}{n_2}{n_3}$, and requiring that each edge resides in the opposite side of the corresponding separating plane to $y_0$.}
\end{figure}
\subsubsection{ARAP Deformation} To simulate more complex material behaviors of soft robots and deformable objects, we need to introduce material models more general than the mass-spring system. One of the widely used material model is the ARAP elastic energy~\cite{sorkine2007rigid,kwok2017gdfe,fang2018geometry}. These techniques decompose a $d$-dimensional deformable object into a set of simplex elements. Each element is equipped with $d+1$ vertices so that $y\in\mathbb{R}^{(d+1)d}$ and a deformable gradient $F(y)\in\mathbb{R}^{d\times d}$. The ARAP energy is then formulated as:
\begin{align}
\tilde{P}_F(y)=\|F(y)-R(F(y))\|_F^2,
\end{align}
where $R(F(y))$ is the rotational part of the polar decomposition of $F(y)$. Intuitively, the ARAP energy encourages the soft robot to undergo a purely rigid transformation dictated by $R(F)$, and induces elastic energy otherwise. Similar to the case with mass-spring models, we can also introduce the idea of strain limiting. Let us denote $\sigma_k(F)$ as the $k$th singular value of $F$. If any $\sigma_k(F)$ is close to zero, then the corresponding element is significantly compressed along a right singular vector, which is unlikely to happen in real-world scenarios. To prevent these situations from happening, a well-known energy is the NeoHookean elastic potential~\cite{kim2012comparison}, which is formulated as:
\begin{equation}
\begin{aligned}
\tilde{P}_F(y)=&\|F(y)-R(F(y))\|_F^2-\logdet(F(y))+\\
&\sum_{i=1}^d\iota_{\sigma_k(F)>0}(F),
\end{aligned}
\end{equation}
where the additional $\iota_{\sigma_k(F)>0}$ functions ensure that $F$ has all positive singular values and the $-\logdet(F(y))$ term ensures that the potential function is continuous. Unfortunately, the NeoHookean does not have a bi-convex relaxation. Instead, we propose a novel biconvex relaxation to prevent $\sigma_k(F)$ from getting negative as illustrated in~\prettyref{fig:ARAP}, following a similar idea as~\cite{liang2024second}. We introduce an additional point $y_0\in\mathbb{R}^d$ as a decision variable. We further introduce $d+1$ separating plane normals $n_i$ to separate $y_0$ and each of the $d+1$ surfaces of the element. Mathematically, we define:
\begin{equation}
\begin{aligned}
&P_{F,\epsilon}(y,z)=\|F-R\|_F^2+\\
&\sum_{i=1}^d\iota_{\sigma_k(F)>0}(F)+\\
&\iota_{\CH(y_1,\cdots,y_{d+1})}(y_0)+\\
&\sum_{i=1}^{d+1}\frac{\sigma}{2}\|n_i\|^2
-\sum_{i=1}^{d+1}\sum_{0<j\neq i}\log_\epsilon(n_i^T(y_j-y_0)),
\end{aligned}
\end{equation}
where the additional $\iota_{\CH(y_1,\cdots,y_{d+1})}(y_0)$ term ensures that $y_0$ belongs to the simplicial complex, the regularization $\sigma\|n_i\|^2/2$ ensures $P_{F,\epsilon}$ is $\sigma$-strongly convex in $n_i$. Finally, we re-define $y=\THREE{y_0}{\cdots}{y_{d+1}}\in\mathbb{R}^{(d+2)d}$ and $z=\FOUR{n_1}{\cdots}{n_{d+1}}{R}$ and $\mathcal{Z}=\prod_{i=1}^{d+1}\{n_i|\|n_i\|\leq1\}\times\text{SO}(3)$. It is easy to see that $P_{F,\epsilon}(y,z)$ is a biconvex energy. Again, since $\mathcal{Z}$ is not a convex set, it does not satisfy our~\prettyref{ass:bivariable-case1}. We use the more general assumption and prove that our $P_{F,\epsilon}(y,z)$ satisfies it in \iflong \prettyref{lem:ARAP}. \else our extended paper. \fi Note that the subproblems in~\prettyref{eq:step-b} and~\prettyref{eq:step-c} of BC-ADMM become easy to solve, where~\prettyref{eq:step-b} amounts to a small convex optimization, while~\prettyref{eq:step-c} amounts to small convex optimizations and a polar decomposition.
\subsection{Numerical Stability}
Compared with general-purpose infeasible point constrained optimizers~\cite{gill2005snopt,wachter2002interior}, the higher numerical stability and robustness are the main benefits of ADMM-type algorithms. Indeed, when the constraint qualifications (nearly) fail, i.e., when the Jacobian of the constraint functions is nearly singular, the QP subproblem can be ill-conditioned. Instead, stable and robust algorithms exist for each and every subproblem of BC-ADMM. Specifically,~\prettyref{eq:step-a} corresponds to an unconstrained optimization of a sufficiently smooth function, for which either gradient descend or Newton's method can be used with proper globalization techniques such as line-search methods. \prettyref{eq:step-b} corresponds to a strongly convex optimization, for which robust algorithms have been well discussed in~\cite{boyd2004convex}. \prettyref{eq:step-c} amounts to compute the function $z(y^{k+1})$. If~\prettyref{ass:bivariable-case1} holds, then~\prettyref{eq:step-c} again amounts to a strongly convex optimization. If our more general~\prettyref{ass:bivariable} holds, then the stability of solving~\prettyref{eq:step-c} needs to be analyzed in a case-by-case manner. Fortunately, our function $z(y^{k+1})$ admits closed-form solutions in all the practical problems. In the case of the modified spring energy $P_{k,l,\underline{l}}(y,z)$, for example, the closed-form global minimizer is $z(y)=(y_i-y_j)/\|y_i-y_j\|$. In the case of the modified ARAP energy with strain limiting $P_{F,\epsilon}(y,z)$, the rotation matrix $R$ can be computed via the polar decomposition.
\section{Evaluation}
We use C++ to implement~\prettyref{alg:practical} and OpenMP to solve $yz$-subproblems in parallel. In all our examples, $f$ is always a quadratic function, so that the $x$-subproblem amounts to a linear system solve for which we use the sparse direct solver in the Eigen library. For all our examples, we choose $\kappa_y=2$, $\kappa=2.1$, $\gamma=0.95$, $\eta=1-10^{-5},$ and run experiments on a single desktop machine with the AMD Ryzen Threadripper 3970X CPU. The remaining parameters $\beta_y $ and $\beta$ need to be tuned for each problem, which is a standard treatment in practical implementations~\cite{overby2017admmpd,ni2022robust}. We evaluate our method on a row of four benchmarks listed below, where we compare our method with two standard baselines: Gradient Descent (GD) and Newton's Method. As a major difference from BC-ADMM, these algorithms solve~\prettyref{eq:snlp} without decoupling $f$ and $g$. Therefore, we only need to update two variables, $x$ and $z$, in an alternative manner. In GD, the simple first-order update rule for $x$ is used:
\begin{align}
x\gets x-\alpha_\text{GD}\left[\nabla_xf(x)+A^T\nabla_yg(Ax,z)\right],
\end{align}
with $\alpha_\text{GD}$ being the sufficiently small step size found by the line search procedure to ensure a strict function value decrease. The Newton's method is defined similarly, but uses the following second-order update rule:
\begin{small}
\begin{equation}
\begin{aligned}
x\gets x-&\left[\nabla_x^2f(x)+A^T\nabla_y^2g(Ax,z)A+\alpha_\text{Newton} I\right]^{-1}\\
&\left[\nabla_xf(x)+A^T\nabla_yg(Ax,z)\right].
\end{aligned}
\end{equation}
\end{small}
Here $\alpha_\text{Newton}$ is the sufficiently large regularization found by the line search procedure to ensure a strict function value decrease. After updating $x$, both GD and Newton's method update $z$ by setting $z\gets z(Ax)$. Typically, Newton's method achieves faster convergence, but each iteration is slower by requiring solving a large-scale sparse linear system. In our implementation, we use the state-of-the-art sparse linear solver~\cite{chen2008algorithm} to compute the second-order search direction. Finally, both methods suffer from the drawback of having to search for a global step size $\alpha_\text{GD}$ and $\alpha_\text{Newton}$ that can slow down the progress of the overall optimization.

\begin{figure}[ht]
\centering
\includegraphics[width=\linewidth,trim=25cm 0 25cm 0,clip]{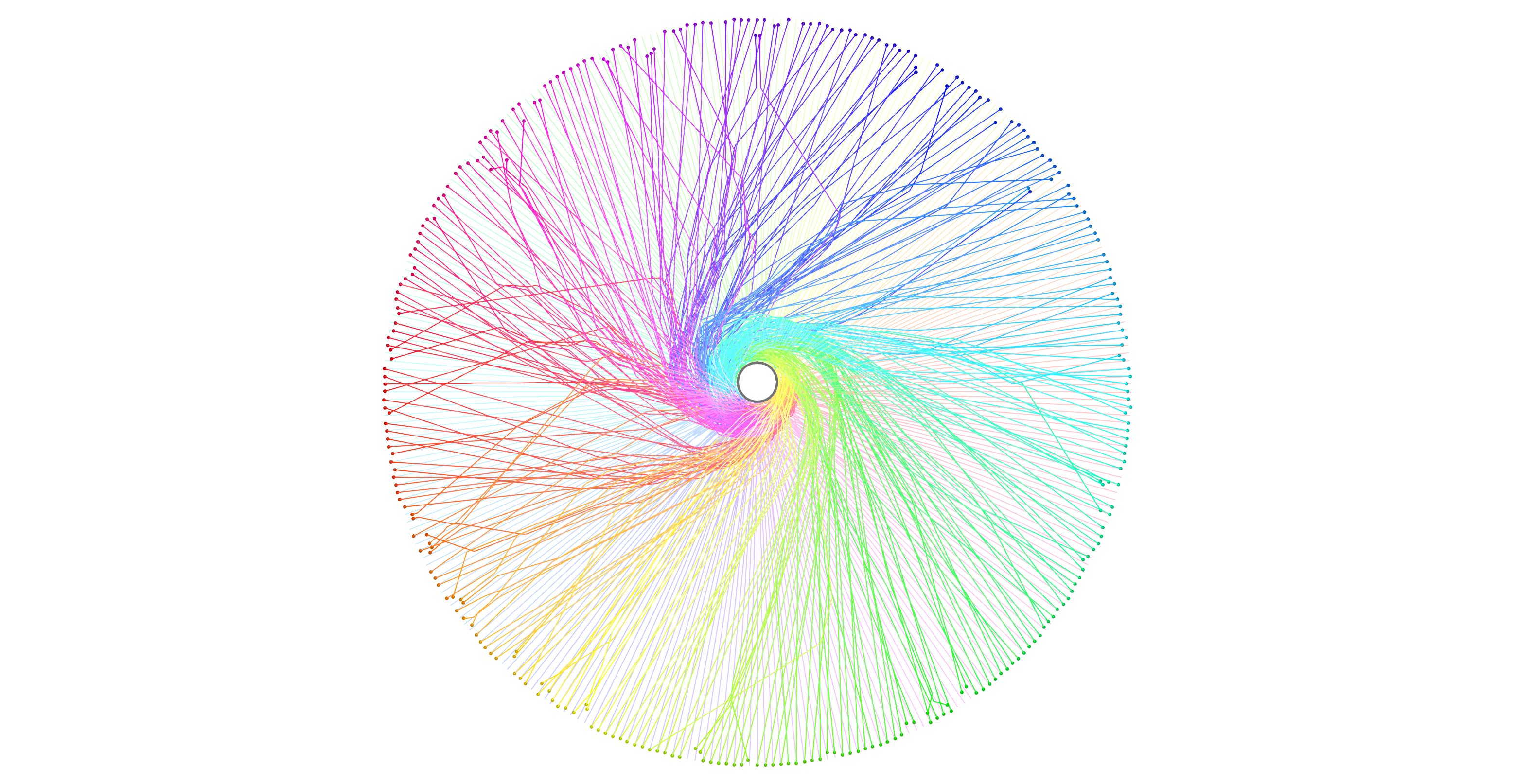}
\caption{\label{fig:navigation}We plan motions for $N=300$ robots with an obstacle in the middle and plot the trajectories computed by BC-ADMM.}
\end{figure}
\begin{figure}[ht]
\centering
\setlength{\tabcolsep}{0px}
\begin{tabular}{cc}
\includegraphics[height=.37\linewidth]{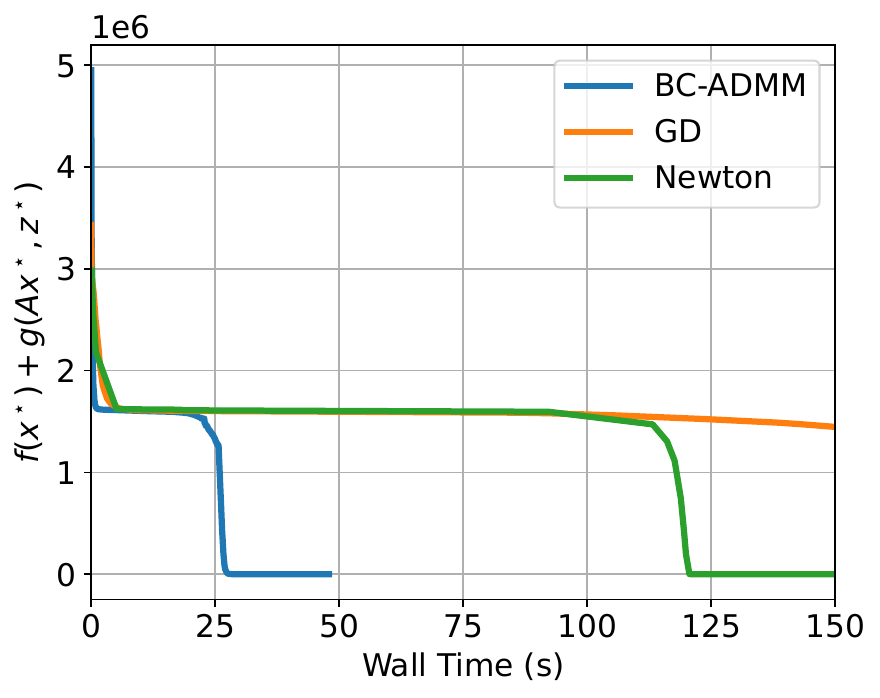}&
\includegraphics[height=.37\linewidth]{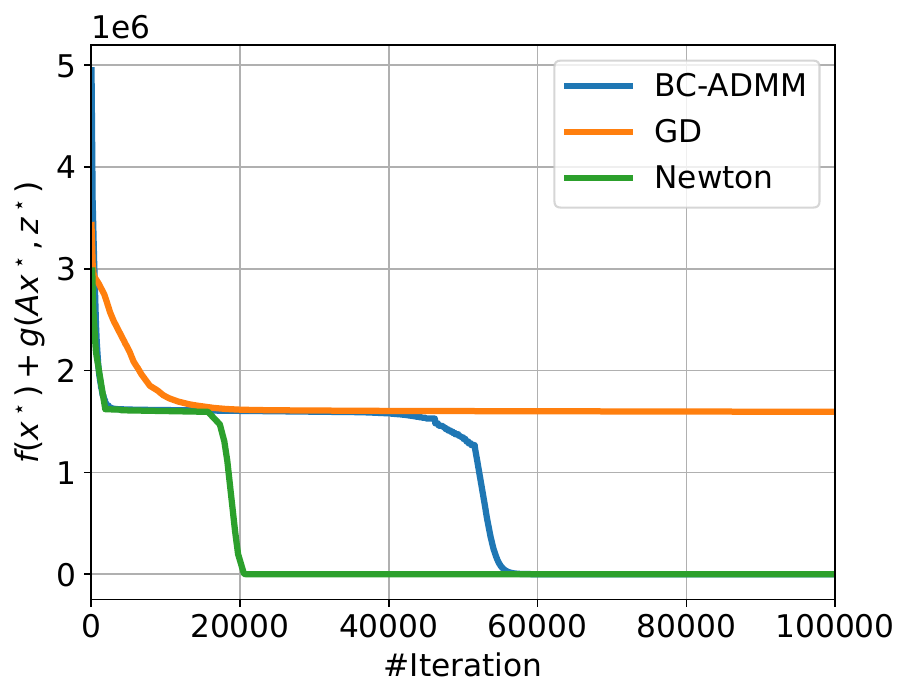}
\end{tabular}
\caption{\label{fig:navigation-conv-E} Objective function level for multi-agent navigation, plotted against wall time (left) and iteration number (right). All three solves would have the robot congested in the middle of the circle, forming a congested configuration. This corresponds to the plateau in the middle of convergence curve. After a while, BC-ADMM and Newton's method find wa ay for all robots to reach the other end.}
\end{figure}
\paragraph{Multi-Agent Navigation} We solve the large-scale multi-agent navigation problem~\cite{karamouzas17,weiss2017position}. We notice that we can compute collision-free navigation trajectories by solving a robot pose optimization problem as illustrated in~\prettyref{fig:fail}. Specifically, for a set of $N$ robots $x=\FOUR{x_1}{x_2}{\cdots}{x_N}$, we adopt the following setup:
\begin{equation}
\begin{aligned}
f(x)=&\sum_{i=1}^N\underbrace{\frac{1}{2}\|x_i-x_i^\star\|^2}_{\text{goal-reaching}}\\
g(Ax,z)
=&\sum_{i=1}^N\sum_{j=i+1}^N\underbrace{P_r(\TWO{x_i}{x_j},z_{ij})}_{\text{self-collision}}+\\
&\sum_{o_l}\sum_{i=1}^N\underbrace{P_r(\TWO{x_i}{o_l},z_{il})}_{\text{obstacle-collision}},
\end{aligned}
\end{equation}
where $x_i^\star$ is the goal of the $i$th robot, $r$ is the robot radius, and $P_r$ is the collision potential between a pair of geometric entities. Besides robot-robot collisions, there are static obstacles in the environment (the $l$th obstacle denoted as $o_l$), and we use two sets of collision potential terms to prevent their interpenetrations. The potential number of collision terms is quadratic, which leads to a large number of $yz$-subproblems to solve. To accelerate computation, we use a 2D collision checker as our constraint detector, where we only include $P_r$ terms in our optimization when the distance between two agents is less than $(2+\epsilon_\text{margin})r$ with $\epsilon_\text{margin}$ set to a margin distance. We emphasize that, although the above formulation only solves for a single pose for each robot, the optimizer would generate a series of intermediary poses during each iteration. These intermediary poses can be connected into a robot trajectory. Specifically, let us denote $k(i)$ as the index of the subset of iterations where the two conditions below~\prettyref{ln:lazy-z} of~\prettyref{alg:practical} are satisfied, i.e. $z^{k(i)+1}$ is updated to $z(y^{k(i)+1})$ at the end of the $k(i)$th iteration. After the optimization, we reconstruct the $i$th robot trajectory as $\FOUR{x_i^0}{x_i^{k(1)+1}}{x_i^{k(2)+1}}{\cdots}$. We claim that these trajectories are valid, collision-free trajectories navigating robots towards their goal $x_i^\star$. To see this is the case, note that $z^{k+1}$ is not updated between iterations $k(i-1)$ and $k(i)$. Since the second condition below~\prettyref{ln:lazy-z} is satisfied at $k(i)$th iteration, we have $\Theta^{k(i)+1}<\infty$ and thus:
\begin{align}
f(x^{k(i)+1})+g(Ax^{k(i)+1},z^{k(i-1)+1})<\infty.
\end{align}
Since the first condition below~\prettyref{ln:lazy-z} is satisfied at $k(i-1)$th iteration, we have:
\begin{align}
f(x^{k(i-1)+1})+g(Ax^{k(i-1)+1},z^{k(i-1)+1})<\infty.
\end{align}
By the definition of $P_r$, the above two conditions imply that robots at configuration $x^{k(i-1)+1}$ and $x^{k(i)+1}$ are both collision-free. Further, in both of these configurations, the pair of robots $x_i$ and $x_j$ is separated by the same plane $z_{ij}^{k(i-1)+1}$. Now, since the plane-separated regions are convex regions and linear interpolated points of a convex region stay inside this region, we conclude that any linear interpolated robot configurations between $x^{k(i-1)+1}$ and $x^{k(i)+1}$ are also collision-free. This proves that the piecewise linear robot trajectories $\FOUR{x_i^0}{x_i^{k(1)+1}}{x_i^{k(2)+1}}{\cdots}$ are valid, collision-free paths. In conclusion, we have shown that BC-ADMM is a suitable solver for multi-agent navigation problems. As illustrated in~\prettyref{fig:navigation} we conduct experiments on a standard benchmark with $N=300$ robots arranged in a circle with an obstacle in the middle, where each robot has their goal on the opposite side of the circle. All three solvers would generate similar robot trajectories, where robots are first congested in the middle and then push each other to find ways to reach the goals. The convergence history of all three algorithms is plotted in~\prettyref{fig:navigation-conv-E}. Our method converges after only $30$ seconds of computation. In comparison, GD cannot make sufficient progress even after $150$ seconds of computation. Instead, Newton's method takes a much smaller number of iterations to converge. But since each iteration is much more costly, their overall computational cost is higher. Finally, we compare the performance of a variant of~\prettyref{alg:practical} with the check for the two conditions below~\prettyref{ln:lazy-z} turned off. In other words, we always update $z^{k+1}\gets z(y^{k+1})$. In this case, there is no guarantee that the algorithm converges, unless we use the parameters $\beta_x,\beta_y,\beta$ satisfying~\prettyref{thm:BC-ADMM-case1} that are difficult to estimate. Using the same set of practical parameters, the performance of these two variants of BC-ADMM is profiled in~\prettyref{fig:navigation-conv-E-variant}. The result shows that this variant of BC-ADMM can frequently violate the constraints, leading to the energy value being $\infty$. Therefore, we cannot recover collision-free trajectories for the robots, rendering this variant of BC-ADMM an inappropriate solver for multi-agent navigation problems.
\begin{figure}[ht]
\centering
\setlength{\tabcolsep}{0px}
\begin{tabular}{cc}
\includegraphics[width=\linewidth]{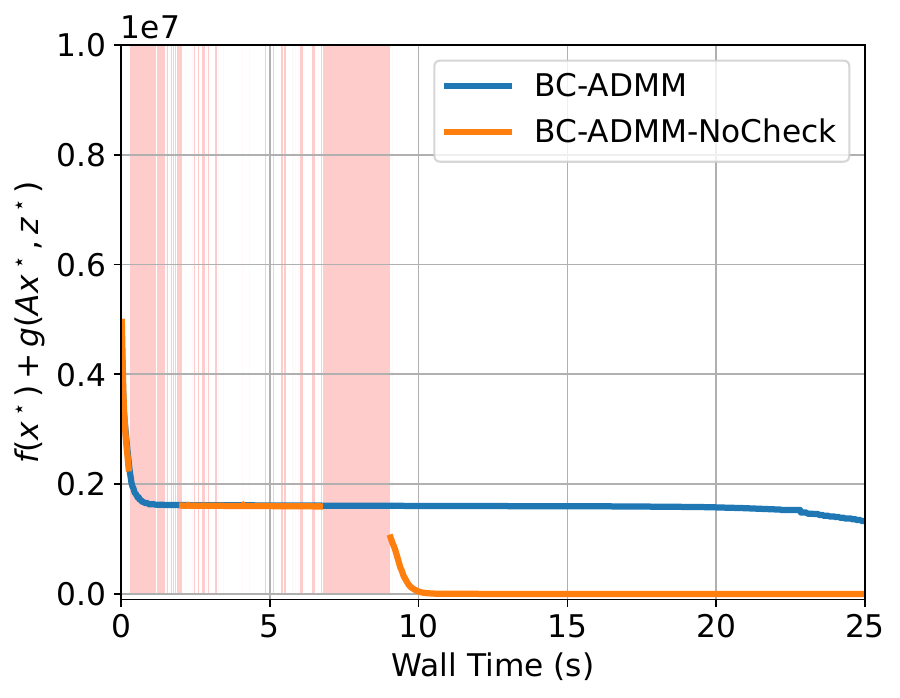}
\end{tabular}
\caption{\label{fig:navigation-conv-E-variant} We run multi-agent navigation using standard~\prettyref{alg:practical} (BC-ADMM) and a variant of~\prettyref{alg:practical} with the check for the two conditions below~\prettyref{ln:lazy-z} turned off (BC-ADMM-NoCheck). From the convergence history, we can see that, in many intermediary iterations, BC-ADMM-NoCheck violates the constraints, leading to the energy value being $\infty$, as marked by the vertical red lines. As a result, we cannot reconstruct collision-free trajectories for the robots, rendering BC-ADMM-NoCheck an inappropriate solver for multi-agent navigation problems.}
\end{figure}

\begin{figure}[ht]
\centering
\setlength{\tabcolsep}{0px}
\includegraphics[width=.99\linewidth]{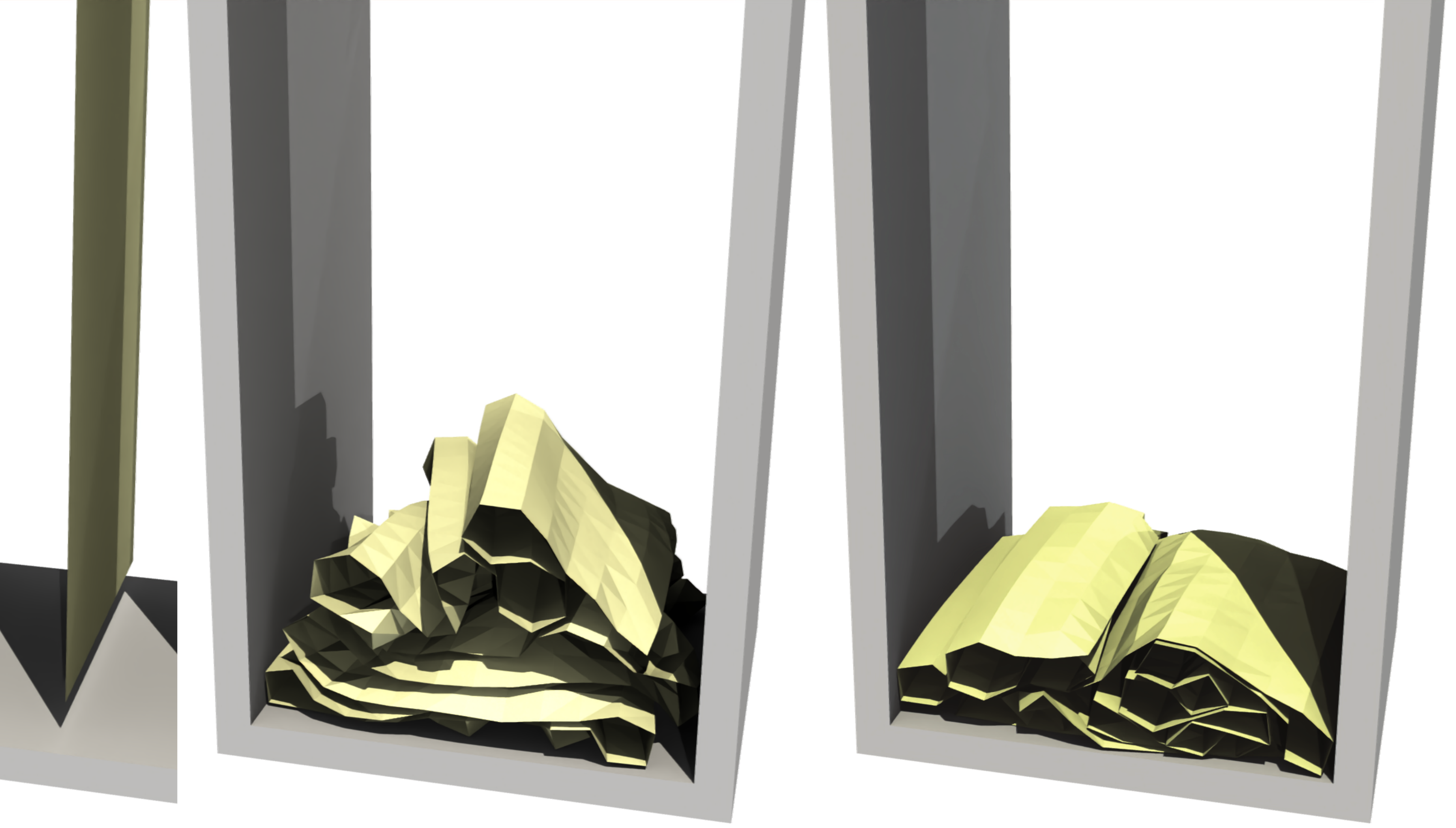}
\caption{\label{fig:object-settle}We drop a flat piece of cloth (left) on the ground and use BC-ADMM to compute its final configuration under a high (middle) and low (right) cloth thickness.}
\end{figure}
\begin{figure}[ht]
\centering
\setlength{\tabcolsep}{0px}
\begin{tabular}{cc}
\includegraphics[height=.37\linewidth]{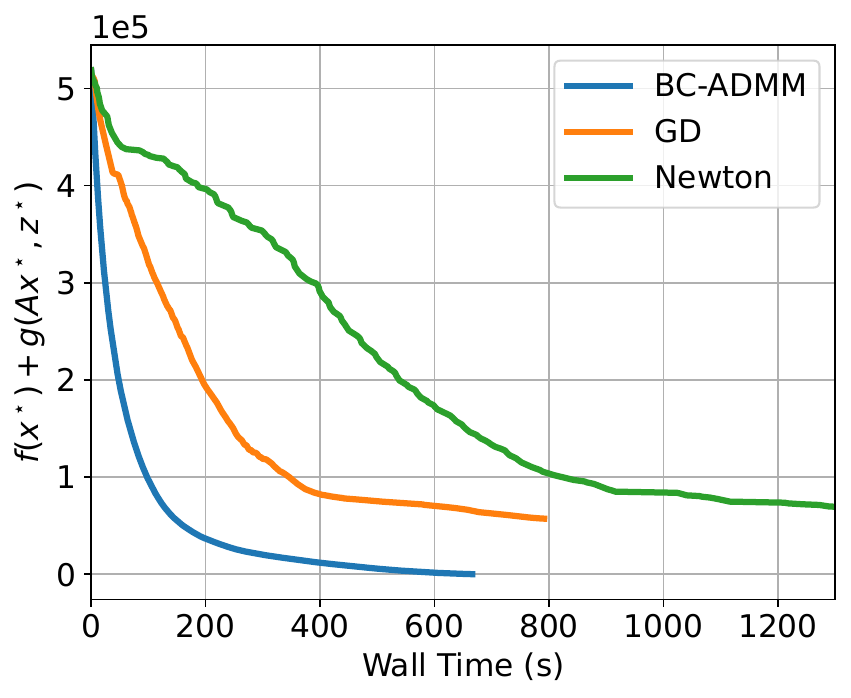}&
\includegraphics[height=.37\linewidth]{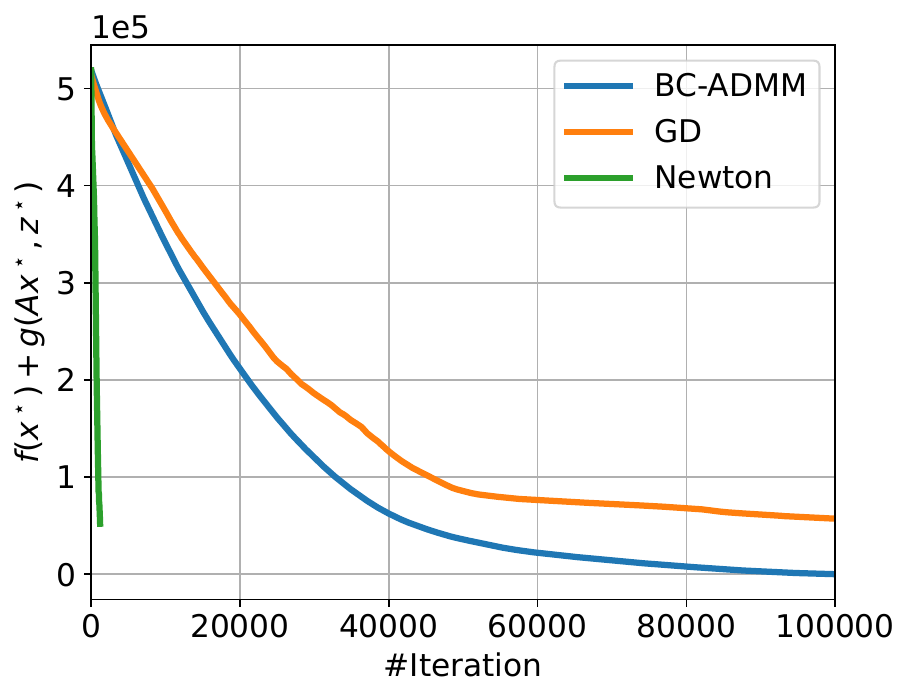}
\end{tabular}
\caption{\label{fig:cloth-conv-E} Objective function level for deformable object simulation, plotted against wall time (left) and iteration number (right). When BC-ADMM computes the ultimate pose for the piece of cloth, the two other solvers are still far from convergence.}
\end{figure}
\paragraph{Quasistatic Deformable Object Simulation} Deformable object simulation is a foundational module in robotic manipulation tasks~\cite{lin2022diffskill}. Thanks to our reformulated mass-spring material model $P_{k,l,\underline{l}}$, our BC-ADMM can be used as a deformable object simulator. Although ADMM algorithms have been used as physical simulators in prior works~\cite{overby2017admmpd,ouyang2020anderson}. However,~\citet{overby2017admmpd} used a soft penalty function to handle constraints, which is not guaranteed to resolve all interpenetrations. This is a well-known challenge in cloth simulation~\cite{harmon2009asynchronous,li2020incremental} where inter-penetrations are notoriously hard to resolve. Similarly,~\citet{ouyang2020anderson} did not consider any collision constraints or strain limits. Instead, our novel analysis and algorithm design allow all the constraints to be incorporated, making BC-ADMM a full-fledged and robust physical simulator. 

In our second benchmark, we consider a quasistatic simulation problem where a long flat piece of cloth is dropped into a box and our goal is to compute the ultimate configuration of the cloth. This problem is formulated as a potential energy minimization problem. We represent the cloth as a triangle mesh with $x$ being the concatenated vertex positions, and we denote each triangle of the mesh as consisting of three vertices, denoted $x_i,x_j,x_k$. We minimize the combined gravitational energy $f(x)=-x^TMg$ with $M$ being the mass matrix and $g$ being the gravitational acceleration. Further, the term $g$ is a summation over mass-spring energy terms $P_{k,l,\underline{l}}$ and collision potential terms $P_r$. Specifically, we model the box as a set of convex box-shaped obstacles. Between a pair of obstacle (the $l$th obstacle denoted as $o_l$) and the triangle $x_i,x_j,x_k$, we introduce a collision potential $P_r$ to prevent collision between the two geometric primitives. Another set of collision potential $P_r$ is introduced to prevent self-collisions. To model the material of cloth, we introduce two types of springs as done in~\cite{liu2013fast}, the 1-ring and 2-ring spring models, which model the stretch and bending resistance of the cloth, respectively. Note that in this benchmark, we do not have $f\neq0$, which violates our assumptions. But since the piece of cloth is confined to a box, we know that $f$ cannot go to $-\infty$, i.e. it is lower-bounded. Therefore, we can simply add a constant to $f$ to make it positive. Put together, our objective function is formulated as:
\begin{small}
\begin{equation}
\begin{aligned}
&f(x)=\underbrace{-x^TMg+\text{const.}}_{\text{gravitational-force}}\\
&g(Ax,z)\\
=&\sum_{\langle x_{ijk},o_l\rangle}\underbrace{P_r(\FOUR{x_i}{x_j}{x_k}{o_l},z_{ijkl})}_{\text{obstacle-collision}}+\\
&\sum_{\langle x_{ijk},x_{lmn}\rangle}\underbrace{P_r(\SIX{x_i}{x_j}{x_k}{x_l}{x_m}{x_n},z_{ijklmn})}_{\text{self-collision}}+\\
&\sum_{x_{ij}}\underbrace{P_{k,l,\underline{l}}(\TWO{x_i}{x_j},z_{ij})}_{\text{mass-spring}}.
\end{aligned}
\end{equation}
\end{small}
Our benchmark involves a mesh with $1111$ vertices, i.e. $x\in\mathbb{R}^{3333}$. Our formulation allows us to simulate a piece of cloth with different thickness, since the parameter $r$ in $P_r(y,z)$ can be modeled as the thickness parameter and the solutions under a high- and a low-thickness are illustrated in~\prettyref{fig:object-settle}. The convergence history of all three algorithms are summarized in~\prettyref{fig:cloth-conv-E}. In this setting, the iterative cost of Newton's method is much higher than the two other solvers, so that the first method outperforms. Again, BC-ADMM is the fastest to converge.

\begin{figure}[ht]
\centering
\includegraphics[width=\linewidth,trim=8cm 8cm 8cm 8cm,clip]{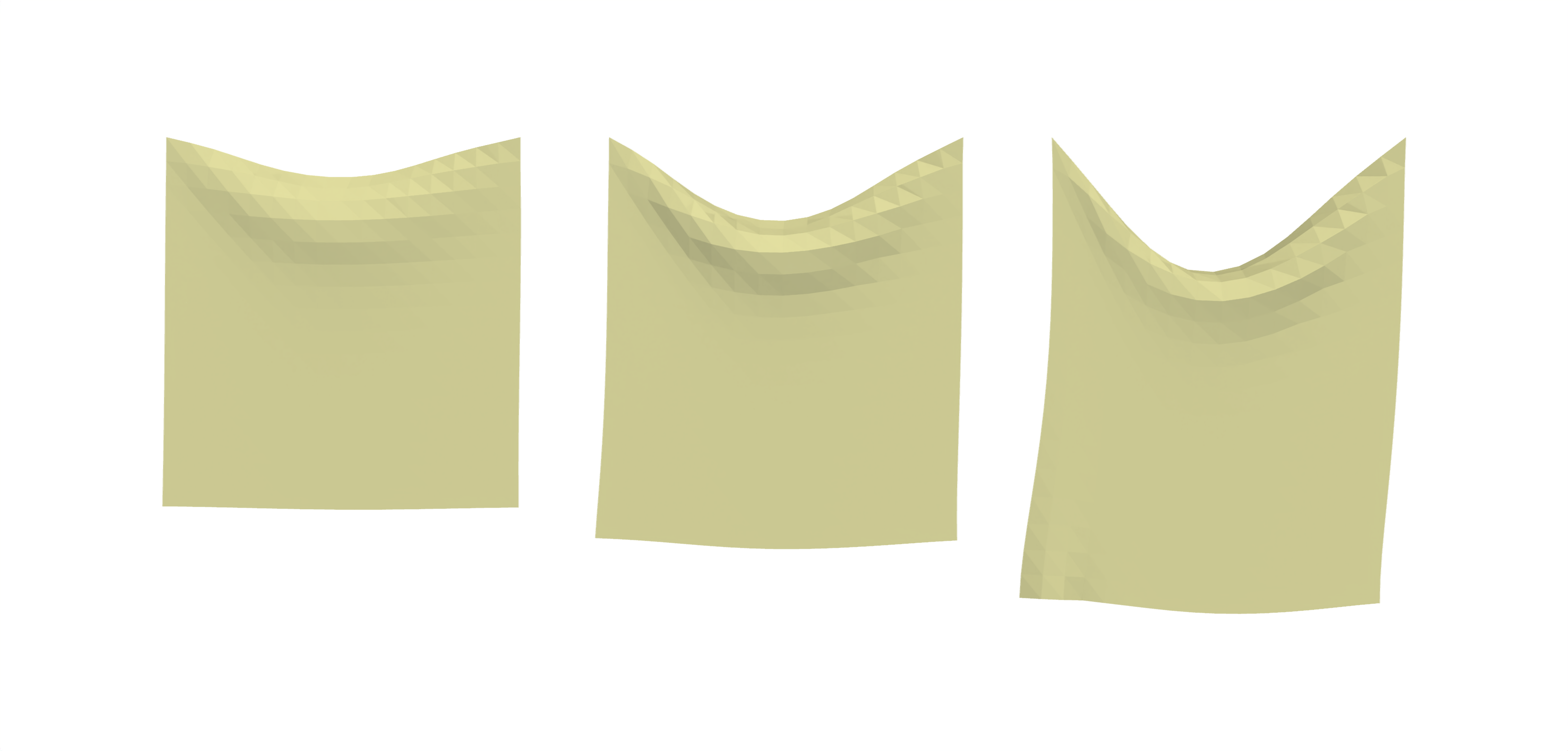}
\caption{\label{fig:inextensible}We drape three pieces of cloth with different extensible limit. From left to right, we use $(\bar{l}-l)/l=0.1, 0.2, 0.4$.}
\end{figure}
\begin{figure}[ht]
\centering
\includegraphics[width=\linewidth]{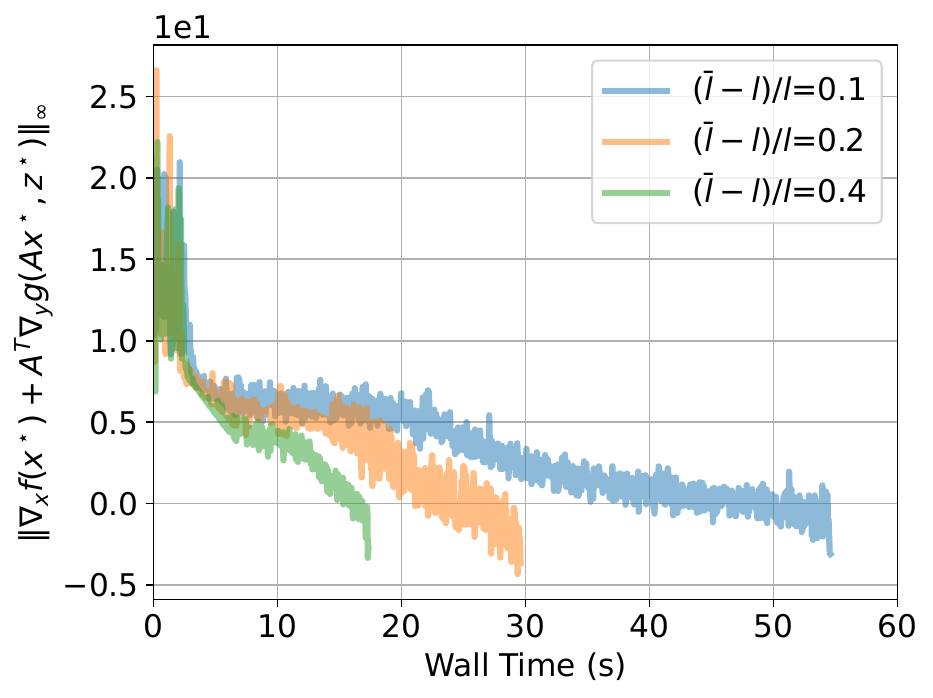}
\caption{\label{fig:inextensible-conv-E} With a more stringent extensible limit, BC-ADMM experiences slower convergence.}
\end{figure}
We further show that we can slightly extend our formulation to simulate nearly inextensible cloth~\cite{goldenthal2007efficient}, thus allowing our method to incorporate a wider range of material models. An inextensible cloth has a stretch limit of $\bar{l}>l$ such that a spring between $y_i$ and $y_j$ has the maximal length of $\bar{l}$. This behavior can be implemented by an additional energy term denoted as $P_{\bar{l}}(y)=-\log_\epsilon(\bar{l}-\|y_i-y_j\|)$. This energy term is convex in $y$ and thus trivially bi-convex, which is compatible with our assumption. Further, we note that $P_{\bar{l}}(y)$ is not differentiable and has a non-smooth singularity when $y_i=y_j$. But due to the strain limit in the mass-spring energy $P_{k,l,\underline{l}}$, we know that $\|y_i-y_j\|\geq\underline{l}$, so that this singularity is removable. Combined, the inextensible mass-spring potential is denoted as $P_{k,l,\underline{l},\bar{l}}(y,z)=P_{k,l,\underline{l}}(y,z)+P_{\bar{l}}(y)$. In~\prettyref{fig:inextensible}, we drape three pieces of cloth with different extensible limits. As shown in~\prettyref{fig:inextensible-conv-E}, a more stringent extensible limit leads to slower convergence, which is consistent with the observation in~\cite{goldenthal2007efficient}.

\paragraph{UAV Trajectory Optimization} In this benchmark, we solve the classical problem of optimizing a smooth and collision-free UAV trajectory~\cite{8462878,ni2021robust} by minimizing the trajectory-integrated jerk under various constraints. Specifically, a UAV trajectory is represented by a piecewise B\`ezier curve with each curve represented as a $5$th-order polynomial. Consecutive B\`ezier curves are $C^2$-smoothly connected with matching derivatives up to second order. We represent all the B\`ezier curves using their control points concatenated into a vector $x$, where we use $A_{ij}$ to denote the $ij$th block of matrix $A$ that extracts the $j$th control points of the $i$th B\`ezier curve. Similarly, $A_i$ is the $i$th block of matrix $A$ that extracts all the $6$ control points of the curve. There are obstacles in the environment, each represented using a convex object. To prevent the B\`ezier curve from intersecting the obstacle, we introduce a collision potential $P_r$ for each pair of B\`ezier curve $(A_ix)$ and convex obstacle $o_l$. Our collision potential follows a similar idea as~\cite{ni2021robust} that ensures all the control points and the obstacle lay on two different sides of the separating plane. Now since the actual curve is bounded by the convex hull of the control points, our collision potential ensures the entire curve is collision-free. Again, we use a collision detector to progressively introduce collision terms into the optimizer. We further introduce velocity and acceleration limits to ensure that the trajectory is executable. Specifically, to introduce a velocity upper bound of $\bar{v}$, we compute the derivative of each curve and extract its control points, with the $j$th control point of the $i$th derivative curve's denoted as $A_{ij}'x$ and we bound all the control points to be within the range $\bar{v}$ by introducing the log-barrier function $-\log_\epsilon(\bar{v}-\sqrt{\|A_{ij}'x\|^2+\epsilon})$, where we introduce a small regularization $\epsilon$ to ensure the function is smooth. A similar technique can be used to adopt the acceleration bound. Put together, we use the following problem definition:
\begin{equation}
\begin{aligned}
f(x)=&\underbrace{\frac{1}{2}x^TJx}_{\text{trajectory-integrated jerk}}\\
g(Ax,z)=&\sum_{\langle A_i,o_l\rangle}\underbrace{P_r(\TWO{A_ix}{o_l},z_{il})}_{\text{obstacle-collision}}+\\
&\sum_{A_{ij}'}\underbrace{-\log_\epsilon(\bar{v}-\sqrt{\|A_{ij}'x\|^2+\epsilon})}_{\text{velocity-bound}},
\end{aligned}
\end{equation}
where $J$ is the Hessian matrix of the trajectory-integrated jerk. Our formulation is the same as~\cite{ni2021robust}, which is more robust than prior work~\cite{8462878} by ensuring collision-free properties using log-barrier functions $P_r$. However, it is known that the log-barrier functions lead to slow convergence. Therefore,~\citet{ni2022robust} proposed to accelerate the solver by using ADMM. Unfortunately, their formulation does not allow splitting the collision potential terms $P_r$ into sub-problems, i.e., all the $P_r$ terms are treated as part of $f$ instead of $g$, which requires their solver to handle all the collision potentials together and use a line-search to upper bound the overall step size. Instead, our formulation allows a fully decoupled treatment where $P_r$ terms are handled separately using a $y$-subproblem. As illustrated in~\prettyref{fig:UAV}, we optimize a very complex trajectory that has the UAV flying through a series of circle-shaped obstacles of a varying radius, where the UAV trajectory is represented using 393 B\`ezier curves. The convergence history of all three algorithms is plotted in~\prettyref{fig:UAV-conv-E} in terms of energy level and gradient norm. In this example, BC-ADMM has almost the same performance as Newton's method in terms of wall time, which is much faster than GD. However, since each iteration of BC-ADMM is much faster than Newton's method, our algorithm is more amenable to early stop and returns a solution with moderately high quality, which is a common practice in the community of UAV trajectory optimization. From the plot of the gradient norm, we can see that BC-ADMM is the best performer. Although both BC-ADMM and GD are first-order algorithms, GD exhibits very noisy gradients, while BC-ADMM consistently reduces the gradient norm.
\begin{figure}[ht]
\centering
\includegraphics[width=\linewidth,trim=1cm 16cm 1cm 16cm,clip]{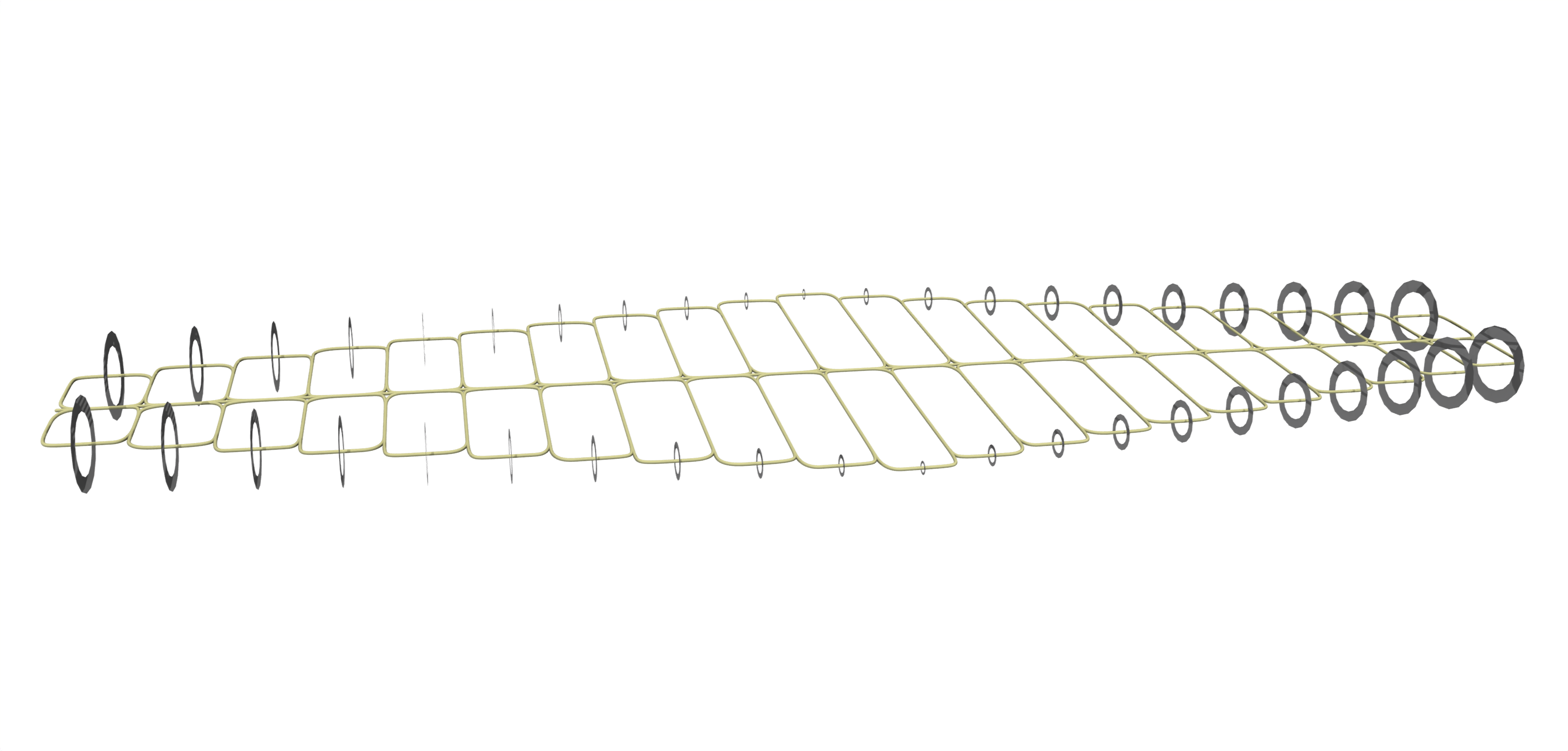}
\includegraphics[width=\linewidth,trim=1cm 16cm 1cm 16cm,clip]{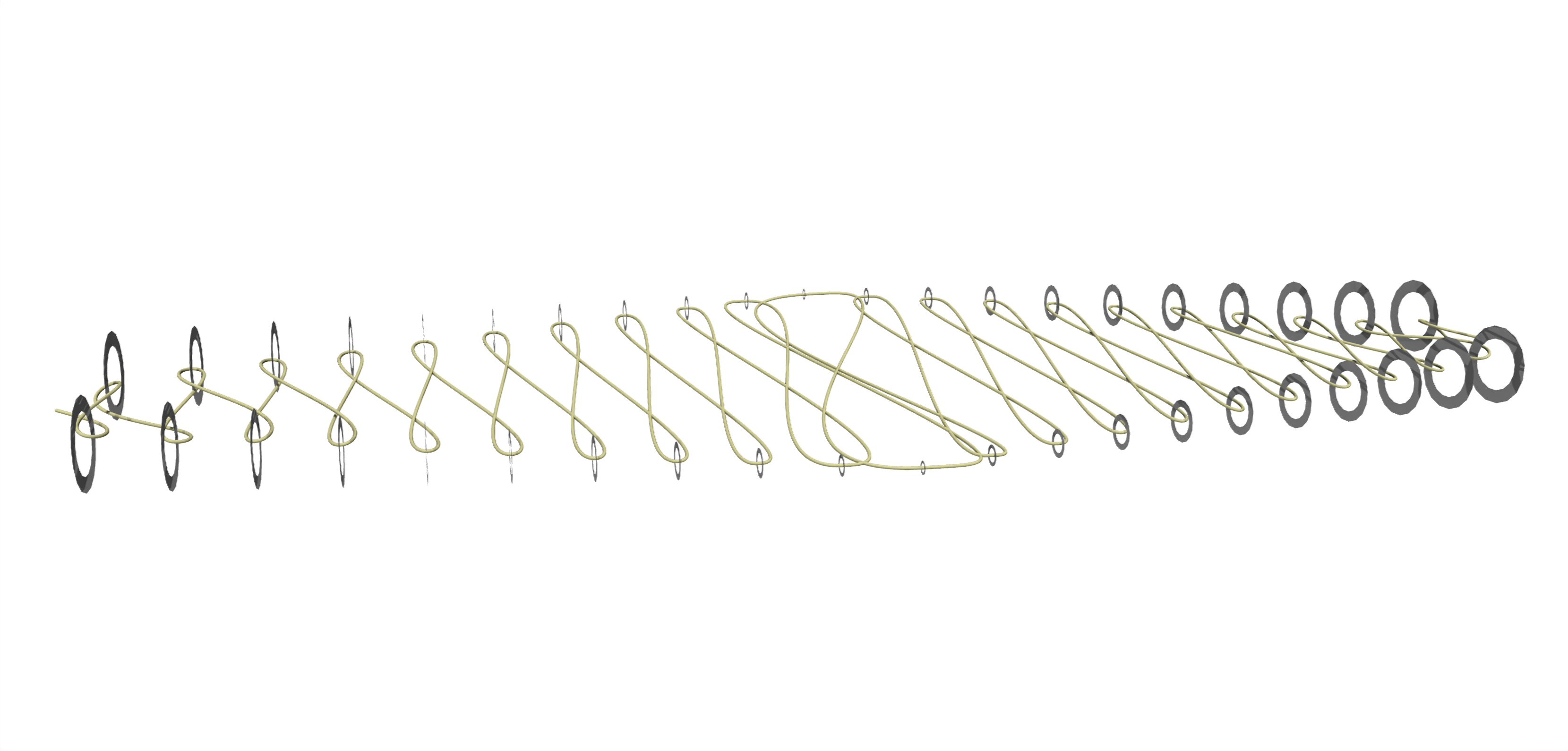}
\caption{\label{fig:UAV}We optimize a complex UAV trajectory where the UAV flies through a series of circle-shaped obstacles of varying radius. The top figure shows our initial trajectory, and the bottom figure shows the optimized configuration.}
\end{figure}
\begin{figure}[ht]
\centering
\setlength{\tabcolsep}{0px}
\begin{tabular}{cc}
\includegraphics[height=.37\linewidth]{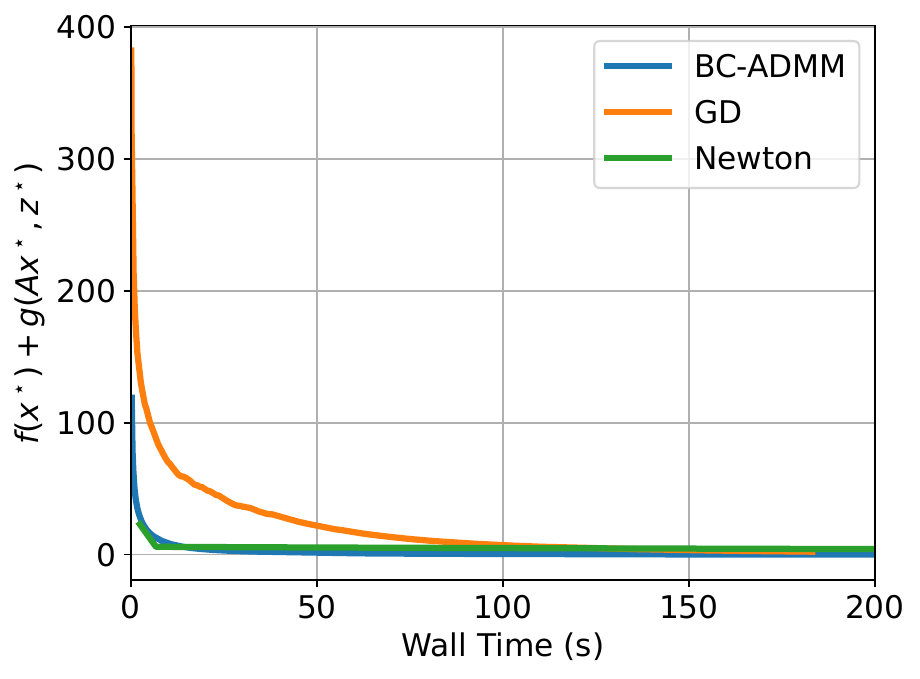}&
\includegraphics[height=.37\linewidth]{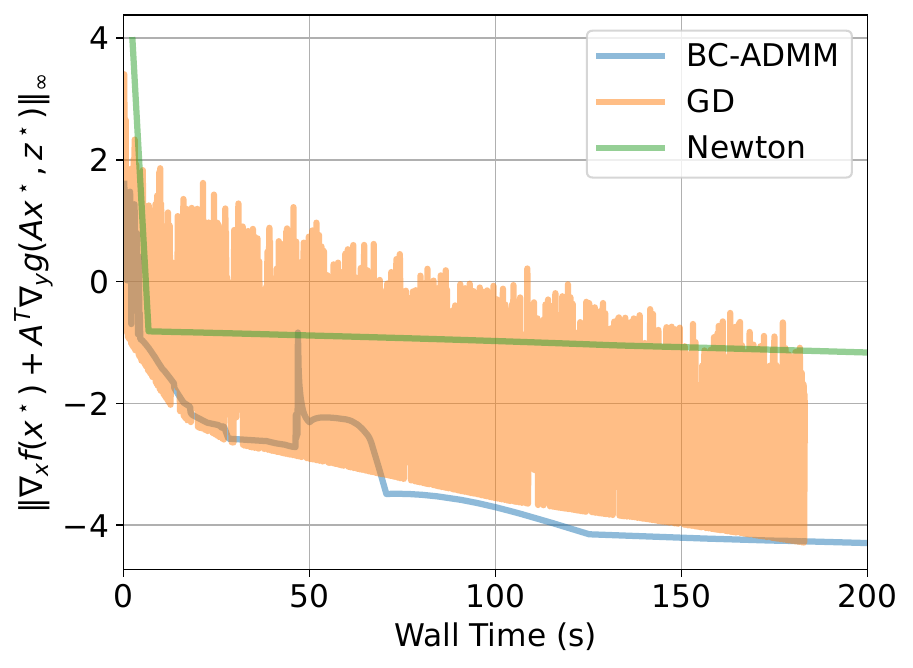}
\end{tabular}
\caption{\label{fig:UAV-conv-E} We plot the convergence history of the three algorithms in optimizing the UAV trajectory, where the left figure shows the energy value and the right figure shows the gradient norm. BC-ADMM has similar performance as Newton's method, significantly outperforms GD.}
\end{figure}

\paragraph{Soft Robot Gripper Simulation} We solve the problem proposed in~\cite{fang2018geometry}, where a soft robot is actuated by cables and serves as fingers of a gripper, treating ADMM as a quasistatic soft robot simulator. In our problem setting, the robot is discretized using a finite element mesh, and the internal elastic energy is discretized using our modified, bi-convex ARAP energy $P_{F,\epsilon}$. To model the cable-driven actuator, we notice that a cable segment connecting two points $x_i$ and $x_j$ can be exactly modeled as a convex potential energy:
\begin{align}
P_\text{cable,a}(x_i,x_j)=a\sqrt{\|x_i-x_j\|^2+\epsilon},
\end{align}
where $\epsilon$ is a small regularization to ensure smoothness and $a$ is the positive constant representing the internal tension force magnitude in the cable. It can be verified that the negative gradient of $P_\text{cable,a}$ is the force vector with magnitude $a$ and pointing along the line segment between $x_i$ and $x_j$. As usual, all the collisions are handled using our collision potential term $P_r$, and we use collision detectors to progressively introduce these potential terms. Put together, our problem setup is summarized below:
\begin{small}
\begin{equation}
\begin{aligned}
&f(x)=\underbrace{-x^TMg+\text{const.}}_{\text{gravitational-force}}\\
&g(Ax,z)\\
=&\sum_{\langle x_{ijk},o_l\rangle}\underbrace{P_r(\FOUR{x_i}{x_j}{x_k}{o_l},z_{ijkl})}_{\text{obstacle-collision}}+\\
&\sum_{\langle x_{ijk},x_{lmn}\rangle}\underbrace{P_r(\SIX{x_i}{x_j}{x_k}{x_l}{x_m}{x_n},z_{ijklmn})}_{\text{self-collision}}+\\
&\sum_{x_{ijkl}}\underbrace{P_{F,\epsilon}(\FOUR{x_i}{x_j}{x_k}{x_l},z_{ijkl})}_{\text{ARAP-material-model}}+\\
&\sum_{x_{ij}}\underbrace{P_\text{cable,a}(x_i,x_j)}_{\text{cable-driven actuator}}.
\end{aligned}
\end{equation}
\end{small}
In our benchmark, we model the soft robot gripper to grasp another soft ball, where the ball is modeled using ARAP material as well. We setup the ball to have a much smaller stiffness than the finger (we use $100:1$ stiffness ratio), so that the fingers can impose large grasping forces on the soft ball, forming a firm grip and causing extreme deformations in the ball. Starting from a trivial initialization, our goal is to compute the ultimate pose of the fingers and the soft ball. This is a challenging benchmark where large forces need to be transferred from the gripper to the ball, through collision potential terms $P_r$. Further, the extreme deformations can cause complex collisions between objects, and any violations of collision constraints can lead to infinite energy values and failure of optimization. Again, we show the initial and optimized configurations of this benchmark in~\prettyref{fig:Finger3D} and we plot the convergence history in~\prettyref{fig:Finger3D-conv-E}. Our BC-ADMM significantly outperforms the two other algorithms in terms of energy value and gradient norm. Thanks to our modified ARAP energy with strain limiting, BC-ADMM can simulate extremely large deformations without inverted elements as compared with the standard ARAP formulation, as illustrated in~\prettyref{fig:ARAP-Flip}.
\begin{figure}[ht]
\centering
\scalebox{0.9}{
\includegraphics[width=1.05\linewidth]{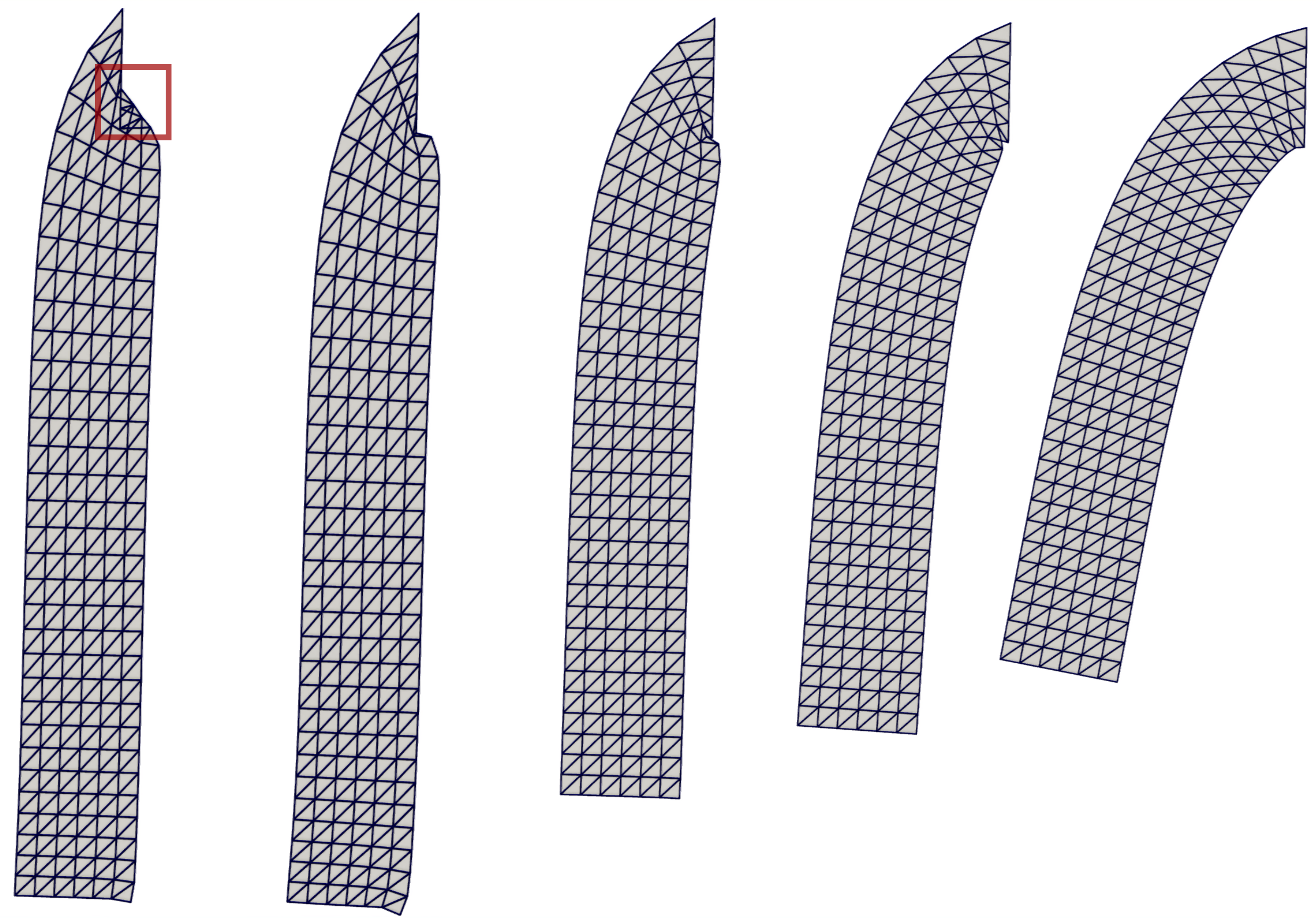}
\put(-60,40){\tiny$\|g\|=2.5$}
\put(-100,30){\tiny$\|g\|=5$}
\put(-145,17){\tiny$\|g\|=10$}
\put(-200,-5){\tiny$\|g\|=20$}
\put(-250,-5){\tiny$\|g\|=20$}
\put(-252,-12){\tiny{No Strain Limit}}}
\caption{\label{fig:ARAP-Flip}We simulate a soft beam under different gravitational coefficients (marked below) and compare our modified ARAP energy $P_{F,\epsilon}$ with the standard ARAP energy without strain limiting by setting $P_{F}=\|F-R\|_F^2$. Under a large gravitational force $\|g\|=20$, the result without strain limiting exhibits an inverted element (red box), while our ARAP energy with strain limiting always predicts faithful deformation results.}
\end{figure}
\begin{figure}[ht]
\centering
\setlength{\tabcolsep}{0px}
\begin{tabular}{cc}
\includegraphics[width=.5\linewidth,trim=33cm 0cm 34cm 0cm,clip]{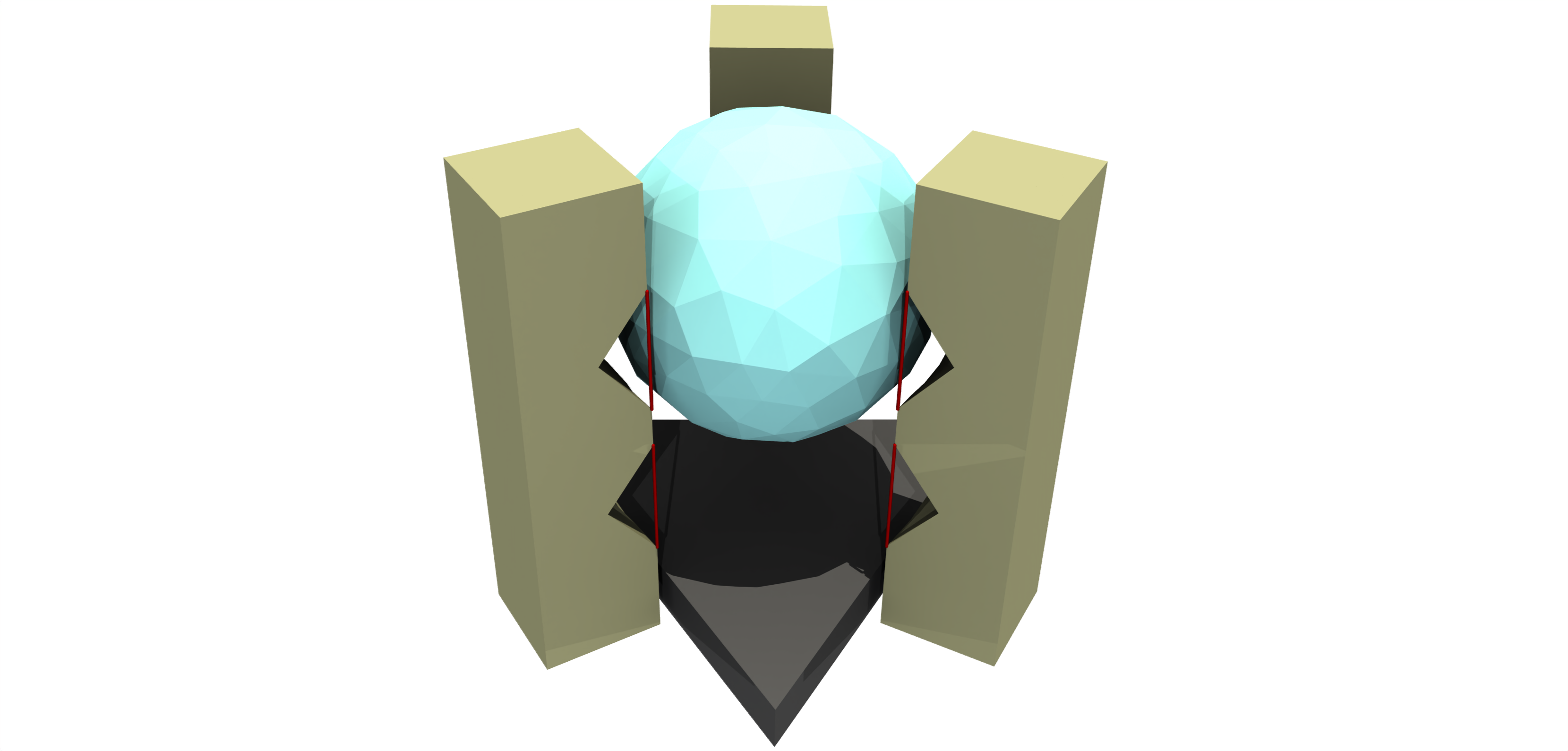}&
\includegraphics[width=.5\linewidth,trim=34cm 0cm 36cm 0cm,clip]{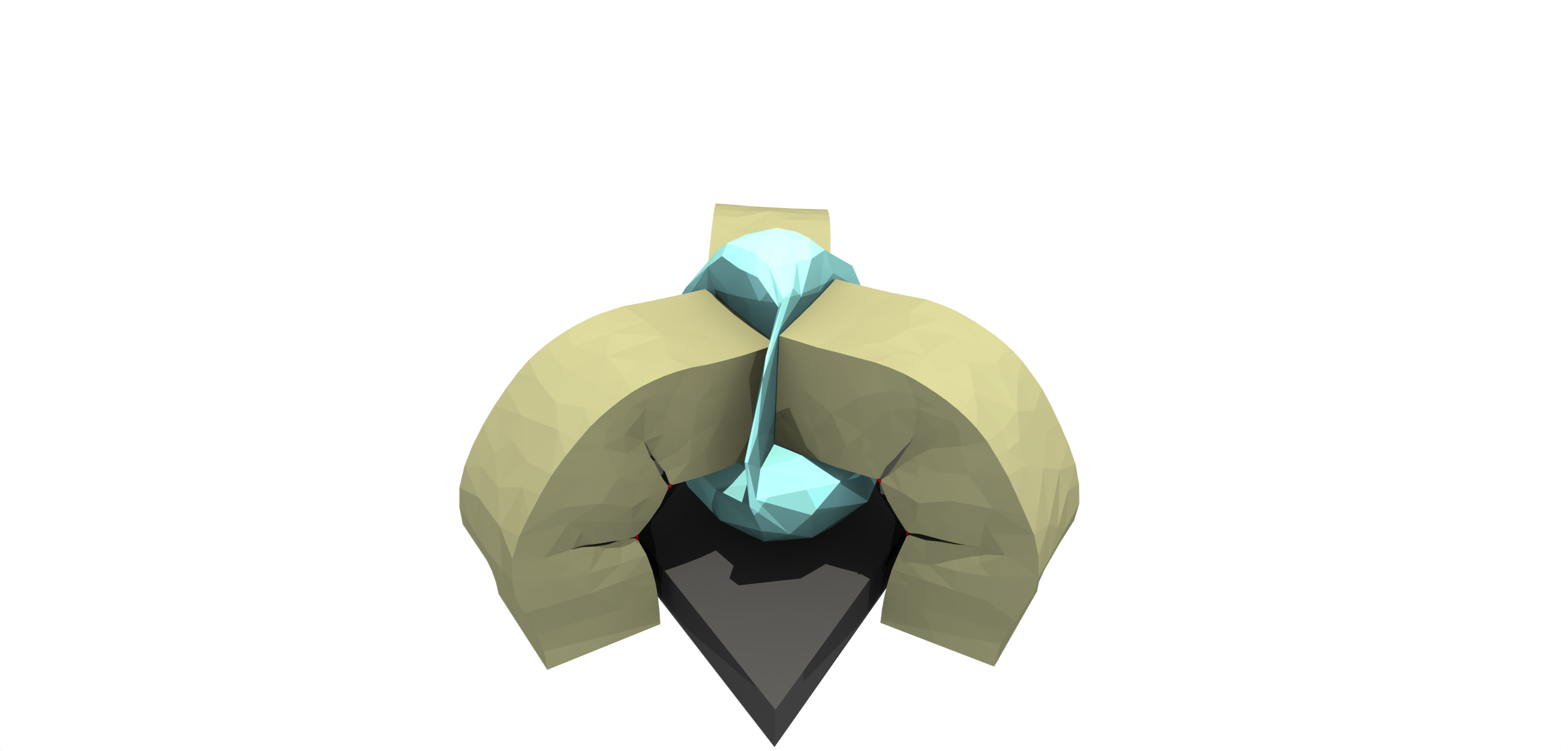}
\end{tabular}
\caption{\label{fig:Finger3D} We simulate a 3-fingered, cable-actuated (red line segment) soft robot gripper (brown) to grasp a soft-ball (blue), the initial and optimized configuration are shown on the left and right, respectively. The ball is much softer than the gripper, so that the gripper can impose extreme forces and cause significant deformations.}
\end{figure}
\begin{figure}[ht]
\centering
\setlength{\tabcolsep}{0px}
\begin{tabular}{cc}
\includegraphics[height=.37\linewidth]{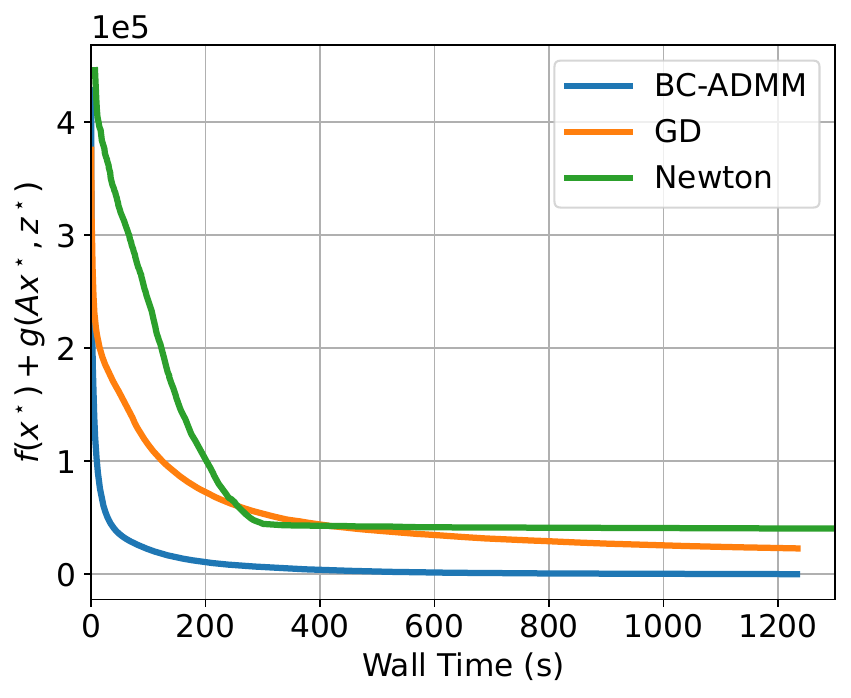}&
\includegraphics[height=.37\linewidth]{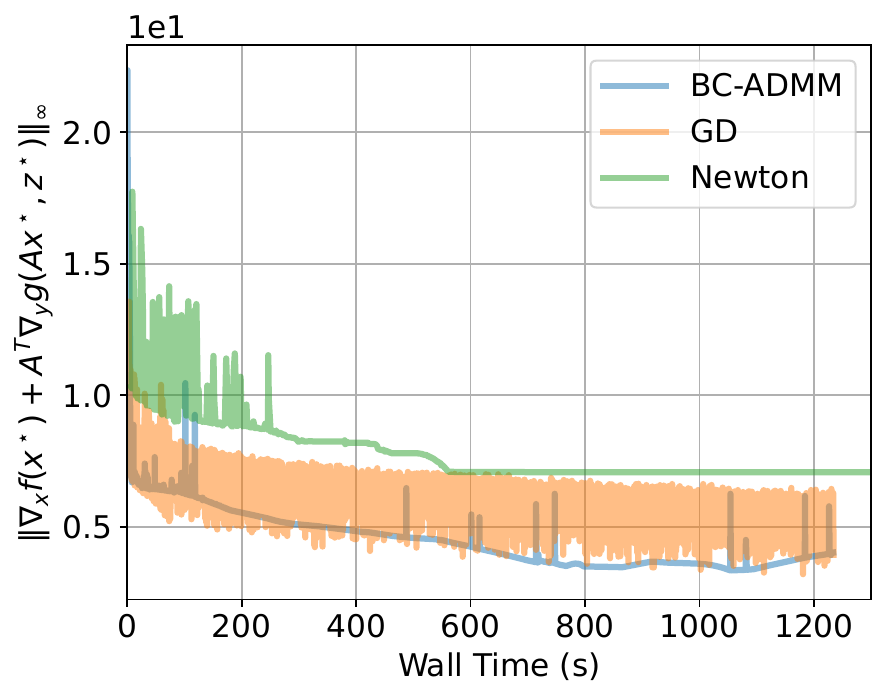}
\end{tabular}
\caption{\label{fig:Finger3D-conv-E} We plot the convergence history of the three algorithms in the soft robot gripper simulation benchmark, where the left figure shows the energy value and the right figure shows the gradient norm. BC-ADMM outperforms the two other algorithms in both criteria.}
\end{figure}

\begin{figure*}[ht]
\centering
\includegraphics[width=.24\linewidth]{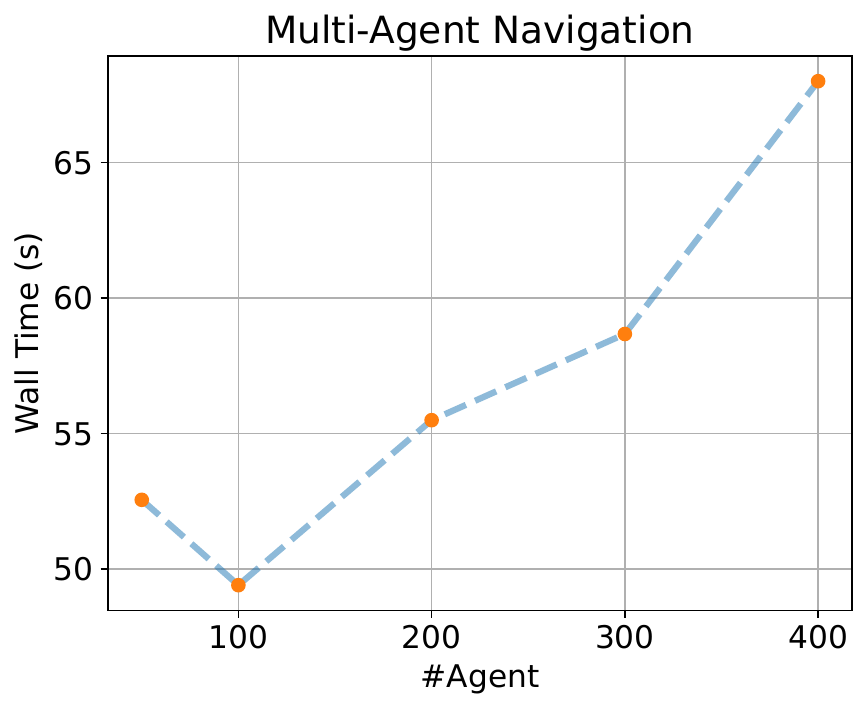}
\includegraphics[width=.24\linewidth]{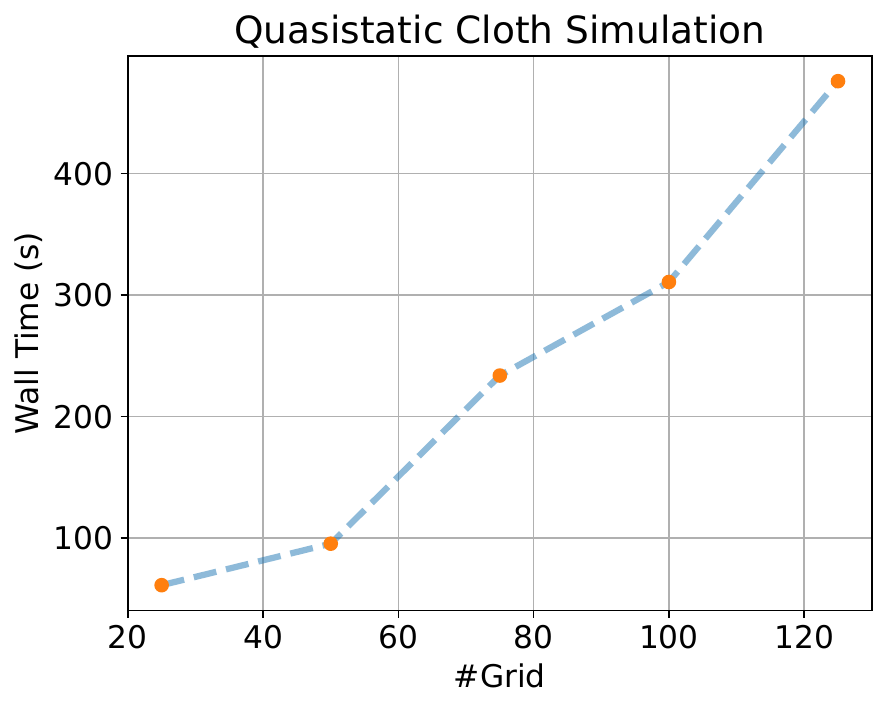}
\includegraphics[width=.24\linewidth]{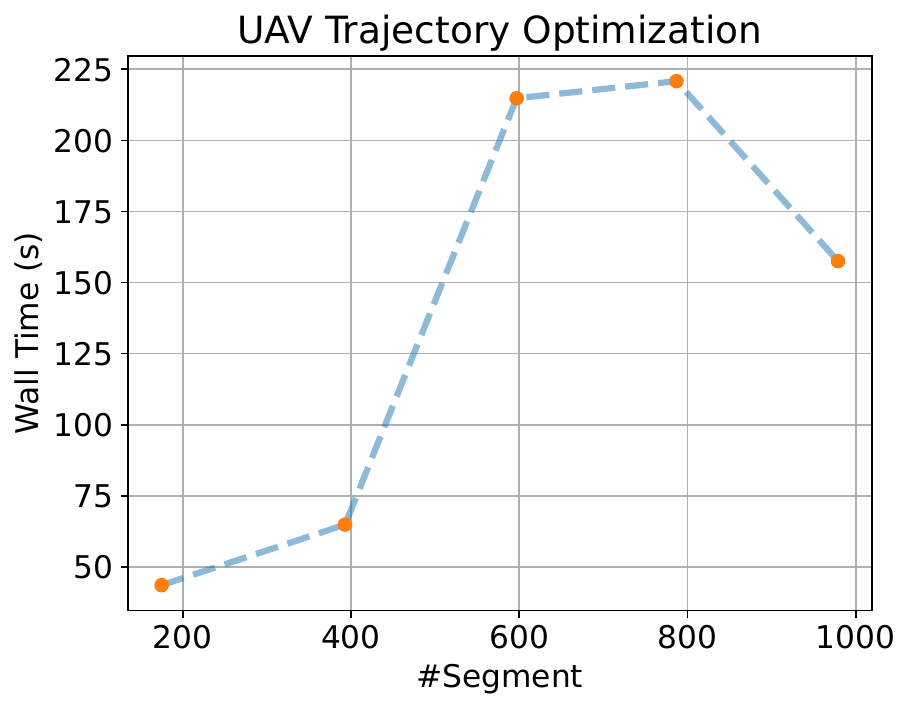}
\includegraphics[width=.24\linewidth]{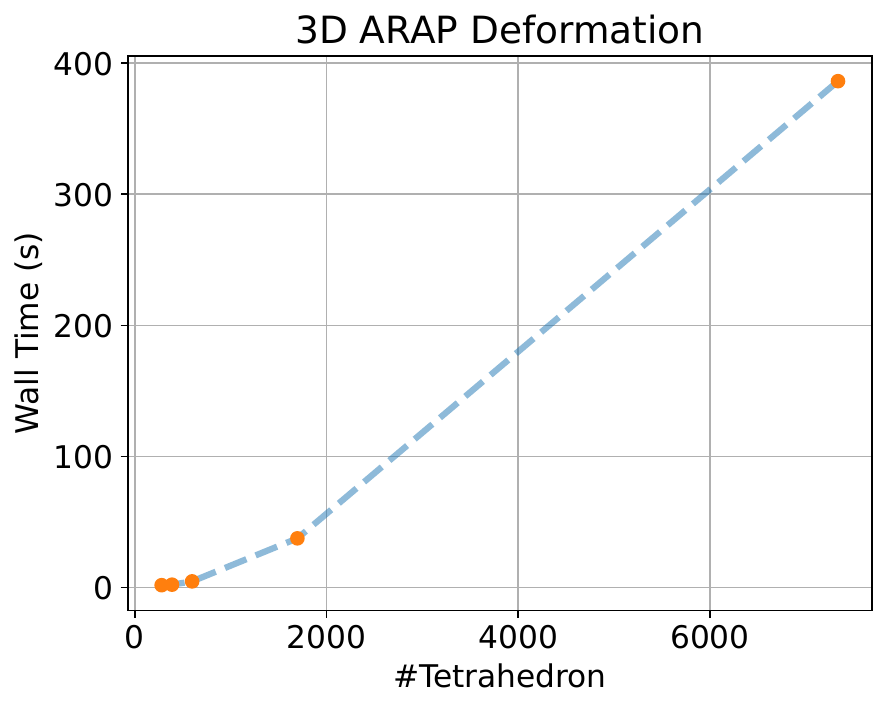}
\caption{\label{fig:scalability} We plot the computational cost as a function of problem size across four tasks. For the multi-agent navigation problem, we use the agent setup shown in~\prettyref{fig:navigation}, where a varying number of robots are initialized in a circular configuration. In the quasistatic cloth simulation task, following the setup in~\prettyref{fig:object-settle}, we vary the cloth size by increasing the number of rectangular grid cells along the vertical dimension. For the UAV trajectory optimization task, based on the configuration in~\prettyref{fig:UAV}, we vary the number of circular obstacles. As the number of obstacles increases, the UAV trajectory is subdivided into a corresponding number of B\`ezier curve segments. Lastly, for ARAP deformation, we adopt the 3D version of the setup in~\prettyref{fig:ARAP-Flip}, where we vary problem size by discretizing the 3D beam using different numbers of tetrahedral elements.}
\end{figure*}
\subsection{Scalability}
Empirically, larger problem instances lead to an accumulation of log-barrier terms, which increases the numerical stiffness of the optimization problem and results in a greater number of iterations to reach convergence. We examine how computational cost scales with problem size across four benchmark tasks. While it is common practice in robotics applications to terminate optimization after a fixed number of iterations, such comparisons only capture per-iteration cost and fail to reflect true convergence behavior. Instead, we evaluate the total computational cost required for convergence. To do so, we first run the optimizer for a large number of iterations to ensure full convergence and obtain a reference solution, denoted by $\TWO{x^{\star\star}}{z^{\star\star}}$. We then define convergence as the point at which the relative objective value is within $1\%$ of this reference:
\begin{align*}
\frac{|[f(x^\star)+g(Ax^\star,z^\star)]-[f(x^{\star\star})+g(Ax^{\star\star},z^{\star\star})]|}{|f(x^{\star\star})+g(Ax^{\star\star},z^{\star\star})|} < 1\%.
\end{align*}
This criterion is sufficient for nearly all practical applications, as we observe no measurable differences in robot behavior or object deformation once the threshold is met. Our findings are summarized in~\prettyref{fig:scalability}. In some tasks, such as multi-agent navigation and UAV trajectory optimization, we even observe a decrease in computational cost with increasing problem size. Nevertheless, the overall scaling trend across all tasks remains approximately linear with respect to problem size.

\begin{figure}[ht]
\centering
\includegraphics[width=.95\linewidth]{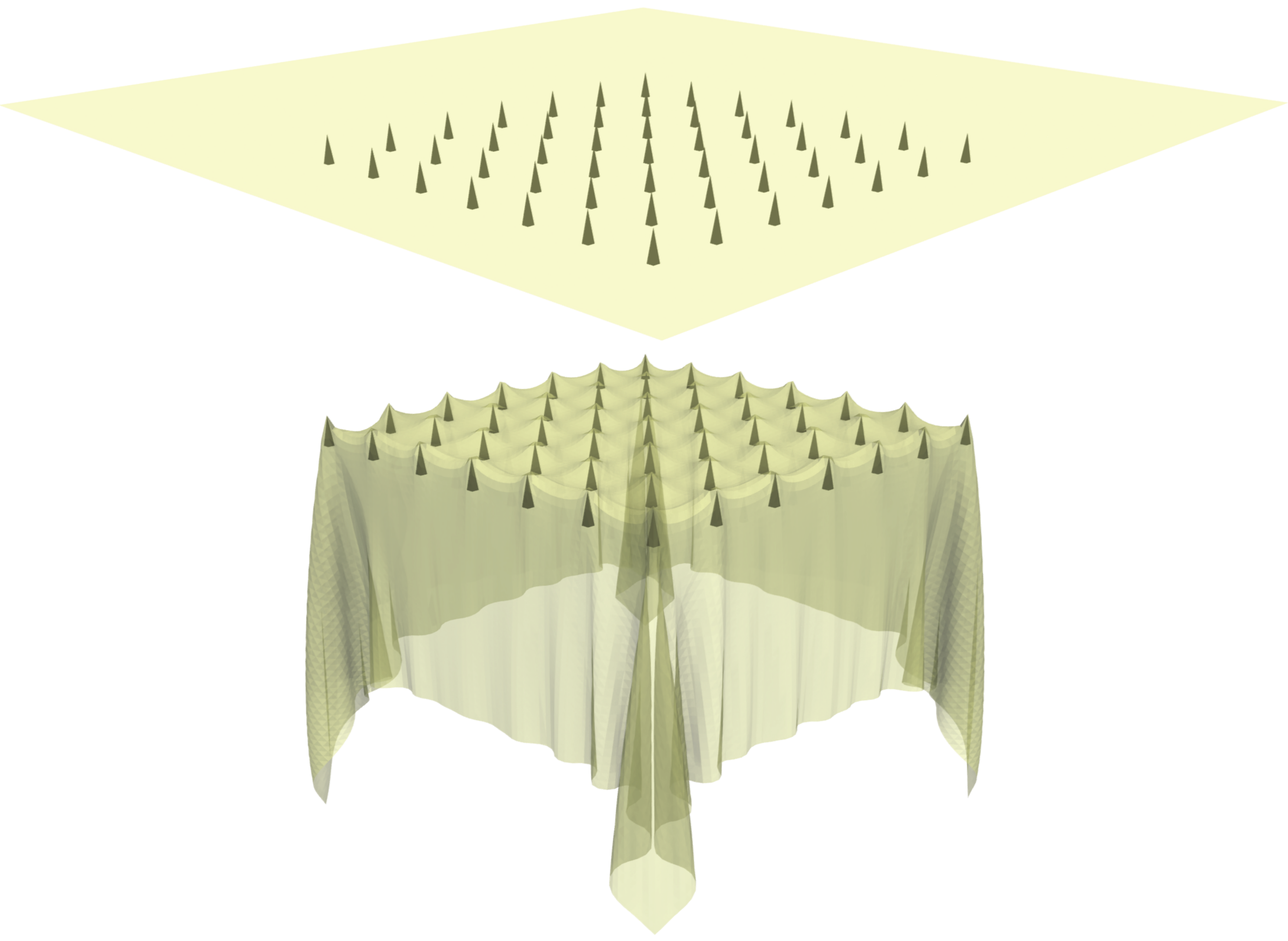}
\caption{\label{fig:3DPin}The initial (top) and final (bottom) iteration during the quasistatic optimization of a piece of cloth dropped onto a grid of needles.}
\end{figure}
\begin{figure*}[ht]
\centering
\begin{tabular}{cccc}
\includegraphics[height=.185\linewidth]{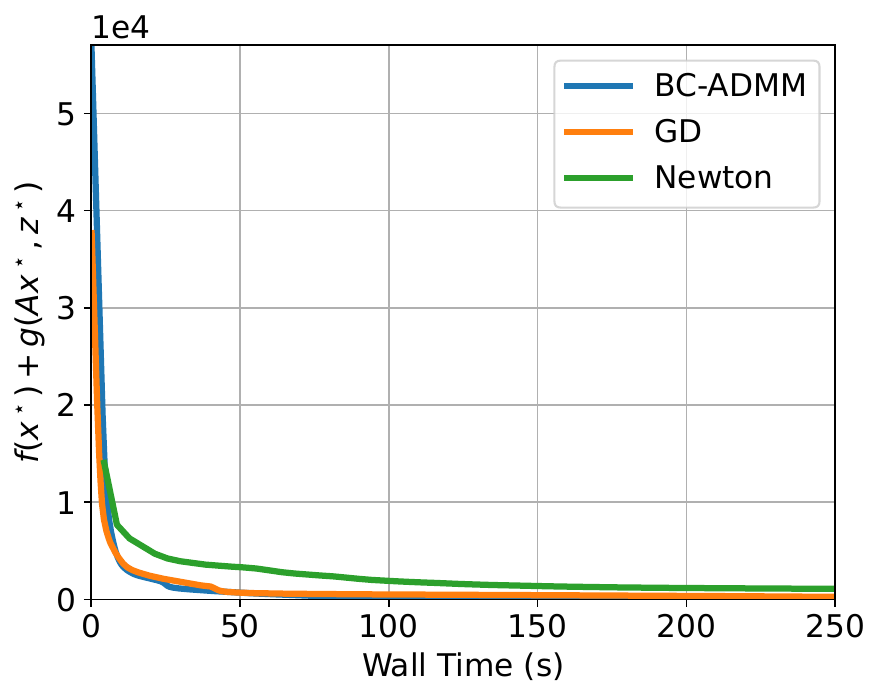}&
\includegraphics[height=.185\linewidth]{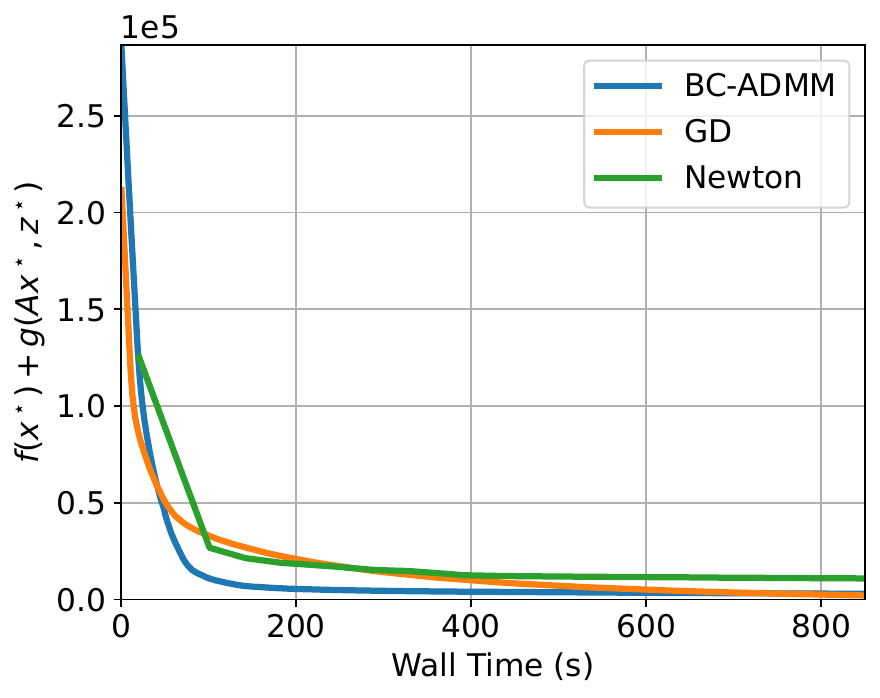}&
\includegraphics[height=.185\linewidth]{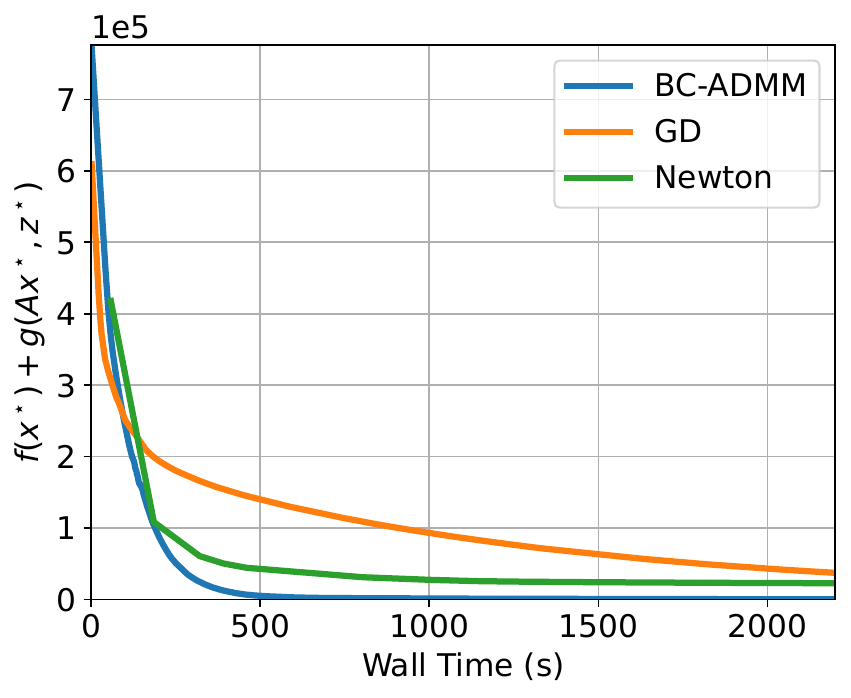}&
\includegraphics[height=.185\linewidth]{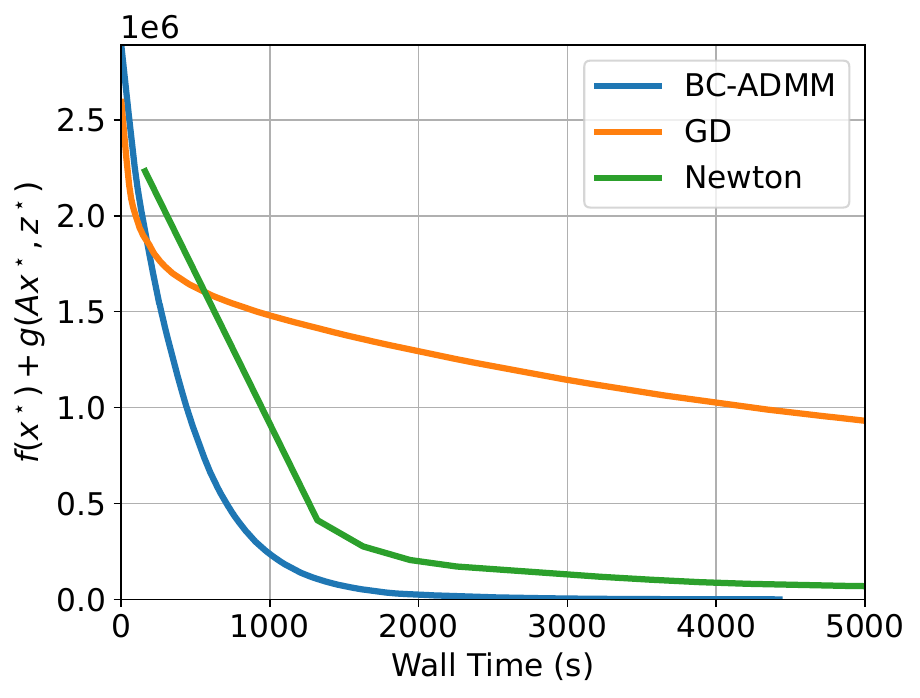}
\end{tabular}
\caption{\label{fig:3DPin-Scalability} The convergence history of the three algorithms in the example of~\prettyref{fig:3DPin}. From left to right, we use a cloth resolution of $30^2$, $60^2$, $90^2$, and $120^2$ vertices. Our results show that the performance advantage of our method becomes more pronounced at higher resolutions.}
\end{figure*}
We further analyzed the scalability of all three methods using the example shown in~\prettyref{fig:3DPin}. In this benchmark, a piece of cloth is dropped onto a grid of needles---a standard scenario for evaluating the robustness and scalability of cloth simulation algorithms. By varying the resolution of the cloth mesh, we profile the convergence behavior of all three methods, as illustrated in~\prettyref{fig:3DPin-Scalability}. Our results show that the performance advantage of our method becomes more pronounced at higher resolutions. This is primarily due to the fact that the per-iteration complexity of BC-ADMM scales approximately linearly with problem size, while the direct matrix factorization required by Newton’s method scales superlinearly. Additionally, the global line-search strategies employed by both Newton’s method and gradient descent tend to become more restrictive as the problem size increases. Finally, we note the inherent challenge in ensuring fair comparisons across varying problem sizes, as different solvers may converge to different local solutions. Furthermore, a rigorous theoretical characterization of the scalability of ADMM-based optimization methods remains an open problem in the field.

\begin{figure}[ht]
\includegraphics[width=.49\linewidth]{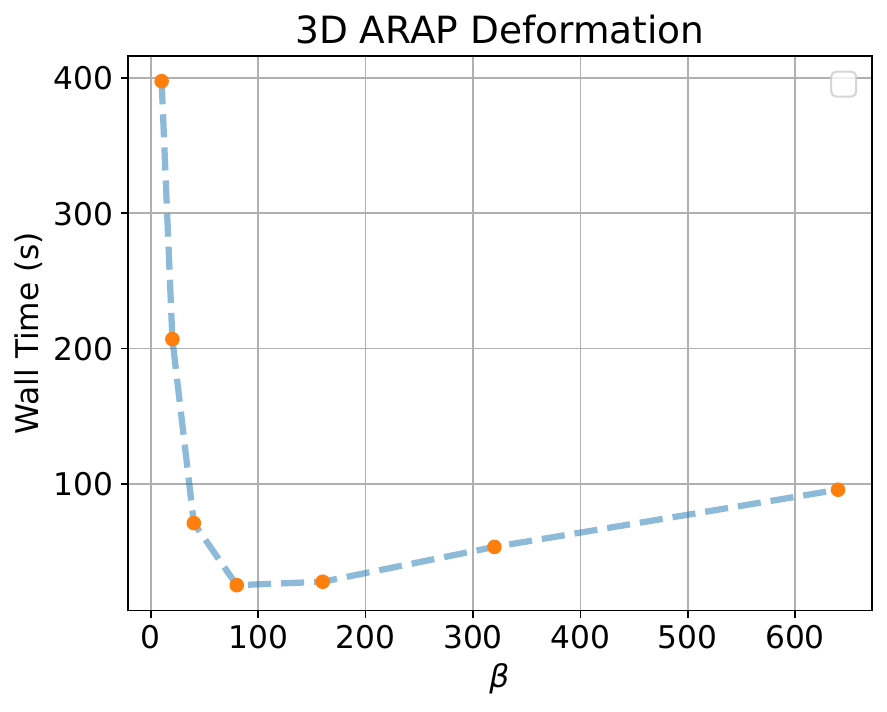}
\includegraphics[width=.49\linewidth]{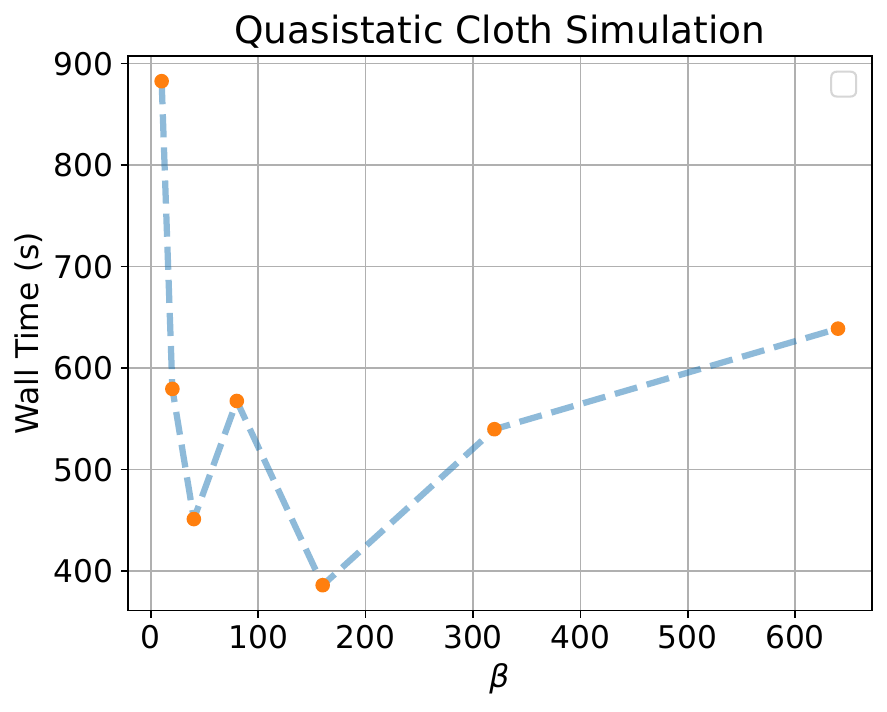}
\caption{\label{fig:beta}We profile the cost to converge under a wide range of parameter choices for two of our benchmarks.}
\end{figure}
\subsection{Parameter Sensitivity}
Our main result in~\prettyref{ass:bivariable-case1} indicates that the optimizer can be sensitive to all three parameters: $\beta_x$, $\beta_y$, and $\beta$. In practice, however, we find that~\prettyref{alg:practical} provides a robust mechanism for automatically tuning $\beta_x$ and $\beta_y$ to ensure convergence. In all experiments, we initialize $\beta_y = 1$ and allow the algorithm to adaptively adjust both $\beta_x$ and $\beta_y$ during optimization. In contrast, the algorithm remains sensitive to the global penalty parameter $\beta$, which requires manual tuning across different benchmarks. This sensitivity is consistent with observations in prior work, such as~\cite{overby2017admmpd}. In~\prettyref{fig:beta}, we compare convergence behavior for two examples under varying values of $\beta$. The results reveal the existence of a sweet spot—where the optimizer converges most efficiently—and demonstrate degraded convergence when $\beta$ is set either too high or too low. For deformable object simulation, we recommend following the empirical guideline proposed in~\cite{overby2017admmpd}, setting $\beta$ to the square root of the average material stiffness, i.e., $\beta = \sqrt{k}$ in $P_{k,l,\underline{l}}$. However, we note that the stiffness of the log-barrier function can vary dynamically depending on the proximity between objects and obstacles. Therefore, for other applications such as multi-agent navigation and UAV trajectory optimization, we simply use a fixed value of $\beta = 100$.
\section{Conclusion}
In this paper, we introduce BC-ADMM, a new optimizer that can solve non-convex constrained optimization problems. Our theoretical analysis shows that BC-ADMM has convergence speed guarantees and practical convergence guarantees in an asymptotic sense. We further show that BC-ADMM can be used in a wide range of robotic applications, including multi-agent navigation, UAV trajectory optimization, deformable object simulation, and soft robot simulation. Our numerical experiments show that, when solving large-scale problem instances, BC-ADMM outperforms both gradient descent and Newton's method in terms of wall clock time. However, our formulation still suffers from several limitations. First, our theoretical analysis does not guarantee a convergence speed of the practical~\prettyref{alg:practical}. Instead, our convergence speed guarantee requires sufficiently large parameters that are hard to estimate. Second, our method heavily relies on a bi-convex relaxation of non-convex constraints, which potentially limits its range of applications in other robotic problems. An important class of robotic applications is the optimization under dynamic constraints, which has been investigated in other structured optimizers~\cite{stella2017simple,8814732}. Incorporating dynamic constraints in BC-ADMM poses a promising direction of future research. Finally, our practical~\prettyref{alg:practical} requires a centralized check to determine whether variables can be updated, which prevents BC-ADMM from solving very large-scale problems in a decentralized manner. We hope these limitations inspire future research to improve our techniques.

\bibliographystyle{SageH}
\bibliography{references}
\clearpage
\iflong
\section{\label{sec:convergence}Convergence Analysis}
We need a novel technique to analyze the convergence of BC-ADMM. Our main idea is to construct an Lyapunov candidate and establish its descendant property. However, such analysis requires the function $g$ to have Lipschitz continuous gradient. To work around this obstacle, we propose to use Whitney extension theorem~\cite{whitney1992analytic} and construct a function $G$ with Lipschitz-continuous gradient. We then consider BC-ADMM with $g$ replaced by $G$ and denote the algorithm as BC-ADMM$_G$ (\prettyref{sec:BC-ADMM-G}) and establish its convergence by Lyapunov analysis (\prettyref{sec:Lyapunov}). Finally, we show that the two algorithms, BC-ADMM and BC-ADMM$_G$, generate the same sequence (\prettyref{sec:connection}).

\subsection{\label{sec:BC-ADMM-G}BC-ADMM$_G$}
To construct the modified algorithm BC-ADMM$_G$, we introduce the following result to prove the existence of $G$. Specifically, for a function twice differentiable on a compact set of $\mathbb{R}^{m+l}$, we can extend it to a function that is twice differentiable in the entire $\mathbb{R}^{m+l}$. Further, we can slightly modify this function to equip it with Lipschitz continuous gradient. Let us define the set:
\begin{align*}
\mathcal{S}(r,M)&=\{\TWO{y}{z}|g(y,z)\leq M\land\|\TWO{y}{z}\|\leq r\}\\
\mathcal{S}_\mathcal{Z}(r,M)&=\{\TWO{y}{z}|g(y,z)\leq M\land\|\TWO{y}{z}\|\leq r\land z\in\mathcal{Z}\}.
\end{align*}
We aim to prove a theorem that is slightly more general theorem based on a weaker assumption on the function $g$ as follows, which allows us to solve more problem instances:
\begin{assume}
\label{ass:bivariable}
i) $f:\mathbb{R}^n\to\mathbb{R}^+\cup\{\infty\}$ is continuous and differentiable in $\dom(f)$ and $g:\mathbb{R}^{m+l}\to\mathbb{R}^+\cup\{\infty\}$ is continuous and twice continuously differentiable in $\dom(g)$. ii) $f(\bullet)+g(A\bullet,\bullet)$ is coercive. iii) A strictly feasible solution $x^0\in\mathbb{R}^n,z^0\in\mathcal{Z}$ is given, such that $f(x^0)+g(Ax^0,z^0)<\infty$. iv) $g(\bullet,z)$ is convex. vi) For $\TWO{y}{z}\in\mathcal{S}_\mathcal{Z}(r,M)$, $g(y,\bullet)+\iota_\mathcal{Z}(\bullet)$ has a unique global minimizer on the closed set $\mathcal{Z}$ denoted as a function $z(y)$, such that $\TWO{y}{z(y)}\in\mathcal{S}_\mathcal{Z}(\bar{r}_{zy}^{r,M},M)$. Further for $\TWO{y}{z(y)}\in\mathcal{S}_\mathcal{Z}(r,M)$, $z(y)$ is an $L_{zy}^{r,M}$-Lipschitz continuous function.
\end{assume}
We first present a set of preparatory results:
\begin{lemma}
\label{lem:Lg}
We take~\prettyref{ass:bivariable}. Then for any $r>0$ and $M>0$, $g$ has the Lipschitz continuous gradient on $\mathcal{S}(r,M)$ with a modulus $L_g^{r,M}<\infty$.
\end{lemma}
\begin{proof}
For each $\TWO{y}{z}\in\mathcal{S}(r,M)$, we can find an open neighborhood $\mathcal{G}_{y,z}(\rad(y,z))$ with radius $\rad(y,z)$ such that $\|\nabla^2g(y',z')\|\leq L(y,z)<\infty$ for $\TWO{y'}{z'}\in\mathcal{G}_{y,z}(\rad(y,z))$. Now since $\mathcal{S}(r,M)$ is compact, we can find a finite sub-cover and we choose $L_g^{r,M}$ to be the maximum over these finite modulus $L(y,z)$.
\end{proof}
\begin{lemma}
\label{lem:margin}
We take~\prettyref{ass:bivariable}. Then for any $r>0$ and $N>M>0$, there exists some $\delta^{r,M,N}>0$, such that for any $\TWO{y}{z}\in\mathcal{S}(r,M)$, $\mathcal{G}_{y,z}(\delta^{r,M,N})\subseteq\mathcal{S}(r+\delta^{r,M,N},N)$.
\end{lemma}
\begin{proof}
Consider another set $\mathcal{S}'(r+1,N)=\{g(y,z)\geq N\land\|\TWO{y}{z}\|\leq r+1\}$. Clearly, $\mathcal{S}(r,M)$ and $\mathcal{S}'(r+1,N)$ are both compact and disjoint because $g$ is continuous, so the distance between them is a positive value denoted as $\delta^{r,M,N,\star}>1$. We can show that $\delta^{r,M,N}=\min(\delta^{r,M,N,\star},1)$ satisfies our needs.
\end{proof}
\begin{lemma}
\label{lem:zy}
We take~\prettyref{ass:bivariable}. Suppose $\TWO{y}{z}\in\mathcal{S}_\mathcal{Z}(r,M)$ and $\|y'-y\|\leq\delta^{r,M,N}$, then $\|z(y)-z(y')\|\leq L_{zy}^{r,M,N}\|y'-y\|$ for some $L_{zy}^{r,M,N}<\infty$.
\end{lemma}
\begin{proof}
For any point $y''$ on the line segment between $y$ and $y'$, we have $\TWO{y''}{z}\in\mathcal{S}_\mathcal{Z}(r+\delta^{r,M,N},N)$ by~\prettyref{lem:margin}, so by our~\prettyref{ass:bivariable} vi), $\TWO{y''}{z(y'')}\in\mathcal{S}_\mathcal{Z}(\bar{\bar{r}}_{zy}^{r,M,N},M)$ where we define $\bar{\bar{r}}_{zy}^{r,M,N}=\bar{r}_{zy}^{r+\delta^{r,M,N},N}$. Further by our~\prettyref{ass:bivariable} vi), $z(y'')$ restricted to this set is Lipschitz continuous with a modulus of $L_{zy}^{\bar{\bar{r}}_{zy}^{r,M,N},M}$ re-defined as $L_{zy}^{r,M,N}$, which is our desired result.
\end{proof}
\begin{lemma}
\label{lem:extension}
We take~\prettyref{ass:bivariable}. Then for any $r,M>0$ such that $\mathcal{S}(r,M)$ is not empty, we can define a twice-differentiable real function $G_{r,M}:\mathbb{R}^{m+l}\to\mathbb{R}$ such that: i) $G_{r,M}$ and $g$ match up to second derivatives on $\mathcal{S}(r,M)$; ii) $G_{r,M}=0$ when $\|\TWO{y}{z}\|\geq 2r$; iii) $G_{r,M}$ has $L_G^{r,M}$-Lipschitz continuous gradient with $L_g^{r,M}\leq L_G^{r,M}<\infty$; iv) $G_{r,M}\geq\underline{G}_r>-\infty$.
\end{lemma}
\begin{proof}
Clearly $\mathcal{S}(r,M)$ is compact, so the Whitney extension theorem results in a twice-differentiable function $G_{r,M}^\star$ such that $G_{r,M}^\star$ and $g$ match up to second derivatives on $\mathcal{S}(r,M)$. We then define a smooth decay function:
\begin{align*}
\eta_r(d)=
\begin{cases}
1&d\leq r\\
{(2r-x)^3(6x^2-9rx+4r^2)}/{r^5}&d\in(r,2r)\\
0&d\geq 2r\\
\end{cases}.
\end{align*}
Finally, we can satisfy all the conditions by defining: $G_{r,M}(y,z)=\eta_r(\|y,z\|)G_{r,M}^\star(y,z)$. Indeed, i) and ii) are trivial. To establish iii), note that $G_{r,M}$ has $L_g^{r,M}$-Lipschitz continuous gradient on $\mathcal{S}(r,M)$ with $L_g^{r,M}<\infty$ by~\prettyref{lem:Lg}, $G_{r,M}$ has $0$-Lipschitz continuous gradient when $\|\TWO{y}{z}\|\geq2r$. The remaining area is bounded. Similarly, to establish iv), note that $G_{r,M}=g\geq0$ on $\mathcal{S}(r,M)$, $G_{r,M}=0$ when $\|\TWO{y}{z}\|\geq2r$, and the remaining area is bounded.
\end{proof}
The main benefit of analyzing the extended function $G_{r,M}$ instead of $g$ lies in the fact that $G_{r,M}$ has $L_G^{r,M}$-Lipschitz continuous gradient, so that we have the following convenient result in the entire Euclidean space:
\begin{equation}
\begin{aligned}
\label{eq:Lipschitz-a}
&\|\nabla G_{r,M}(y,z)-\nabla G_{r,M}(y',z')\|\\
\leq&L_G^{r,M}\|\TWO{y}{z}-\TWO{y'}{z'}\|,
\end{aligned}
\end{equation}
and:
\begin{equation}
\begin{aligned}
\label{eq:Lipschitz-b}
&G_{r,M}(y',z')\\
\leq&G_{r,M}(y,z)+\nabla G_{r,M}(y,z)^T(\TWO{y'}{z'}-\TWO{y}{z})+\\
&\frac{L_G^{r,M}}{2}\|\TWO{y}{z}-\TWO{y'}{z'}\|^2.
\end{aligned}
\end{equation}
We are now ready to define our modified BC-ADMM$_G$ algorithm that aims to minimize the objective with $g$ replaced by $G_{r,M}$. First, since $f(\bullet)+g(A\bullet,\bullet)$ is coercive by~\prettyref{ass:bivariable}, we must have that all the $\TWO{x}{z}$ satisfying $f(x)+g(Ax,z)\leq M$ is bounded in the hyperball with radius $R^M$, i.e., $\|\TWO{x}{z}\|\leq R^M$. With this notation, we assume the following parameter choice:
\begin{assume}
\label{ass:param}
The following parameters are used for BC-ADMM$_G$:
\footnotesize
\begin{subequations}
\begin{align}[left={\empheqlbrace}]
\label{eq:param-a}
&M=5(f(x^0)+g(Ax^0))\\
\begin{split}
&r^1=R^{2M/5}, r^2=r^1+\delta^{r^1,2M/5,3M/5}\\
&r^3=r^2+\delta^{r^2,3M/5,4M/5}, r^4=r^3+L_{zy}^{r^2,3M/5,4M/5}\delta\\
\end{split}\\
\label{eq:param-a2}
&\delta=\min(\delta^{r^1,2M/5,3M/5},\delta^{r^2,3M/5,4M/5},\delta^{r^4,4M/5,M})\\
\label{eq:param-a3}
&r=r^4+\delta^{r^4,4M/5,M}\\
\label{eq:param-b}
&\beta_x\geq\frac{2\|A\|\|\nabla f(x^0)+A^T\nabla_yg(Ax^0,z^0)\|}{\delta}+L_f\\
\label{eq:param-b2}
&\beta_x\geq\frac{4(2M/5-\underline{G}_{r,M})\|A\|^2}{\delta^2}\\
\label{eq:param-b3}
&\beta_y\geq\frac{4(2M/5-\underline{G}_{r,M})}{\delta^2}\\
\label{eq:param-c}
&\beta\geq2L_G^{r,M}\\
\label{eq:param-c2}
&\beta\geq\frac{32\beta_y^2+16\beta_y^2(L_G^{r,M}L_{zy}^{r^2,3M/5,4M/5})^2+16(L_G^{r,M})^2}{\beta_y}\\
\label{eq:param-c3}
&\beta\geq\frac{2(\beta_y-L_G^{r,M})}{\beta_y+L_G^{r,M}-2}\\
\label{eq:param-c4}
&\beta\geq\frac{\delta^2}{f(x^0)+g(Ax^0,z^0)}\\
\label{eq:param-c5}
\begin{split}
&\beta\left[1-\frac{(\beta_y+L_G^{r,M}L_{zy}^{r^2,3M/5,4M/5})^2}{8\beta_y^2(1+(L_G^{r,M}L_{zy}^{r^2,3M/5,4M/5})^2)}\right]\geq\\
&L_G^{r,M}+\frac{2(2M/5-\underline{G}_{r,M})}{\delta^2}
\end{split}
\end{align}
\end{subequations}
\normalsize
\end{assume}
It is easy to see that the parameters in~\prettyref{ass:param} can be satisfied by choosing $M,\delta,r,\beta_x,\beta_y,\beta$ in order to satisfy all the conditions. Next, we introduce the following Lagrangian function:
\small
\begin{equation}
\begin{aligned}
&\mathcal{L}_{r,M}(x,y,z,\lambda)\\
=&f(x)+G_{r,M}(y,z)+\frac{\beta}{2}\|Ax-y\|^2+\lambda^T(Ax-y).
\end{aligned}
\end{equation}
\normalsize
Our BC-ADMM$_G$ proceeds to update $x,y,z,\lambda$ iteratively by the following rule, with superscript denoting the iteration number:
\footnotesize
\begin{subequations}
\begin{align}[left={\rotatebox[origin=c]{90}{$\text{BC-ADMM}_G$} \empheqlbrace}]
y^0=&Ax^0\quad
z^0=z(y^0)\quad
\lambda^0=\nabla_yg(y^0,z^0)\\
\label{eq:step-a-G}
x^{k+1}=&\argmin{x}\mathcal{L}_{r,M}(x,y^k,z^k,\lambda^k)+\frac{\beta_x}{2}\|x-x^k\|^2\\
\label{eq:step-b-G}
y^{k+1}=&\argmin{y}\mathcal{L}_{r,M}(x^{k+1},y,z^k,\lambda^k)+\frac{\beta_y}{2}\|y-y^k\|^2\\
\label{eq:step-c-G}
z^{k+1}=&\CLIP(z(y^{k+1}),z^k,L_{zy}^{r^2,3M/5,4M/5}\|y^{k+1}-y^k\|)\\
\label{eq:step-d-G}
\lambda^{k+1}=&\lambda^k+\beta(Ax^{k+1}-y^{k+1}).
\end{align}
\end{subequations}
\normalsize
We make two changes in BC-ADMM$_G$. First, we cannot compute the globally optimal solution in~\prettyref{eq:step-b-G} because $G_{r,M}(\bullet,z)$ is not convex in general, so we only assume that a locally optimal solution is computed starting from the initial guess of $y^k$. Second, note that we cannot reuse equation~\prettyref{eq:step-c} with $g$ replaced by $G_{r,M}$ because $G_{r,M}(y,\bullet)$ is no longer convex in general. Therefore, we introduce a clip function defined as:
\begin{align*}
&\CLIP(z(y^{k+1}),z^k,l)=\\
&\begin{cases}
z^k&\text{otherwise}\\
z(y^{k+1})&\text{if }\begin{cases}
\|z(y^{k+1})-z^k\|\leq l \text{ and}\\
G_{r,M}(y^{k+1},z(y^{k+1}))\leq
G_{r,M}(y^{k+1},z^k)
\end{cases}.
\end{cases}
\end{align*}
Note that, unlike BC-ADMM, we cannot use mathematical induction to show that $g(y^{k+1},\bullet)$ is proper in BC-ADMM$_G$, so we have to consider the case where $z(y^{k+1})$ does not exist and our $\CLIP$ function simply set $z^{k+1}=z^k$ in this case. Finally, similar to BC-ADMM, we assume that~\prettyref{eq:step-a-G} and~\prettyref{eq:step-b-G} is solved by local gradient descend starting from the initial guess of $x^k$. By the first-order optimality condition and the definition of $\CLIP$, we have at every iteration:

\tiny
\begin{align}
\label{eq:optimality-a-G}
0=&\beta_x(x^{k+1}-x^k)+\nabla f(x^{k+1})+\beta A^T(Ax^{k+1}-y^k)+A^T\lambda^k\\
\label{eq:optimality-a2-G}
=&\beta_x(x^{k+1}-x^k)+\nabla f(x^{k+1})+A^T\lambda^{k+1}+\beta A^T(y^{k+1}-y^k)\\
0
\label{eq:optimality-b-G}
=&\beta_y(y^{k+1}-y^k)+\nabla_yG_{r,M}(y^{k+1},z^k)-\lambda^{k+1}\\
\label{eq:CLIP}
l\geq&\|\CLIP(a,b,l)-b\|.
\end{align}
\normalsize

In the next section, we would construct Lyapunov candidate for BC-ADMM$_G$ and establish its descendent property.
\subsection{\label{sec:Lyapunov}Lyapunov Analysis of FT-ADMM$_G$}
In this section, we analyze the convergence by constructing a Lyapunov candidate in a similar fashion as~\cite{jiang2019structured}. We start from a convenient lemma:
\begin{lemma}
\label{lem:lambda}
Taking~\prettyref{ass:bivariable}, the update of $\lambda^k$ in FT-ADMM$_G$ satisfies the following inequality:
\begin{equation}
\begin{aligned}
&\|\lambda^{k+1}-\lambda^k\|^2\leq
4(\beta_y^2+(L_G^{r,M})^2)\|y^{k+1}-y^k\|^2+\\
&4\beta_y^2(1+(L_G^{r,M}L_{zy}^{r^2,3M/5,4M/5})^2)\|y^k-y^{k-1}\|^2.
\end{aligned}
\end{equation}
\end{lemma}
\begin{proof}
This result is derived by the following inequality:
\begin{align*}
&\|\lambda^{k+1}-\lambda^k\|^2\\
\numleq{1}&4\|\beta_y(y^{k+1}-y^k)\|^2+
4\|\beta_y(y^k-y^{k-1})\|^2+\\
&4\|\nabla_yG_{r,M}(y^{k+1},z^k)-\nabla_yG_{r,M}(y^k,z^k)\|^2+\\
&4\|\nabla_yG_{r,M}(y^k,z^k)-\nabla_yG_{r,M}(y^k,z^{k-1})\|^2\\
\numleq{2}&4(\beta_y^2+(L_G^{r,M})^2)\|y^{k+1}-y^k\|^2+4\beta_y^2\|y^k-y^{k-1}\|^2+\\
&4\beta_y^2(L_G^{r,M})^2\|z^k-z^{k-1}\|^2\\
\numleq{3}&4(\beta_y^2+(L_G^{r,M})^2)\|y^{k+1}-y^k\|^2+4\beta_y^2\|y^k-y^{k-1}\|^2+\\
&4\beta_y^2(L_G^{r,M}L_{zy}^{r^2,3M/5,4M/5})^2\|y^k-y^{k-1}\|^2.
\end{align*}
Here (1) is derived using~\prettyref{eq:optimality-b-G}, (2) is due to~\prettyref{lem:extension}, and (3) is due to~\prettyref{lem:zy} and~\prettyref{eq:CLIP}.
\end{proof}
Next, we propose the following Lyapunov function:
\begin{align*}
&\Psi_{r,M}(x^k,y^k,z^k,\lambda^k,y^{k-1})=\mathcal{L}_{r,M}(x^k,y^k,z^k,\lambda^k)+\\
&\frac{4\beta_y^2(1+(L_G^{r,M}L_{zy}^{r^2,3M/5,4M/5})^2)}{\beta}\|y^k-y^{k-1}\|^2,
\end{align*}
and establish its monotonically decreasing property:
\begin{lemma}
\label{lem:Lyapunov}
Taking~\prettyref{ass:bivariable} and~\prettyref{eq:param-c2}, then we have for FT-ADMM$_G$:
\begin{equation}
\begin{aligned}
\label{eq:delta}
&\frac{\beta_x}{4}\|x^{k+1}-x^k\|^2+\frac{\beta_y}{4}\|y^{k+1}-y^k\|^2\\
\leq&\Psi_{r,M}(x^k,y^k,z^k,\lambda^k,y^{k-1})-\\
&\Psi_{r,M}(x^{k+1},y^{k+1},z^{k+1},\lambda^{k+1},y^k).
\end{aligned}
\end{equation}
\end{lemma}
\begin{proof}
From~\prettyref{eq:step-a-G} and~\prettyref{eq:step-b-G}, we have:
\begin{equation}
\begin{aligned}
\label{eq:xy-optimality-G}
&\mathcal{L}_{r,M}(x^{k+1},y^{k+1},z^k,\lambda^k)+\\
&\frac{\beta_y}{2}\|y^{k+1}-y^k\|^2+\frac{\beta_x}{2}\|x^{k+1}-x^k\|^2
\\\leq&\mathcal{L}_{r,M}(x^{k+1},y^k,z^k\lambda^k)+\frac{\beta_x}{2}\|x^{k+1}-x^k\|^2\\
\leq&\mathcal{L}_{r,M}(x^k,y^k,z^k,\lambda^k).
\end{aligned}
\end{equation}
Now from~\prettyref{eq:xy-optimality-G}, ~\prettyref{eq:step-d-G}, and~\prettyref{lem:lambda}, we have:
\begin{align*}
&\mathcal{L}_{r,M}(x^{k+1},y^{k+1},z^{k+1},\lambda^{k+1})\\
\numleq{1}&\mathcal{L}_{r,M}(x^{k+1},y^{k+1},z^k,\lambda^{k+1})\\
=&\mathcal{L}_{r,M}(x^{k+1},y^{k+1},z^k,\lambda^k)+\\
&(\lambda^{k+1}-\lambda^k)^T(Ax^{k+1}-y^{k+1})\\
=&\mathcal{L}_{r,M}(x^{k+1},y^{k+1},z^k,\lambda^k)+\frac{1}{\beta}\|\lambda^{k+1}-\lambda^k\|^2\\
\numleq{2}&\mathcal{L}_{r,M}(x^{k+1},y^{k+1},z^k,\lambda^k)+\\
&\frac{4(\beta_y^2+(L_G^{r,M})^2)}{\beta}\|y^{k+1}-y^k\|^2+\\
&\frac{4\beta_y^2(1+(L_G^{r,M}L_{zy}^{r^2,3M/5,4M/5})^2)}{\beta}\|y^k-y^{k-1}\|^2\\
\numleq{3}&\mathcal{L}_{r,M}(x^k,y^k,z^k,\lambda^k)+\\
&\left[\frac{4(\beta_y^2+(L_G^{r,M})^2)}{\beta}-\frac{\beta_y}{2}\right]\|y^{k+1}-y^k\|^2+\\
&\frac{4\beta_y^2(1+(L_G^{r,M}L_{zy}^{r^2,3M/5,4M/5})^2)}{\beta}\|y^k-y^{k-1}\|^2-\\
&\frac{\beta_x}{2}\|x^{k+1}-x^k\|^2,
\end{align*}
where, specifically, (1) is due to the definition of $\CLIP$, (2) is due to~\prettyref{lem:lambda}, and (3) is due to~\prettyref{eq:xy-optimality-G}. The result follows by our choice of parameter:
\begin{align*}
&\frac{4(\beta_y^2+(L_G^{r,M})^2)}{\beta}-\frac{\beta_y}{2}\\
\leq&-\frac{4\beta_y^2(1+(L_G^{r,M}L_{zy}^{r^2,3M/5,4M/5})^2)}{\beta}-\frac{\beta_y}{4}\\
\Rightarrow&\beta\geq\frac{32\beta_y^2+16\beta_y^2(L_G^{r,M}L_{zy}^{r^2,3M/5,4M/5})^2+16(L_G^{r,M})^2}{\beta_y},
\end{align*}
and all is proved.
\end{proof}
We can further bound our proposed Lyapunov candidate $\Psi_{r,M}$ from below as follows:
\begin{lemma}
\label{lem:LowerBound}
Taking~\prettyref{ass:bivariable} and~\prettyref{eq:param-c5}, then we have for FT-ADMM$_G$:
\begin{align*}
&\Psi_{r,M}(x^k,y^k,z^k,\lambda^k,y^{k-1})\\
\geq&f(x^k)+G_{r,M}(Ax^k,z^k)+\\
&\left[\frac{\beta-L_{G}^{r,M}}{2}-\frac{\beta(\beta_y+L_G^{r,M}L_{zy}^{r^2,3M/5,4M/5})^2}{16\beta_y^2(1+(L_G^{r,M}L_{zy}^{r^2,3M/5,4M/5})^2)}\right]\\
&\|Ax^k-y^k\|^2\\
\geq&f(x^k)+G_{r,M}(Ax^k,z^k)
\geq\underline{G}_{r,M}.
\end{align*}
\end{lemma}
\begin{proof}
Our desired result then follows due to the following inequality:
\begin{align*}
&\mathcal{L}_{r,M}(x^k,y^k,z^k,\lambda^k)\\
=&f(x^k)+G_{r,M}(y^k,z^k)+\\
&\frac{\beta}{2}\|Ax^k-y^k\|^2+(\lambda^k)^T(Ax^k-y^k)\\
\numeq{1}&f(x^k)+G_{r,M}(y^k,z^k)+\frac{\beta}{2}\|Ax^k-y^k\|^2+\\
&(\beta_y(y^k-y^{k-1})+\nabla_yG_{r,M}(y^k,z^{k-1}))^T(Ax^k-y^k)\\
=&f(x^k)+G_{r,M}(y^k,z^k)+\frac{\beta}{2}\|Ax^k-y^k\|^2+\\
&(\beta_y(y^k-y^{k-1})+\nabla_yG_{r,M}(y^k,z^k))^T(Ax^k-y^k)+\\
&(\nabla_yG_{r,M}(y^k,z^{k-1})-\nabla_yG_{r,M}(y^k,z^k))^T(Ax^k-y^k)\\
\numgeq{2}&f(x^k)+G_{r,M}(Ax^k,z^k)+\\
&\frac{\beta-L_{G}^{r,M}}{2}\|Ax^k-y^k\|^2+\beta_y(y^k-y^{k-1})^T(Ax^k-y^k)+\\
&(\nabla_yG_{r,M}(y^k,z^{k-1})-\nabla_yG_{r,M}(y^k,z^k))^T(Ax^k-y^k)\\
\numgeq{3}&f(x^k)+G_{r,M}(Ax^k,z^k)+\frac{\beta-L_{G}^{r,M}}{2}\|Ax^k-y^k\|^2-\\
&\beta_y\|y^k-y^{k-1}\|\|Ax^k-y^k\|-\\
&L_{G}^{r,M}\|z^k-z^{k-1}\|\|Ax^k-y^k\|\\
\numgeq{4}&f(x^k)+G_{r,M}(Ax^k,z^k)+\frac{\beta-L_{G}^{r,M}}{2}\|Ax^k-y^k\|^2\\
&-(\beta_y+L_G^{r,M}L_{zy}^{r^2,3M/5,4M/5})\|y^k-y^{k-1}\|\|Ax^k-y^k\|\\
\geq&f(x^k)+G_{r,M}(Ax^k,z^k)+\\
&\left[\frac{\beta-L_{G}^{r,M}}{2}-\frac{\beta(\beta_y+L_G^{r,M}L_{zy}^{r^2,3M/5,4M/5})^2}{16\beta_y^2(1+(L_G^{r,M}L_{zy}^{r^2,3M/5,4M/5})^2)}\right]\\
&\|Ax^k-y^k\|^2\\
&-\frac{4\beta_y^2(1+(L_G^{r,M}L_{zy}^{r^2,3M/5,4M/5})^2)}{\beta}\|y^k-y^{k-1}\|^2,
\end{align*}
and the result follows by our choice of~\prettyref{eq:param-c5}. Here (1) is due to~\prettyref{eq:optimality-b-G}, (2) and (3) are due to~\prettyref{lem:extension} iii) as well as~\prettyref{eq:Lipschitz-a} and~\prettyref{eq:Lipschitz-b}, and (4) is due to~\prettyref{eq:CLIP}.
\end{proof}
\subsection{\label{sec:connection}Connection between BC-ADMM and BC-ADMM$_G$}
We show that the $x^k,y^k,z^k$ sequences generated by BC-ADMM and BC-ADMM$_G$ are the same under the proper parameter choice. as a result, BC-ADMM has the same convergence speed as BC-ADMM$_G$, aka. $O(\epsilon^{-2})$. To begin with, we need to analyze the behavior of the first iteration of BC-ADMM$_G$ as follows:
\begin{lemma}
We take~\prettyref{ass:bivariable} and~\prettyref{ass:param}. The first iteration of BC-ADMM$_G$ satisfies:
\begin{subequations}
\begin{align}[left={\empheqlbrace}]
\label{eq:init-a}
&\|x^1-x^0\|\leq\frac{\delta}{2\|A\|}\\
\label{eq:init-b}
&\|y^1-Ax^1\|\leq\frac{\beta_y+L_G^{r,M}}{\beta_y+\beta-L_G^{r,M}}\frac{\delta}{2}\leq\frac{\delta}{2}\\
\label{eq:init-c}
&\|y^1-y^0\|\leq\delta\\
\label{eq:init-d}
&z^1=z(y^1).
\end{align}
\end{subequations}
\end{lemma}
\begin{proof}
From our choice of initial value and~\prettyref{eq:optimality-b-G}, we have:
\begin{align*}
0=&\beta_y(y^1-Ax^0)+\nabla_yG_{r,M}(y^1,z^0)+\\
&\beta(y^1-Ax^1)-\nabla_yG_{r,M}(Ax^0,z^0)\\
=&\beta_y(y^1-Ax^1+Ax^1-Ax^0)+\beta(y^1-Ax^1)+\\
&[\nabla_yG_{r,M}(y^1,z^0)-\nabla_yG_{r,M}(Ax^1,z^0)]+\\
&[\nabla_yG_{r,M}(Ax^1,z^0)-\nabla_yG_{r,M}(Ax^0,z^0)].
\end{align*}
By the triangle inequality, we have:
\begin{equation}
\begin{aligned}
\label{eq:first-iter-b}
&(\beta_y+\beta)\|y^1-Ax^1\|\\
\leq&(\beta_y+L_G^{r,M})\|A(x^1-x^0)\|+L_G^{r,M}\|y^1-Ax^1\|\\
\Rightarrow&\|y^1-Ax^1\|\leq\frac{(\beta_y+L_G^{r,M})\|A\|}{\beta_y+\beta-L_G^{r,M}}\|x^1-x^0\|.
\end{aligned}
\end{equation}
From our choice of initial value and~\prettyref{eq:optimality-a-G}, we have:
\begin{align*}
0=&\beta_x(x^1-x^0)+\nabla f(x^1)+\\
&\beta A^T(Ax^1-y^0)+A^T\nabla_yG_{r,M}(Ax^0,z^0)\\
=&\beta_x(x^1-x^0)+[\nabla f(x^1)-\nabla f(x^0)]+\\
&\beta A^T(Ax^1-Ax^0)+\nabla f(x^0)+A^T\nabla_yG_{r,M}(Ax^0,z^0).
\end{align*}
Again by the triangle inequality, we have:
\begin{equation}
\begin{aligned}
\label{eq:first-iter-a}
&\beta_x\|x^1-x^0\|
\leq\|[\beta_xI+\beta A^TA](x^1-x^0)\|\\
\leq&L_f\|x^1-x^0\|+\|\nabla f(x^0)+A^T\nabla_yG_{r,M}(Ax^0,z^0)\|\\
\Rightarrow&\|x^1-x^0\|\\
\leq&\frac{1}{\beta_x-L_f}\|\nabla f(x^0)+A^T\nabla_yG_{r,M}(Ax^0,z^0)\|\\
\leq&\frac{1}{\beta_x-L_f}\|\nabla f(x^0)+A^T\nabla_yg(Ax^0,z^0)\|\numleq{1}\frac{\delta}{2\|A\|},
\end{aligned}
\end{equation}
where in (1) we have used~\prettyref{eq:param-b} and we have proved~\prettyref{eq:init-a}. By combining~\prettyref{eq:first-iter-b} and~\prettyref{eq:first-iter-a}, we have:
\begin{align}
\label{eq:first-iter-ab}
\|y^1-Ax^1\|\leq&\frac{\beta_y+L_G^{r,M}}{\beta_y+\beta-L_G^{r,M}}\frac{\delta}{2}\numleq{1}\frac{\delta}{2},
\end{align}
where in (1) we have used~\prettyref{eq:param-c} and we have proved~\prettyref{eq:init-b}. Further, we have by combining~\prettyref{eq:first-iter-b} and~\prettyref{eq:first-iter-ab}:
\begin{align}
\label{eq:first-iter-abc}
&\|y^1-y^0\|\leq\|y^1-Ax^1\|+\|Ax^1-Ax^0\|\leq\delta,
\end{align}
and we have proved~\prettyref{eq:init-c}. Finally, we have the following chain of inequality:
\begin{equation}
\begin{aligned}
\label{eq:z1}
&G_{r,M}(y^1,z^1)\numeq{4}G_{r,M}(y^1,z(y^1))\numeq{3}g(y^1,z(y^1))\\
\numleq{2}&g(y^1,z(y^0))\numeq{1}G_{r,M}(y^1,z(y^0))=G_{r,M}(y^1,z^0).
\end{aligned}
\end{equation}
To see the above inequality, we note that $\TWO{y^0}{z^0}\in\mathcal{S}_\mathcal{Z}(r^1,2M/5)\subseteq\mathcal{S}_\mathcal{Z}(r^4,4M/5)$ and $\|y^1-y^0\|\leq\delta\leq\delta^{r^4,4M/5,M}$ by~\prettyref{eq:first-iter-abc}. Now by~\prettyref{lem:extension}, we have (1) and (2) is due to optimality of $z(y^1)$. To see (3), we can apply~\prettyref{lem:zy} to show that $\|z(y^1)-z(y^0)\|\leq L_{zy}^{r^2,3M/5,4M/5}\|y^1-y^0\|$ so that we have:
\begin{align*}
\|\TWO{y^1}{z(y^1)}\|
\leq&\|\TWO{y^0}{z^0}\|+\|\TWO{y^1}{z(y^1)}-\TWO{y^0}{z^0}\|\\
\leq&r^1+\delta(1+L_{zy}^{r^2,3M/5,4M/5})\leq r.
\end{align*}
Also note that we have already shown by (1) and (2) that $g(y^1,z(y^1))\leq M$, so we have $\TWO{y^1}{z(y^1)}\in\mathcal{S}(r,M)$ where $G_{r,M}=g$, which is (3). Now by~\prettyref{eq:step-c-G} and the definition of $\CLIP$, we have proved (4), which is exactly~\prettyref{eq:init-d} and all is proved.
\end{proof}
Next, we show that the Lyapunov candidate in the first iteration can be upper bounded:
\begin{lemma}
\label{lem:initial-psi}
We take~\prettyref{ass:bivariable} and~\prettyref{ass:param}. The Lyapunov candidate in first iteration of BC-ADMM$_G$ satisfies
$\Psi_{r,M}(x^1,y^1,z^1,\lambda^1,y^0)\leq 2M/5$.
\end{lemma}
\begin{proof}
This is due to the following inequality:
\begin{align*}
&f(x^0)+g(Ax^0,z^0)\\
\numeq{1}&f(x^0)+G_{r,M}(Ax^0,z^0)=\mathcal{L}_{r,M}(x^0,y^0,z^0,\lambda^0)\\
\numgeq{2}&\mathcal{L}_{r,M}(x^1,y^1,z^0,\lambda^0)+
\frac{\beta_x}{2}\|x^1-x^0\|^2+
\frac{\beta_y}{2}\|y^1-y^0\|^2\\
\numgeq{3}&\mathcal{L}_{r,M}(x^1,y^1,z^1,\lambda^0)+
\frac{\beta_x}{2}\|x^1-x^0\|^2+
\frac{\beta_y}{2}\|y^1-y^0\|^2\\
=&\mathcal{L}_{r,M}(x^1,y^1,z^1,\lambda^1)-(\lambda^1-\lambda^0)^T(Ax^1-y^1)+\\
&\frac{\beta_x}{2}\|x^1-x^0\|^2+
\frac{\beta_y}{2}\|y^1-y^0\|^2\\
=&\mathcal{L}_{r,M}(x^1,y^1,z^1,\lambda^1)-\beta\|Ax^1-y^1\|^2+\\
&\frac{\beta_x}{2}\|x^1-x^0\|^2+
\frac{\beta_y}{2}\|y^1-y^0\|^2\\
\geq&\Psi_{r,M}(x^1,y^1,z^1,\lambda^1,y^0)-\beta\|Ax^1-y^1\|^2+\\
&\frac{\beta_x}{2}\|x^1-x^0\|^2+\frac{\beta_y}{2}\|y^1-y^0\|^2-\\
&\frac{4\beta_y^2(1+(L_G^{r,M}L_{zy}^{r^2,3M/5,4M/5})^2)}{\beta}\|y^1-y^0\|^2\\
\numgeq{4}&\Psi_{r,M}(x^1,y^1,z^1,\lambda^1,y^0)-\beta\|Ax^1-y^1\|^2+\\
&\frac{\beta_x}{2}\|x^1-x^0\|^2+\frac{\beta_y}{4}\|y^1-y^0\|^2\\
\numgeq{5}&\Psi_{r,M}(x^1,y^1,z^1,\lambda^1,y^0)-\frac{\beta(\beta_y+L_G^{r,M})^2}{(\beta_y+\beta-L_G^{r,M})^2}\frac{\delta^2}{4}\\
\numgeq{6}&\Psi_{r,M}(x^1,y^1,z^1,\lambda^1,y^0)-\frac{\delta^2}{\beta}.
\end{align*}
Here (1) uses our choice of initial guess and the fact that $\TWO{Ax^0}{z^0}\in\mathcal{S}(r,M)$. (2) is due to~\prettyref{eq:xy-optimality-G}. (3) is due to~\prettyref{eq:z1}. (4) is due to~\prettyref{eq:param-c2}. (5) is due to~\prettyref{eq:init-b}. (6) is due to~\prettyref{eq:param-c3}. Finally, our desired result is due to~\prettyref{eq:param-c4}. The proof is complete.
\end{proof}
We have established all the results about the initial iteration and we are now ready to show the equivalence between BC-ADMM and BC-ADMM$_G$. We prove the following result by induction:
\begin{lemma}
\label{lem:same-sequence}
We take~\prettyref{ass:bivariable} and~\prettyref{ass:param}. Then every iteration generated by BC-ADMM$_G$ satisfies:
\begin{subequations}
\begin{align}[left={\empheqlbrace}]
\label{eq:seq-a}
&y^k \text{ by BC-ADMM and BC-ADMM}_G\text{ is same}\\
\label{eq:seq-b}
&z^k=z(y^k)\\
\label{eq:seq-c}
&f(x^k)+g(Ax^k,z^k)\leq 2M/5.
\end{align}
\end{subequations}
In summary, the $x^k,y^k,z^k$ sequence generated by BC-ADMM$_G$ and BC-ADMM are the same.
\end{lemma}
\begin{proof}
We use mathematical induction. The case with $k=0$ is trivially due to our choice of initial value.
Base Step ($k=1$,~\prettyref{eq:seq-b} and~\prettyref{eq:seq-c}): We have $f(x^0)+g(Ax^0,z^0)\leq M/5\leq 2M/5$ by our initial guess so $\TWO{Ax^0}{z^0}\in\mathcal{S}_\mathcal{Z}(r^1,2M/5)$. We have $\|Ax^0-Ax^1\|\leq\delta\leq\delta^{r^1,2M/5,3M/5}$ by~\prettyref{eq:init-a} and:
\begin{align*}
\|\TWO{Ax^1}{z^0}\|
\leq&\|\TWO{Ax^0}{z^0}\|+\|Ax^1-Ax^0\|\\
\leq&r^1+\delta^{r^1,2M/5,3M/5}=r^2,
\end{align*}
leading to $\TWO{Ax^1}{z^0}\in\mathcal{S}_\mathcal{Z}(r^2,3M/5)$. We also have $\|Ax^0-y^0\|=0\leq\delta^{r^1,2M/5,3M/5}$ so $\TWO{y^0}{z^0}\in\mathcal{S}_\mathcal{Z}(r^1,3M/5)\subseteq\mathcal{S}_\mathcal{Z}(r^2,3M/5)$. Next, we have $\|Ax^1-y^1\|\leq\delta\leq\delta^{r^2,3M/5,4M/5}$ by~\prettyref{eq:init-b} and:
\begin{align*}
\|\TWO{y^1}{z^0}\|
\leq&\|\TWO{Ax^1}{z^0}\|+\|y^1-Ax^1\|\\
\leq&r^2+\delta^{r^2,3M/5,4M/5}=r^3,
\end{align*}
leading to $\TWO{y^1}{z^0}\in\mathcal{S}_\mathcal{Z}(r^3,4M/5)$. Then, we have $\TWO{y^0}{z^0}\in\mathcal{S}_\mathcal{Z}(r^2,3M/5)$ and $\|y^1-y^0\|\leq\delta\leq\delta^{r^2,3M/5,4M/5}$ so by~\prettyref{lem:zy} and~\prettyref{eq:init-d}:
\begin{align*}
\|z^1-z^0\|=\|z(y^1)-z(y^0)\|\leq L_{zy}^{r^2,3M/5,4M/5}\|y^1-y^0\|.
\end{align*}
Now by optimality of $z(y^1)$, we know that $g(y^1,z^1)=g(y^1,z(y^1))\leq g(y^1,z^0)\leq 4M/5$ and:
\begin{align*}
\|\TWO{y^1}{z^1}\|
\leq&\|\TWO{y^1}{z^0}\|+\|z^1-z^0\|\\
\leq&r^3+L_{zy}^{r^2,3M/5,4M/5}\delta=r^4,
\end{align*}
leading to $\TWO{y^1}{z^1}\in\mathcal{S}_\mathcal{Z}(r^4,4M/5)$. Finally, we have $\|Ax^1-y^1\|\leq\delta\leq\delta^{r^4,4M/5,M}$ so that: $\TWO{Ax^1}{z^1}\in\mathcal{S}_\mathcal{Z}(r^4+\delta^{r^4,4M/5,M},M)=\mathcal{S}_\mathcal{Z}(r,M)$. However, we also have:
\begin{align*}
0\leq&f(x^1)+g(Ax^1,z^1)=f(x^1)+G_{r,M}(Ax^1,z^1)\\
\numleq{2}&\Psi_{r,M}(x^1,y^1,z^1,\lambda^1,y^0)\numleq{1}2M/5,
\end{align*}
where (1) is due to~\prettyref{lem:initial-psi} and (2) is due to~\prettyref{lem:LowerBound}. We have essentially proved~\prettyref{eq:seq-c} and~\prettyref{eq:seq-b} is due to~\prettyref{eq:init-d}. Now we have finished proving the base case.

Inductive Step ($k>1$,~\prettyref{eq:seq-b} and~\prettyref{eq:seq-c}): We have $f(x^k)+g(Ax^k,z^k)\leq M/5\leq 2M/5$ by our inductive assumption so $\TWO{Ax^k}{z^k}\in\mathcal{S}_\mathcal{Z}(r^1,2M/5)$. We have the following estimate in the change of $x^k$:
\begin{align*}
&\|x^{k+1}-x^k\|\\
\numleq{1}&\sqrt{\frac{4\left(\begin{array}{c}\Psi_{r,M}(x^k,y^k,z^k,\lambda^k,y^{k-1})-\\
\Psi_{r,M}(x^{k+1},y^{k+1},z^{k+1},\lambda^{k+1},y^k)\end{array}\right)}{\beta_x}}\\
\numleq{2}&\sqrt{\frac{4(\Psi_{r,M}(x^1,y^1,z^1,\lambda^1,y^0)-\underline{G}_{r,M})}{\beta_x}}\\
\numleq{3}&\sqrt{\frac{4(2M/5-\underline{G}_{r,M})}{\beta_x}}
\numleq{4}\frac{\delta}{\|A\|},
\end{align*}
where (1) is due to~\prettyref{eq:delta}, (2) is due to~\prettyref{lem:LowerBound}, (3) is due to~\prettyref{lem:initial-psi}, and (4) is due to~\prettyref{eq:param-b2}. We further have the following estimate in teh change of $y^k$:
\begin{align*}
&\|y^{k+1}-y^k\|\\
\numleq{1}&\sqrt{\frac{4\left(\begin{array}{c}\Psi_{r,M}(x^k,y^k,z^k,\lambda^k,y^{k-1})-\\
\Psi_{r,M}(x^{k+1},y^{k+1},z^{k+1},\lambda^{k+1},y^k)\end{array}\right)}{\beta_y}}\\
\numleq{2}&\sqrt{\frac{4(\Psi_{r,M}(x^1,y^1,z^1,\lambda^1,y^0)-\underline{G}_{r,M})}{\beta_y}}\\
\numleq{3}&\sqrt{\frac{4(2M/5-\underline{G}_{r,M})}{\beta_y}}
\numleq{4}\delta,
\end{align*}
where (1) is due to~\prettyref{eq:delta}, (2) is due to~\prettyref{lem:LowerBound}, (3) is due to~\prettyref{lem:initial-psi}, and (4) is due to~\prettyref{eq:param-b3}. We thus have $\|Ax^k-Ax^{k+1}\|\leq\delta\leq\delta^{r^1,2M/5,3M/5}$ and:
\begin{align*}
\|\TWO{Ax^{k+1}}{z^k}\|
\leq&\|\TWO{Ax^k}{z^k}\|+\|Ax^{k+1}-Ax^k\|\\
\leq&r^1+\delta^{r^1,2M/5,3M/5}=r^2,
\end{align*}
leading to $\TWO{Ax^{k+1}}{z^k}\in\mathcal{S}_\mathcal{Z}(r^2,3M/5)$. We further have the following bound on $\|Ax^{k+1}-y^{k+1}\|$ or $\|Ax^k-y^k\|$:
\begin{align*}
&\max(\|Ax^{k+1}-y^{k+1}\|^2,\|Ax^k-y^k\|^2)\\
\numleq{1}&\sqrt{\frac{2M/5-\underline{G}_{r,M}}{\left[\frac{\beta-L_{G}^{r,M}}{2}-\frac{\beta(\beta_y+L_G^{r,M}L_{zy}^{r^2,3M/5,4M/5})^2}{16\beta_y^2(1+(L_G^{r,M}L_{zy}^{r^2,3M/5,4M/5})^2)}\right]}}\numleq{2}\delta,
\end{align*}
where (1) is due to~\prettyref{lem:LowerBound} and~\prettyref{lem:initial-psi}, while (2) is due to~\prettyref{eq:param-c5}. Combined with the fact that $\TWO{Ax^k}{z^k}\in\mathcal{S}_\mathcal{Z}(r^1,2M/5)$, we know that $\TWO{y^k}{z^k}\in\mathcal{S}(r^2,3M/5)$. Next, we combine the fact that $\TWO{Ax^{k+1}}{z^k}\in\mathcal{S}_\mathcal{Z}(r^2,3M/5)$ and:
\begin{align*}
\|\TWO{y^{k+1}}{z^k}\|
\leq&\|\TWO{Ax^{k+1}}{z^k}\|+\|y^{k+1}-Ax^{k+1}\|\\
\leq&r^2+\delta^{r^2,3M/5,4M/5}=r^3,
\end{align*}
leading to $\TWO{y^{k+1}}{z^k}\in\mathcal{S}_\mathcal{Z}(r^3,4M/5)$. Then, we have $\TWO{y^k}{z^k}\in\mathcal{S}_\mathcal{Z}(r^2,3M/5)$ and $\|y^{k+1}-y^k\|\leq\delta\leq\delta^{r^2,3M/5,4M/5}$ so by~\prettyref{lem:zy} and:
\begin{align*}
\|z^{k+1}-z^k\|=&\|z(y^{k+1})-z(y^k)\|\\
\leq&L_{zy}^{r^2,3M/5,4M/5}\|y^{k+1}-y^k\|,
\end{align*}
and~\prettyref{eq:seq-b} follows by an identical argument as in~\prettyref{eq:z1}. Now optimality of $z(y^{k+1})$, we know that $g(y^{k+1},z^{k+1})=g(y^{k+1},z(y^{k+1}))\leq g(y^{k+1},z^k)\leq 4M/5$ as well as:
\begin{align*}
\|\TWO{y^{k+1}}{z^{k+1}}\|
\leq&\|\TWO{y^{k+1}}{z^k}\|+\|z^{k+1}-z^k\|\\
\leq&r^3+L_{zy}^{r^2,3M/5,4M/5}\delta=r^4,
\end{align*}
leading to $\TWO{y^{k+1}}{z^{k+1}}\in\mathcal{S}_\mathcal{Z}(r^4,4M/5)$. Finally, we have $\|Ax^{k+1}-y^{k+1}\|\leq\delta\leq\delta^{r^4,4M/5,M}$ so that: $\TWO{Ax^{k+1}}{z^{k+1}}\in\mathcal{S}_\mathcal{Z}(r^4+\delta^{r^4,4M/5,M},M)=\mathcal{S}_\mathcal{Z}(r,M)$. However, we also have:
\begin{align*}
0\leq&f(x^{k+1})+g(Ax^{k+1},z^{k+1})\\
=&f(x^{k+1})+G_{r,M}(Ax^{k+1},z^{k+1})\\
\numleq{2}&\Psi_{r,M}(x^{k+1},y^{k+1},z^{k+1},\lambda^{k+1},y^k)\numleq{1}2M/5,
\end{align*}
where (1) is due to~\prettyref{lem:initial-psi} and (2) is due to~\prettyref{lem:LowerBound}. We have essentially proved~\prettyref{eq:seq-c}. Now we have finished proving the inductive case.

Base \& Inductive Step ($k\geq1$,~\prettyref{eq:seq-a}): From the analysis above, we know that $\TWO{y^k}{z^k}\in\mathcal{S}_\mathcal{Z}(r^2,3M/5)$ for every $k\geq1$ and $\|y^{k+1}-y^k\|\leq\delta\leq\delta^{r^2,3M/5,4M/5}$. As a result, $\TWO{y^{k+1}}{z^k}\in\mathcal{S}(r^3,4M/5)$. Therefore, the gradient of $G_{r,M}$ and $g$ matches and we have at $y^{k+1}$ generated by BC-ADMM$_G$:
\begin{align*}
0=&\nabla_y\left[\mathcal{L}_{r,M}(x^{k+1},y^{k+1},z^k,\lambda^k)+\frac{\beta_y}{2}\|y^{k+1}-y^k\|^2\right]\\
=&\nabla_y\left[\mathcal{L}(x^{k+1},y^{k+1},z^k,\lambda^k)+\frac{\beta_y}{2}\|y^{k+1}-y^k\|^2\right].
\end{align*}
But since the objective function in~\prettyref{eq:step-b} of BC-ADMM is strongly convex, whose optimal solution is unique, we must have $y^{k+1}$ generated by BC-ADMM is equivalent to that generated by BC-ADMM$_G$.

Finally, we still need to show that BC-ADMM$_G$ and BC-ADMM generates the same sequence. Note that the two algorithm uses different steps in the update of $y^k$ and $z^k$. But we have shown that $z^k=z(y^k)$ for every $k$ and that the update from the same $y^k$ leads to the same $y^{k+1}$. Our desired result follows.
\end{proof}
As our last step of analysis, we establish the convergence speed of BC-ADMM$_G$ in the following result:
\begin{lemma}
\label{lem:speed}
Under~\prettyref{ass:bivariable} and~\prettyref{ass:param}, BC-ADMM$_G$ converges at a speed of $O(\epsilon^{-2})$ to $\epsilon$-stationary point of~\prettyref{eq:sadmm}.
\end{lemma}
\begin{proof}
We can repeatedly applying~\prettyref{lem:Lyapunov} and, by letting $\beta_{xy}=\min(\beta_x,\beta_y)$, we have:
\begin{align*}
&\frac{\beta_{xy}}{4}\sum_{k=1}^{K}\|x^{k+1}-x^k\|^2+\|y^{k+1}-y^k\|^2\\
\leq&\frac{\beta_x}{4}\sum_{k=1}^{K}\|x^{k+1}-x^k\|^2+
\frac{\beta_y}{4}\sum_{k=1}^{K}\|y^{k+1}-y^k\|^2\\
\leq&\Psi_{r,M}(x^1,y^1,\lambda^1,y^0)-\Psi_{r,M}(x^{K+1},y^{K+1},\lambda^{K+1},y^K).
\end{align*}
Adding above equations twice and define $\theta^k=\|x^{k+1}-x^k\|^2+\|y^{k+1}-y^k\|^2+\|x^k-x^{k-1}\|^2+\|y^k-y^{k-1}\|^2$, we have by the second term in parameter choice and~\prettyref{lem:LowerBound}:
\begin{align*}
&\frac{\beta_{xy}}{4}\sum_{k=2}^{K}\theta^k\\
=&\frac{\beta_{xy}}{4}\left[\begin{array}{c}\sum_{k=1}^{K-1}\|x^{k+1}-x^k\|^2+\|y^{k+1}-y^k\|^2+\\
\sum_{k=2}^{K}\|x^{k+1}-x^k\|^2+\|y^{k+1}-y^k\|^2\end{array}\right]\\
\leq&2\Psi_{r,M}(x^1,y^1,\lambda^1,y^0)-2\Psi_{r,M}(x^{K+1},y^{K+1},\lambda^{K+1},y^K)\\
\leq&2(\Psi_{r,M}(x^1,y^1,\lambda^1,y^0)-\underline{G}_{r,M}).
\end{align*}
Therefore, we can bound the change in both $x$ and $y$ over consecutive iterations as follows:
\begin{align}
\min_{k=2,\cdots,K}\theta^k
\leq\frac{8(\Psi_{r,M}(x^1,y^1,\lambda^1,y^0)-\underline{G}_{r,M})}{\beta_{xy}K}.
\end{align}
Now define $k=\argmin{k=2,\cdots,K}\theta^k$ and we bound the norm of objective's gradient with respect to $x^k$.
\small
\begin{align*}
&\|\nabla f(x^k)+A^T\nabla_yg(Ax^k,z^k)\|^2\\
=&\|\nabla f(x^k)+A^T\nabla_y G_{r,M}(Ax^k,z^k)\|^2\\
=&\|\nabla f(x^k)+A^T\lambda^k+A^T(\nabla_yG_{r,M}(y^k,z^{k-1})-\lambda^k)+\\
&A^T(\nabla_yG_{r,M}(y^k,z^k)-\nabla_yG_{r,M}(y^k,z^{k-1}))+\\
&A^T(\nabla_yG_{r,M}(Ax^k,z^k)-\nabla_yG_{r,M}(y^k,z^k))\|^2\\
\numeq{1}&\|-\beta_x(x^k-x^{k-1})-\beta A^T(y^k-y^{k-1})-\beta_yA^T(y^k-y^{k-1})+\\
&A^T(\nabla_yG_{r,M}(y^k,z^k)-\nabla_yG_{r,M}(y^k,z^{k-1}))+\\
&A^T(\nabla_yG_{r,M}(Ax^k,z^k)-\nabla_yG_{r,M}(y^k,z^k))\|^2\\
\numleq{2}&4\beta_x^2\|x^k-x^{k-1}\|^2+4(\beta+\beta_y)^2\|A\|^2\|y^k-y^{k-1}\|^2+\\
&4(L_G^{r,M})^2\|A\|^2[\|z^k-z^{k-1}\|^2+\|Ax^k-y^k\|^2]\\
=&4\beta_x^2\|x^k-x^{k-1}\|^2+
4(\beta+\beta_y)^2\|A\|^2\|y^k-y^{k-1}\|^2+\\
&4(L_G^{r,M})^2\|A\|^2\left[\|z^k-z^{k-1}\|^2+\frac{\|\lambda^k-\lambda^{k-1}\|^2}{\beta^2}\right]\\
\numleq{3}&4\beta_x^2\|x^k-x^{k-1}\|^2+
4(\beta+\beta_y)^2\|A\|^2\|y^k-y^{k-1}\|^2+\\
&4(L_G^{r,M}L_{zy}^{r^2,3M/5,4M/5})^2\|A\|^2\|y^k-y^{k-1}\|^2+4\frac{(L_G^{r,M})^2}{\beta^2}\|A\|^2\\
&\left[\begin{array}{c}4(\beta_y^2+(L_G^{r,M})^2)\|y^{k+1}-y^k\|^2+\\
4\beta_y^2(1+(L_G^{r,M}L_{zy}^{r^2,3M/5,4M/5})^2)\|y^k-y^{k-1}\|^2\end{array}\right]\\
=&O(\theta^k).
\end{align*}
\normalsize
Here (1) is due to~\prettyref{eq:optimality-a2-G} and~\prettyref{eq:optimality-b-G}, (2) is due to~\prettyref{lem:extension} iii), and (3) is due to~\prettyref{lem:lambda}. Combined, we have:
\begin{equation}
\begin{aligned}
\label{eq:conv-a}
&\|\nabla f(x^k)+A^T\nabla_yg(Ax^k,z^k)\|^2\\
=&O\left(\frac{8(\Psi_{r,M}(x^1,y^1,\lambda^1,y^0)-\underline{G}_{r,M})}{\beta_{xy}K}\right)=O(\epsilon^2).
\end{aligned}
\end{equation}
Next, we bound the norm of objective's gradient with respect to $z^k$. By~\cite[Theorem 8.15]{rockafellar2009variational} and the fact that $z^k=z(y^k)$ due to~\prettyref{lem:same-sequence}, we have $-\nabla_zg(y^k,z^k)=N_\mathcal{Z}(z^k)$, with $N_\mathcal{Z}(z^k)$ being the normal cone of $\mathcal{Z}$ at $z^k$. As a result, we have: 
\small
\begin{align*}
&\dist(N_\mathcal{Z}(z^k),-\nabla_zg(Ax^k,z^k))^2\\
=&\dist(N_\mathcal{Z}(z^k),-\nabla_zG_{r,M}(Ax^k,z^k))^2\\
\numleq{1}&\|\nabla_zG_{r,M}(Ax^k,z^k)-\nabla_zG_{r,M}(y^k,z^k)\|^2\\
\numeq{2}&(L_G^{r,M})^2\|Ax^k-y^k\|^2=\frac{(L_G^{r,M})^2}{\beta^2}\|\lambda^k-\lambda^{k-1}\|^2\\
\numleq{3}&\frac{(L_G^{r,M})^2}{\beta^2}\left[\begin{array}{c}4(\beta_y^2+(L_G^{r,M})^2)\|y^{k+1}-y^k\|^2+\\4\beta_y^2(1+(L_G^{r,M}L_{zy}^{r^2,3M/5,4M/5})^2)\|y^k-y^{k-1}\|^2\end{array}\right]\\
=&O(\theta^k).
\end{align*}
\normalsize
Here (1) is due to the fact that $-\nabla_zG_{r,M}(y^k,z^k)=-\nabla_zg(y^k,z^k)\in N_\mathcal{Z}(z^k)$ and that $\dist(N_\mathcal{Z}(z^k),a)\leq\|a-b\|+\dist(N_\mathcal{Z}(z^k),b)$, (2) is due to~\prettyref{lem:extension} iii), and (3) is due to~\prettyref{lem:lambda}. Combined we have:
\begin{equation}
\begin{aligned}
\label{eq:conv-b}
&\dist(N_\mathcal{Z}(z^k),-\nabla_zg(Ax^k,z^k))^2\\
=&O\left(
\frac{8(\Psi_{r,M}(x^1,y^1,\lambda^1,y^0)-\underline{G}_{r,M})}{\beta_{xy}K}\right)=O(\epsilon^2).
\end{aligned}
\end{equation}

Our desired result follows by combining~\prettyref{eq:conv-a} and~\prettyref{eq:conv-b}.
\end{proof}
\begin{theorem}
\label{thm:BC-ADMM}
Taking~\prettyref{ass:bivariable} and under sufficiently large $\beta_x,\beta_y,\beta$, each iteration generated by BC-ADMM satisfies $f(x^k)+g(Ax^k,z^k)<\infty$. Further, BC-ADMM converges to the $\epsilon$-stationary solution of~\prettyref{eq:snlp} with an oracle complexity of $O(\epsilon^{-2})$.
\end{theorem}
\begin{proof}
Combine the result of~\prettyref{lem:same-sequence} and~\prettyref{lem:speed}.
\end{proof}
\section{\label{sec:practical-convergence}Convergence of~\prettyref{alg:practical}}
In \prettyref{alg:practical}, we introduced several properties to allow automatic tuning of parameters and progressively introduce constraints. To formally argue for the convergence, we use the observation that with $z$ fixed, our objective $f(\bullet)+g(A\bullet,z)$ is convex and~\prettyref{alg:practical} becomes a special, two-block case of the proximal ADMM algorithm analyzed in~\cite{deng2017parallel}. According to their analysis, the parameters $\beta,\beta_x,\beta_y$ need to satisfy the following conditions:
\begin{align}
\label{eq:convex-convergence}
\begin{cases}
1>\epsilon_x+\epsilon_y\\
\beta_xI\succ\beta(1/\epsilon_x-1)A^TA\\
\beta_y>\beta(1/\epsilon_y-1)\\
\end{cases}.
\end{align}
To prove convergence, we show that our parameter choices in~\prettyref{alg:practical} always satisfies~\prettyref{eq:convex-convergence} by construction. It also satisfies~\prettyref{ass:param} after a sufficient number of parameter tuning iterations (below~\prettyref{ln:feasibility}). As a result, the feasibility of our solution is ensured by~\prettyref{thm:BC-ADMM}. Our result is formally claimed in~\prettyref{thm:asymptotic} and proved here.
\begin{proof}[Proof of~\prettyref{thm:asymptotic}]
We note that~\prettyref{ln:lazy-z} ensures that $x^\star$ stores the best solution with the smallest $f(x^\star)+g(Ax^\star,z^\star)\leq f(x^0)+g(Ax^0,z^0)<\infty$. We prove several properties of~\prettyref{alg:practical}.

We first prove that the parameter tuning iteration below~\prettyref{ln:feasibility} will be entered only finitely many times. This is because the parameter tuning iteration would restart the simulation from the best solution at $\langle x,z\rangle^{k+1}\gets\langle x,z\rangle^\star$ while resetting $y^{k+1}$ and $\lambda^{k+1}$. Further, each parameter tuning iteration would increase $\beta,\beta_y$ indefinitely. Since $\kappa>\kappa_y$, $\beta$ would be ultimately larger than $\beta_y$ by an arbitrarily large factor. We could easily see that $\beta_x$ would also be indefinitely large, which is obvious from the following rewrite:
\begin{align*}
\beta_x\gets\beta(1/\gamma+\beta/\beta_y-1)\|A^TA\|/\gamma.
\end{align*}
As a result, all the requirements from~\prettyref{ass:param} would be satisfied within finitely many parameter tuning iterations, after which all the iterations are strictly feasible by~\prettyref{thm:BC-ADMM} and the condition below~\prettyref{ln:feasibility} will never be met. As a result, we can safely ignore the finitely many iterations and assume that $\beta_x,\beta_y,\beta$ are never modified while all iterations are feasible.

Second, we investigate the number of times $z^\star$ will be updated to $z(y^{k+1})$. Note that we use a lazy strategy and update $z^\star$ only when the two conditions below~\prettyref{ln:lazy-z} are met. Suppose these conditions are not met and $z^\star$ is not updated, then our algorithm becomes the conventional ADMM~\cite{deng2016global} minimizing the convex objective function $f(\bullet)+g(A\bullet,z^\star)$. Due to our parameter choice,~\prettyref{eq:convex-convergence} is always satisfied and \cite[Theorem 2.3]{deng2016global} applies, which shows that $f(x^{k+1})+g(y^{k+1},z^\star)$ converges to the (globally) optimal solution of that objective function denoted as $x^\dagger$, i.e. $\lim_{k\to\infty}x^k=x^\dagger$ and $\lim_{k\to\infty}y^k=Ax^\dagger$. This also implies that $\lim_{k\to\infty}\Phi^{k+1}=0$. As a result, the second condition below~\prettyref{ln:lazy-z}, i.e., $\Phi^{k+1}\leq\eta^K$, will be met for sufficiently large $k$. In order for $z^\star$ to be updated, we also need the first condition to be satisfied. To this end, we consider two cases. Case I: If $g(Ax^\dagger,z(Ax^\dagger))<g(Ax^\dagger,z^\star)$, then we have: 
\begin{align*}
&\lim_{k\to\infty}f(x^{k+1})+g(Ax^{k+1},z^{k+1})
\numeq{1}f(x^\dagger)+g(Ax^\dagger,z(Ax^\dagger))\\
=&f(x^\dagger)+g(Ax^\dagger,z^\star)+
(g(Ax^\dagger,z(Ax^\dagger))-g(Ax^\dagger,z^\star))\\
\numleq{2}&f(x^\star)+g(Ax^\star,z^\star)+
(g(Ax^\dagger,z(Ax^\dagger))-g(Ax^\dagger,z^\star)).
\end{align*}
In (1) above, we used the fact that $\lim_{k\to\infty}z^{k+1}=\lim_{k\to\infty}z(y^{k+1})=z(y^\dagger)=z(Ax^\dagger)$ because $z(y)$ is a locally Lipschitz function of $y$ by our~\prettyref{ass:bivariable} vi). (2) is due to optimality of $x^\dagger$. We further have:
\begin{align*}
\lim_{k\to\infty}\Lambda^{k+1}=&f(x^\star)+g(Ax^\star,z^\star)+\\
&(1-\eta)(g(Ax^\dagger,z(Ax^\dagger))-g(Ax^\dagger,z^\star)).
\end{align*}
Put together, we know that $\lim_{k\to\infty}f(x^{k+1})+g(Ax^{k+1},z^{k+1})<\lim_{k\to\infty}\Lambda^{k+1}$ so that the first condition below~\prettyref{ln:lazy-z} will be satisfied after finitely many iterations where $z^\star$ will be updated. Case II: If $g(Ax^\dagger,z(Ax^\dagger))=g(Ax^\dagger,z^\star)$, then $z^\star=z(Ax^\dagger)$ due to the uniqueness of $z(Ax^\dagger)$ by our~\prettyref{ass:bivariable} vi) and $z^\star$ might never be updated further. Suppose this is the case, due to the fact that $\lim_{k\to\infty}\Phi^{k+1}=0$, we can easily verify that $\langle x^{k+1},y^{k+1},z^\star\rangle$ is the desired $\epsilon$-stationary solution of~\prettyref{eq:snlp} for sufficiently large $k$.

As our final step, we consider the case where $z^\star$ is updated infinitely many times. In this case, we denote the subsequence of iterations that update $z^k$ as $z^{k(i)}$ with $i=1,2,\cdots$. In other words, we assume that, at every iteration $k(i)$, the two conditions below~\prettyref{ln:lazy-z} are met and we have $z^\star=z^{k(i)+1}=z(y^{k(i)+1})$. The sequence $\{\langle x,y,z\rangle^{k(i)+1}\}$ is bounded so that there exists a convergent subsequence. Without a loss of generality in the following arguments, we can assume the entire sequence $\{\langle x,y,z\rangle^{k(i)+1}\}$ converges to some $\langle x,y,z\rangle^\dagger$. Clearly, we have $z^\dagger=z(y^\dagger)$ because $z(y)$ is a locally Lipschitz function of $y$ by our~\prettyref{ass:bivariable} vi).

We first claim that $\lim_{i\to\infty}g(y^{k(i)+1},z^{k(i)+1})-g(y^{k(i)+1},z^{k(i)})=0$. This is because the sequence $\{f(x^\star)+g(Ax^\star,z^\star)\}$ is monotonically decreasing and it reduces by at least $(1-\eta)(g(y^{k(i)+1},z^{k(i)+1})-g(y^{k(i)+1},z^{k(i)}))$ at the $k(i)$th iteration due to the first condition below~\prettyref{ln:lazy-z}. If $\liminf_{i\to\infty}g(y^{k(i)+1},z^{k(i)+1})-g(y^{k(i)+1},z^{k(i)})<0$, then $f(x^\star)+g(Ax^\star,z^\star)$ would tend to $-\infty$, which contradicts the fact that both $f$ and $g$ are positive functions. We then claim that $\lim_{i\to\infty}z^{k(i)+1}-z^{k(i)}=0$. Indeed, suppose otherwise, there exists some $\epsilon>0$ where $\|z^{k(i)+1}-z^{k(i)}\|>\epsilon$ for infinitely many $i$. We can assume that $z^{k(i)}$ converges (otherwise, choose a convergent subsequence), but with a limit different from $z(y^\dagger)$. But we know by our previous claim that $\lim_{i\to\infty}g(y^{k(i)+1},z^{k(i)+1})-g(y^{k(i)+1},z^{k(i)})=g(y^\dagger,z(y^\dagger))-\lim_{i\to\infty}g(y^\dagger,z^{k(i)})=0$. This contradicts our~\prettyref{ass:bivariable} vi) that $z(y^\dagger)$ is unique. Our second claim immediately implies that $\lim_{i\to\infty}z^{k(i)}=z(y^\dagger)$, which combined with the fact that $\lim_{k\to\infty}\Phi^{k(i)+1}=0$ again leads to $\epsilon$-stationary solution of~\prettyref{eq:snlp} for sufficiently large $k$.
\end{proof}
\begin{remark}
In \prettyref{thm:asymptotic}, we have assumed that the constraint detector below~\prettyref{ln:detector} has never been invoked. In practice, there are finitely many constraints so that the detector will be invoked for at most finitely many times. Further, a well-defined constraint detector only ignores constraint function terms when the value of these terms are smaller than some user-defined threshold. It is not difficult to show asymptotic convergence under this setting.
\end{remark}
\section{Special Cases of~\prettyref{thm:BC-ADMM}}
In this section, we prove several special cases of~\prettyref{thm:BC-ADMM}, adapting it to different application domains. First, we show that two terms $g_1$ and $g_2$ can be summed together and still satisfy~\prettyref{ass:bivariable}.
\begin{proof}[Proof of~\prettyref{cor:summation}]
Define $y=\TWO{y_1}{y_2}$, $z=\TWO{z_1}{z_2}$, and $z(y)=\TWO{z_1(y_1)}{z_2(y_2)}$. Suppose $z_1(y_1)$ and $z_2(y_2)$ are unique global minimizers of $g_1(y_1,\bullet)+\iota_{\mathcal{Z}_1}(\bullet)$ and $g_2(y_2,\bullet)+\iota_{\mathcal{Z}_2}(\bullet)$, then obviously $z(y)$ is the unique global minimizer of $g(y,\bullet)+\iota_\mathcal{Z}(\bullet)$, proving the first part of~\prettyref{ass:bivariable} vi).

Now suppose $\TWO{y}{z}\in\mathcal{S}_\mathcal{Z}(r,M)$, then clearly $\TWO{y_1}{z_1}\in\mathcal{S}_{\mathcal{Z}_1}(r,M)$ and $\TWO{y_2}{z_2}\in\mathcal{S}_{\mathcal{Z}_2}(r,M)$ and by our assumption we have $\TWO{y_1}{z_1(y_1)}\in\mathcal{S}_{\mathcal{Z}_1}(\bar{r}_{z_1y}^{r,M},M)$ and $\TWO{y_2}{z_2(y_2)}\in\mathcal{S}_{\mathcal{Z}_2}(\bar{r}_{z_2y}^{r,M},M)$. Combined, we have $\|\TWO{y}{z}\|\leq\|\TWO{y_1}{z_1}\|+\|\TWO{y_2}{z_2}\|\leq \bar{r}_{z_1y}^{r,M}+\bar{r}_{z_2y}^{r,M}\triangleq\bar{r}_{zy}^{r,M}$. Therefore, we have $\TWO{y}{z(y)}\in\mathcal{S}_\mathcal{Z}(\bar{r}_{zy}^{r,M},M)$, proving the second part of~\prettyref{ass:bivariable} vi).

Finally, suppose $\TWO{y}{z(y)}\in\mathcal{S}_\mathcal{Z}(r,M)$ then $\TWO{y_1}{z_1(y_1)}\in\mathcal{S}_{\mathcal{Z}_1}(r,M)$ and $\TWO{y_2}{z_2(y_2)}\in\mathcal{S}_{\mathcal{Z}_2}(r,M)$. Therefore, $z_1(y_1)$ is $L_{z_1y}^{r,M}$-Lipschitz continuous and $z_2(y_2)$ is $L_{z_2y}^{r,M}$-Lipschitz continuous by our assumption. Our result follows by defining $L_{zy}^{r,M}=L_{z_1y}^{r,M}+L_{z_2y}^{r,M}$.
\end{proof}
Next we show that bi-convex $g$ satisfies our assumption.
\begin{lemma}
\label{lem:zy-case1}
\prettyref{ass:bivariable-case1} implies~\prettyref{ass:bivariable}.
\end{lemma}
\begin{proof}
We first reclaim what we need to prove. For $\TWO{y}{z}\in\mathcal{S}_\mathcal{Z}(r,M)$, $g(y,\bullet)+\iota_\mathcal{Z}(\bullet)$ has a unique global minimizer denoted as a function $z(y)$, such that $\TWO{y}{z(y)}\in\mathcal{S}_\mathcal{Z}(\bar{r}_{zy}^{r,M},M)$, in which is an $L_{zy}^{r,M}$-Lipschitz continuous function with $\TWO{r}{M}$-dependent parameters $\bar{r}_{zy}^{r,M}$ and $L_{zy}^{r,M}$. 

The uniqueness of minimizer is due to the $\sigma$-strong convexity of $g$. Next, we know that $\mathcal{S}(r,M)$ is a compact set and $\|\nabla_zg(y,z)\|$ is a continuous function. Therefore, we have $\|\nabla_zg(y,z)\|$ is upper bounded on $\mathcal{S}(r,M)$ with the upper bound denoted as $L_{\nabla_z}^{r,M}$. We thus have the following inequality:
\begin{align*}
M\numgeq{1}&g(y,z)\numgeq{2} g(y,z(y))\\
\numgeq{3}&g(y,z)+\nabla_z g(y,z)^T(z(y)-z)+\frac{\sigma}{2}\|z(y)-z\|^2\\
\geq&g(y,z)-L_{\nabla_z}^{r,M}\|z(y)-z\|+\frac{\sigma}{2}\|z(y)-z\|^2\\
\geq&-L_{\nabla_z}^{r,M}\|z(y)-z\|+\frac{\sigma}{2}\|z(y)-z\|^2,
\end{align*}
where (1) is due to our assumption, (2) is due to optimality, and (3) is due to $\sigma$-strong convexity. Solving the above quadratic inequation and we have:
\begin{align*}
\|z(y)-z\|
\leq\left[\sqrt{2M\sigma+(L_{\nabla_z}^{r,M})^2}+L_{\nabla_z}^{r,M,N}\right]/\sigma
\triangleq\delta_\sigma^{r,M}.
\end{align*}
Therefore, we know that:
\begin{align*}
&\|\TWO{y}{z}-\TWO{y}{z(y)}\|\leq\|\TWO{y}{z}\|+\|z-z(y)\|\\
\leq&r+\left[\sqrt{2M\sigma+(L_{\nabla_z}^{r,M})^2}+L_{\nabla_z}^{r,M,N}\right]/\sigma
\triangleq\bar{r}_{zy}^{r,M}.
\end{align*}

Now suppose $\TWO{y}{z(y)}\in\mathcal{S}_\mathcal{Z}(r,M)$. By the inverse function theorem~\cite[Theorem 2F.7]{dontchev2009implicit}, we know that $0\in N_\mathcal{Z}(z(y))+\nabla_zg(y,z(y))$ where $z(y)$ has a single-valued localization in a neighborhood around $\TWO{y}{z(y)}$ with a Lipschitz modulus upper bounded by:
\begin{align*}
&\|-\nabla_{zz}g(y,z(y))^{-1}\nabla_{zy}g(y,z(y))\|\\
\leq&\|\nabla_{zz}g(y,z(y))^{-1}\|\|\nabla_{zy}g(y,z(y))\|
\numleq{1} L_g^{r,M}/\sigma\triangleq L_{zy}^{r,M},
\end{align*}
where (1) is due to~\prettyref{lem:Lg} and $\sigma$-strong convexity, which is our desired result.
\end{proof}
\begin{proof}[Proof of~\prettyref{thm:BC-ADMM-case1}]
Combine~\prettyref{lem:zy-case1} and~\prettyref{thm:BC-ADMM}.
\end{proof}
We further show that the mass-spring energy $P_{k,l,\underline{l}}(y,z)$ satisfies our~\prettyref{ass:bivariable}.
\begin{lemma}
\label{lem:mass-spring}
$g=P_{k,l,\underline{l}}(y,z)$ and $\mathcal{Z}=\{z|\|z\|=1\}$ satisfies~\prettyref{ass:bivariable}.
\end{lemma}
\begin{proof}
Ohter parts being trivial, we only need to show~\prettyref{ass:bivariable} vi). First, we prove that $z(y)=(y_i-y_j)/\|y_i-y_j\|$ is the unique global minimizer. We have for an arbitrary unit $z$:
\begin{align*}
&\|y_i-y_j-zl\|^2\\
=&\|[z(y)z(y)^T](y_i-y_j-zl)\|^2+\\
&\|[I-z(y)z(y)^T](y_i-y_j-zl)\|^2\\
=&\|y_i-y_j-[z(y)z(y)^T]zl\|^2+\\
&\|[I-z(y)z(y)^T]zl\|^2\\
=&\|\|y_i-y_j\|-\cos(\theta)l\|^2+
\|\sin(\theta)l\|^2\\
=&\|y_i-y_j\|^2+l^2-2l\cos(\theta)\|y_i-y_j\|
\geq(\|y_i-y_j\|-l)^2,
\end{align*}
where $\theta$ is the angle between $z(y)$ and $z$. From above formula, we can see that $z(y)$ is the unique optimal solution making $\cos(\theta)=1$. Further, it can be easily verified that $z(y)$ is also the optimal solution to $-\log_\epsilon(n^T(y_i-y_j)-\underline{l})$. Put together, $z(y)$ is the unique optimal solution minimizing $g=P_{k,l,\underline{l}}(y,z)$, proving the first part of~\prettyref{ass:bivariable} vi). 

Suppose $\TWO{y}{z}\in\mathcal{S}_\mathcal{Z}(r,M)$, then we have $\|y\|\leq r$ and $\|z(y)\|\leq1$, so $\|\TWO{y}{z(y)}\|\leq r+1\triangleq\bar{r}_{zy}^{r,M}$, proving the second part of~\prettyref{ass:bivariable} vi). 

From the solution $z(y)$ above, it is differentiable with respect to $y$ with the derivative being:
\begin{align*}
\FPP{z(y)}{y}=\TWO{I-z(y)z(y)^T}{z(y)z(y)^T-I}\frac{1}{\|y_i-y_j\|}.
\end{align*}
The above matrix has a norm upper bound of $\sqrt{2}/\underline{l}\triangleq L_{zy}^{r,M}$, which is our desired Lipschitz modulus. We have thus proved the last part of~\prettyref{ass:bivariable} vi).
\end{proof}

Finally, we show that the ARAP energy with strain limiting $P_F(y,z)$ satisfies our~\prettyref{ass:bivariable}.
\begin{lemma}
\label{lem:sigma-bound}
If $P_{F,\epsilon}(y,z)\leq M$ then there exists some $\epsilon^M>0$ and $\underline{\sigma}^M>0$ such that i) The sphere centered at $y_0$ with radius $\epsilon^M$ is contained in $\CH(y_1,\cdots,y_{d+1})$ and ii) $\sigma_k(F(y))\geq\underline{\sigma}^M$.
\end{lemma}
\begin{proof}
i) First note that all the terms in $P_{F,\epsilon}$ are positive, so each of them is smaller than $M$. Let us define $\CH_i=\CH(\{y_j|0<j\neq i\})$ as the face of the simplex element opposite to vertex $y_i$. We can claim that there exists some $\epsilon^M>0$ such that $\dist(y_0,\CH_i)\geq\epsilon^M$. Because otherwise, the term $-\sum_{0<j\neq i}\log_\epsilon(n_i^T(y_j-y_0))=\infty$ models a separating plane between $y_0$ and $\CH_i$, which is continuous and tends to infinity as $\dist(y_0,\CH_i)\to0$ according to~\cite{liang2024second}, contradicting the fact that $P_{F,\epsilon}(y,z)\leq M$. We further have $y_0\in\CH(y_1,\cdots,y_{d+1})$ because $0=\iota_{\CH(y_1,\cdots,y_{d+1})}(y_0)\leq M$. This means the $d$-dimensional sphere centered at $y_0$ with radius $\epsilon^M$ is entirely contained the simplex element.

ii) We recall the following formula for calculating the deformation gradient $F(y)$:
\begin{align*}
&X^r=\THREE{y_2^r-y_1^r}{\cdots}{y_{d+1}^r-y_1^r}\\
&X(y)=\THREE{y_2-y_1}{\cdots}{y_{d+1}-y_1}\\
&F(y)=X(y)[X^r]^{-1},
\end{align*}
where we use the superscript $r$ to denote the rest pose. We know that $0<\underline{\sigma}^r\leq\sigma_k(X^r)$ for some $0<\underline{\sigma}^r$, i.e. $X^r$ is non-singular by definition. By the barycentric coordinate interpolation formula, we know that for any $q\in\mathbb{R}^d$ with $\|q\|=\epsilon^M$, we have the barycentric coordinates $0\leq\omega(q)\leq1$, such that $y_0+q=X(y)\omega(q)$. Now since $\iota_{\CH{y_1,\cdots,d+1}}(y_0)=0\leq M$ by assumption, $y_0$ also has the barycentric coordinates, such that $y_0=X(y)\omega(y_0)$. Subtracting the two equations and we have $\|q\|=\|X(y)[\omega(q)-\omega(y_0)]\|$, which immediately implies that $\sigma_k(X(y))\geq\epsilon^M/d$. Combined with the fact that $\sigma_k(X^r)\geq\underline{\sigma}^r$, we conclude that $\sigma_k(F(y))\geq\epsilon^M/(d\underline{\sigma}^r)=\underline{\sigma}^M$, which is our desired result.
\end{proof}
\begin{lemma}
\label{lem:ARAP}
$g=P_{F,\epsilon}(y,z)$ and $\mathcal{Z}=\prod_{i=1}^{d+1}\{n_i|\|n_i\|\leq1\}\times\text{SO}(3)$ satisfies~\prettyref{ass:bivariable}.
\end{lemma}
\begin{proof}
We start by showing~\prettyref{ass:bivariable} vi). First note that when optimizing $z=\FOUR{n_1}{\cdots}{n_{d+1}}{R}$, the $n_i$-subproblems and $R$-subproblem are separable. Suppose $\TWO{y}{z(y)}\in\mathcal{S}_\mathcal{Z}(r,M)$, then we have the analytic solution that $R=R(F)$ because $R$ solves the orthogonal Procrustes problem~\cite{sorkine2007rigid}. The uniqueness of $R$ follows from the uniqueness of polar decomposition. The uniqueness of $n_i$ follows from the $\sigma$-strong convexity.

Now suppose $\TWO{y}{z}\in\mathcal{S}_\mathcal{Z}(r,M)$, then we have $\|y\|\leq r$ and $\|z(y)\|\leq d+1+9$, so $\|\TWO{y}{z(y)}\|\leq r+d+10\triangleq\bar{r}_{zy}^{r,M}$, proving the second part of~\prettyref{ass:bivariable} vi). 

We still need to show that $z(y)=\TWO{n_i}{R(F)}$ is Lipschitz continuous when $\TWO{y}{z(y)}\in(r,M)$. We only need to show that each $n_i$ and $R(F)$ is Lipschitz continuous with respect to $y$ and sum up the Lipschitz modulus to yield $L_{zy}^{r,M}$. The Lipschitz continuity of $n_i$ follows by the same argument as in~\prettyref{lem:zy-case1}. From~\prettyref{lem:sigma-bound} ii), we know that $\sigma_k(F(y))\geq\underline{\sigma}^M>0$. It is known by~\cite{araki1981inequality} that $R(F)^TF$ is a Lipschitz continuous function of $F$ with a modulus of $\sqrt{2}$. As a result, we have the following inequality:
\begin{align*}
&\|R(F(y))^TF(y)-R(F(y'))^TF(y')\|_F\\
\leq&\sqrt{2}\|F(y)-F(y')\|_F.
\end{align*}
We then have the following inequality:
\begin{align*}
&\|R(F(y))^T-R(F(y'))^T\|_F\\
=&\|(R(F(y))^TF(y)-R(F(y'))^TF(y))F(y)^{-1}\|_F\\
=&\|(R(F(y))^TF(y)-R(F(y'))^TF(y')+\\
&R(F(y'))^TF(y')-R(F(y'))^TF(y))F(y)^{-1}\|_F\\
\leq&\|(R(F(y))^TF(y)-R(F(y'))^TF(y')\|_F\|F(y)^{-1}\|_F+\\
&\|R(F(y'))\|_F\|F(y')-F(y)\|_F\|F(y)^{-1}\|_F\\
\leq&(\sqrt{2}+\sqrt{3})\|F(y)^{-1}\|_F\|F(y')-F(y)\|_F\\
\leq&(\sqrt{2}+\sqrt{3})\sqrt{\sum_{k=1}^d\sigma_k(F(y))^{-2}}\|F(y')-F(y)\|_F\\
\leq&(\sqrt{2}+\sqrt{3})\sqrt{\sum_{k=1}^d(\underline{\sigma}^M)^{-2}}\|F(y')-F(y)\|_F,
\end{align*}
which shows that $R=R(F)$ is a Lipschitz-continuous function of $F$. We have thus proved the last part of~\prettyref{ass:bivariable} vi).

Finally, we turn to~\prettyref{ass:bivariable} i) and show that $P_{F,\epsilon}(y,z)$ is continuous in its entire domain, even though the indicator function terms $\iota_{\CH(y_1,\cdots,y_{d+1})}(y_0)$ and $\sum_{i=1}^3\iota_{\sigma_k(F)>0}(F(y))$ are not continuous in general. Let us suppose there is sequence $\TWO{y}{z}\to\TWO{y^\star}{z^\star}$. We first show that the term $\iota_{\CH(y_1,\cdots,y_{d+1})}(y_0)$ does not hinder continuity. In Case I, we assume $\iota_{\CH(y_1,\cdots,y_{d+1})}(y_0^\star)=\infty$ then since $\CH(y_1,\cdots,y_{d+1})$ is a closed set, there is a neighborhood of $\TWO{y^\star}{z^\star}$ in which $\iota_{\CH(y_1,\cdots,y_{d+1})}(y_0)=\infty$. In Case II, we assume $\iota_{\CH(y_1,\cdots,y_{d+1})}(y_0^\star)=0$, then there are two sub-cases. In Case II.a, we assume $y_0^\star\in\CH^\circ(y_1,\cdots,y_{d+1})$, then there is a neighborhood in which $\iota_{\CH(y_1,\cdots,y_{d+1})}(y_0)=0$. In Case II.b, we assume $y_0^\star\in\partial\CH(y_1,\cdots,y_{d+1})$, then by~\prettyref{lem:sigma-bound} i), we know that $P_{F,\epsilon}(y^\star,z^\star)=\infty$ because the plane separating term $-\sum_{0<j\neq i}\log_\epsilon((n_i^\star)^T(y_j^\star-y_0^\star))=\infty$ for some $i$. And since this term is continuous, we have $\lim_{y,z}P_{F,\epsilon}(y,z)=\infty$.

Next, we show that the term $\sum_{i=1}^3\iota_{\sigma_k(F)>0}(F(y))$ does not hinder continuity. Again, we consider two cases. In Case I, we assume $\sum_{i=1}^3\iota_{\sigma_k(F)>0}(F(y^\star))=0$, then since singular values are continuous functions of the matrix, we have a neighborhood in which $\sum_{i=1}^3\iota_{\sigma_k(F)>0}(F(y))=0$. In Case II, we assume $\sigma_k(F(y^\star))<0$ for some $k$, then there is a neighborhood in which $\sum_{i=1}^3\iota_{\sigma_k(F)>0}(F(y))=\infty$. In Case III, we have $\sigma_k(F(y^\star))=0$ but $\sigma_k(F(y))>0$ for every $\TWO{y}{z}$, then again we consider two sub-cases. In Case III.a, we assume $\iota_{\CH(y_1,\cdots,y_{d+1})}(y_0^\star)=\infty$, then by the argument in the previous section, we have $\lim_{y,z}P_{F,\epsilon}(y,z)=\infty$. In Case III.b, we assume $\iota_{\CH(y_1,\cdots,y_{d+1})}(y_0^\star)=0$, then as argued in the previous section, we can again consider two sub-cases. In Case III.b1, $y_0^\star\in\CH^\circ(y_1,\cdots,y_{d+1})$ and there is a neighborhood in which $\iota_{\CH(y_1,\cdots,y_{d+1})}(y_0)=0$. But this is not possible by~\prettyref{lem:sigma-bound} ii). In Case III.b2, $y_0^\star\in\partial\CH(y_1,\cdots,y_{d+1})$ then we again have $\lim_{y,z}P_{F,\epsilon}(y,z)=\infty$ as argued above. Thus all is proved.
\end{proof}
\fi

\end{document}